\documentclass[12pt]{amsart}
\usepackage[utf8]{inputenc}
\usepackage{fullpage}
\usepackage{amsthm}
\usepackage{framed}
\usepackage{url}
\usepackage[hidelinks]{hyperref}
\hypersetup{
	colorlinks   = true,
	citecolor    = red
}
\hypersetup{linkcolor=blue}
\usepackage{bbm}
\usepackage[capitalize,nameinlink]{cleveref}
\usepackage{setspace}

\usepackage{amsmath, amssymb}
\usepackage{enumitem}
\usepackage{cleveref}
\usepackage{tikz}
\usepackage{comment}
\usepackage{tikz-cd}
\usepackage{todonotes}
\usepackage[normalem]{ulem}
\usepackage{multirow}
\usepackage{thmtools}
\usepackage{thm-restate}
\theoremstyle{definition}

\newtheorem{question}{Question}

\newcommand\numberthis{\addtocounter{equation}{1}\tag{\theequation}}

\newcommand{\bbR}{\mathbb{R}}
\newcommand{\bbS}{\mathbb{S}}

\newcommand{\II}{\mathbb{I}}

\newcommand{\Gcalw}{\mathcal{G}_w}

\newcommand{\Mcal}{\mathcal{M}}

\newcommand{\dgh}{d_{\operatorname{GH}}}	

\newtheorem{cor}{Corollary}[section]
\newtheorem{claim}{Claim}

\newtheorem{lem}[cor]{Lemma}
\newtheorem{prop}[cor]{Proposition}
\newtheorem{thm}{Theorem}
\theoremstyle{remark}
\newtheorem{defn}[cor]{Definition}
\newtheorem{ex}[cor]{Example}
\newtheorem{rmk}[cor]{Remark}

\newcommand{\EE}{\mathbb{E}}

\newcommand{\dis}{\operatorname{dis}}	
	
\newcommand{\dispq}{\operatorname{dis}_{p,q}}
\newcommand{\disprimepq}{\operatorname{dis}_{p',q}}

\newcommand{\disprimepprimeq}{\operatorname{dis}_{p',q'}}
\newcommand{\disft}{\operatorname{dis}_{4,2}}

\newcommand{\Sp}{\mathbb{S}}

\newcommand{\Spe}{\mathbb{S}_E}
\newcommand{\Spg}{\mathbb{S}_G}
\newcommand{\Spn}{\mathbb{S}^n}

\newcommand{\diam}{\operatorname{diam}}
\newcommand{\diamp}{\operatorname{diam}_p}

\newcommand{\supp}{\mathrm{supp}}

\newcommand{\dwrr}{{d^{\mathbb{R}}_{{\operatorname{W}}r}}}
\newcommand{\dwrx}{{d^{X}_{{\operatorname{W}}r}}}

\newcommand{\dgwpq}{d_{{\operatorname{GW}}p,q}}

\newcommand{\dgwppqq}{d_{{\operatorname{GW}}p',q'}}
\newcommand{\dgwpone}{d_{{\operatorname{GW}}p,1}}

\newcommand{\dgwft}{d_{{\operatorname{GW}}4,2}}
\newcommand{\dgwp}{d_{{\operatorname{GW}}p}}

\newcommand{\dgwpinf}{d_{{\operatorname{GW}}p,\infty}}

\newcommand{\lbpq}{\operatorname{LB}_{p,q}}

\newcommand{\slbpq}{\operatorname{SLB}_{p,q}}
\newcommand{\slbft}{\operatorname{SLB}_{4,2}}
\newcommand{\tlbpq}{\operatorname{TLB}_{p,q}}

\newcommand{\dlbft}{\operatorname{DLB}_{4,2}}
\newcommand{\dlbpq}{\operatorname{DLB}_{p,q}}
\newcommand{\dlbpp}{\operatorname{DLB}_{p,p}}
\newcommand{\dlbpqmin}{\operatorname{DLB}_{p,p\wedge q}}

\newcommand{\gddx}{dH_{X}}

\newcommand{\lddx}{dh_{X}}

\begin{document}

\title{The Gromov-Wasserstein distance between spheres}

\author[S. Arya]{Shreya Arya}
\address{Shreya Arya,
    Department of Mathematics,
    Duke University
}
\email{shreya.arya@duke.edu}

\author[Arnab Auddy]{Arnab Auddy}
\address{Arnab Auddy,
    Department of Statistics,
    The Ohio State University
}
\email{auddy.1@osu.edu}

\author[R.A.Clark]{Ranthony A. Clark}
\address{Ranthony A.~Clark,
    Department of Mathematics,
    Duke University
}
\email{ranthony.clark@duke.edu}

\author[S. Lim]{Sunhyuk Lim}
\address{Sunhyuk Lim,
    Department of Mathematics,
    Sungkyunkwan University
}
\email{lsh3109@skku.edu}

\author[F. M{\'e}moli]{Facundo M{\'e}moli}
\address{Facundo M{\'e}moli,
    Department of Mathematics,
    The Ohio State University
}
\email{facundo.memoli@gmail.com}

\author[D. Packer]{Daniel Packer}
\address{Daniel Packer, 
    Department of Mathematics,
    The Ohio State University
}
\email{packer.61@osu.edu}


\begin{abstract} 
The Gromov-Wasserstein distance -- a generalization of the usual Wasserstein distance --  permits comparing probability measures defined on possibly different metric spaces. Recently, this notion of distance has found several applications in Data Science and in Machine Learning.  

With the goal of aiding  both the interpretability of dissimilarity measures computed through the Gromov-Wasserstein distance and in the assessment of the approximation quality of computational techniques designed to estimate the Gromov-Wasserstein distance,  in this paper, we determine the precise value of a certain variant of the Gromov-Wasserstein distance between unit spheres of different dimensions. 

Indeed, we consider a two-parameter family $\{\dgwpq\}_{p,q=1}^{\infty}$ of Gromov-Wasserstein distances between metric measure spaces.  By exploiting a suitable interaction between specific values of the parameters $p$ and $q$ and the metric of the underlying spaces, we determine the exact value of the distance $\dgwft$ between all pairs of unit spheres of different dimension endowed with their Euclidean distance and their uniform measure.
\end{abstract}

\maketitle

\tableofcontents

\section{Introduction}
Shape comparison ideas are utilized in a variety of fields with a wide range of application domains ranging from phylogenetics \cite{rohlf1990morphometrics, klingenberg1998geometric, collyer2015method}, medicine \cite{scher2007hippocampal}, neuroscience \cite{miller2004computational,wang2007large}, oral biology \cite{robinson2002impact}, language structure \cite{alvarez2018gromov}, social and biological networks \cite{chowdhury2019gromov, Hendrikson2016usingGD}, to political science \cite{fan2015spatiotemporal, kaufman2021measure} and computer vision \cite{leordeanu2005spectral, Rubner2000earth}. Many context specific tools have been developed to study the diverse set of problems which appear in these domains. Classical approaches such as statistical landmark analysis \cite{bookstein1977study} turn physical shapes into sequences of vectors, allowing for the rotation-dilation based approach of Procrustes Analysis (see \cite{dryden2016statistical,goodall1991procrustes,monteiro2002geometric}).  On the other hand, one can also understand a shape from the perspective of metric geometry, where the essence of a shape is captured by its pairwise interpoint distances \cite{memoli2005theoretical, memoli2007use}. Then in order to compare two shapes, i.e., in order to quantify their failure to be isometric, we compare their metric information directly. The Gromov-Hausdorff distance (see \cite{gromov1999metric,burago2022course}),  $\dgh$, provides a framework to compare distinct (compact) metric spaces $X$ and $Y$, where 
\[
\dgh(X,Y) :=\frac{1}{2} \inf_{R\in\mathcal{R}(X,Y)}\dis(R) 
\]
where $\mathcal{R}(X,Y)$ denotes the collection of all correspondences between $X$ and $Y$, that is, all subsets $R\subseteq X\times Y$ such that the canonical projections of $R$ onto the first and second coordinates satisfy $\pi_1(R)=X$ and $\pi_2(R)=Y$, and where 
$$\dis(R):=\sup_{(x,y),(x',y')\in R}\big|d_X(x,x')-d_Y(y,y')\big|$$ is the \emph{distortion} of $R$. The Gromov-Hausdorff distance has been considered  in the context of shape comparison and shape classification problems \cite{memoli2005theoretical}. However, it does not account for the distributional properties of a given data sample. The Gromov-Wasserstein distance \cite{memoli2011gromov} offers a robust alternative by viewing the shapes as metric measure spaces: triples $(X,d_X,\mu_X)$ where $d_X$ is the metric on $X$ and $\mu_X$ is a fully Borel supported probability measure on $X$.

It is natural to consider metric measure spaces in the context of shape and data comparison, since they allow us to associate to each point in our shape a weight that represents its relative importance within the dataset. The Gromov-Wasserstein distance provides a solution to the problem of finding the ``best'' way to align two shapes equipped with probability measures, where the best alignment is found by making use of the notion of coupling, a cognate of the notion of correspondence which is ubiquitous in the Kantorovich formulation of optimal transport \cite{villani2021topics}. Given measure spaces $(X, \mu_X), \ (Y, \mu_Y)$, a coupling between $X$ and $Y$ is a measure $\gamma$ on the product space $X \times Y$ whose marginals over $X$ and $Y$ are $\mu_X$ and $\mu_Y$ respectively.\footnote{More precisely, the pushforwards of $\gamma$ under $\pi_1$ and $\pi_2$ satisfy $(\pi_1)_\#\gamma = \mu_X$ and $(\pi_2)_\#\gamma = \mu_Y$.} 
We denote the space of all such measures by $\mathcal{M} (\mu_X, \mu_Y)$.
Intuitively, couplings align points in $X$ to those in $Y$. The distortion of a coupling provides insight to how well a given coupling interacts with the underlying metric structures of $X$ and $Y$ in order to preserve distances. For $p\in[1,\infty)$, the $p$-distortion induced by a coupling $\gamma \in \mathcal{M} (\mu_X, \mu_Y)$ is defined as:
\[
\mathrm{dis}_p (\gamma) := \left( \int_{X \times Y} \int_{X \times Y} \big| d_X(x, x') - d_Y(y, y') \big|^p \gamma (dx \times dy) \ \gamma (dx' \times dy') \right)^{1/p}.
\]
This distortion is then minimized (see \cite{memoli2011gromov}) over all possible couplings to define the $p$-Gromov-Wasserstein distance between $X$ and $Y$:
\[
\dgwp (X, Y) := \frac{1}{2} \inf_{\gamma \in \mathcal{M} (\mu_X, \mu_Y)} \mathrm{dis}_p (\gamma).
\]

In this work, we consider a two parameter family $\dgwpq$ (for $p,q\in [1,\infty]$) of Gromov-Wasserstein distances. In contrast to the $\dgwp$ distance above, we consider the $(p,q)$-distortion of a coupling $\gamma\in\mathcal{M}(\mu_X,\mu_Y)$ defined as:
\begin{equation}\label{intro: eqn}
\mathrm{dis}_{p, q} (\gamma) := \left( \int_{X \times Y} \int_{X \times Y} \big| d^q_X(x, x') - d^q_Y(y, y') \big|^{p / q} \gamma 
(dx \times dy) \ \gamma (dx' \times dy') \right)^{1/p},
\end{equation}
which is then minimized over all possible couplings to define the $(p,q)$-Gromov-Wasserstein distance between $X$ and $Y$:
$$
\dgwpq(X,Y):=
\frac{1}{2} \inf_{\gamma \in \mathcal{M} (\mu_X, \mu_Y)} \mathrm{dis}_{p,q} (\gamma).
$$
The $(p,q)$-Gromov-Wasserstein distance $\dgwpq$ interpolates between two previously studied versions of the Gromov-Wasserstein distance: for $q=1$, it reduces to the $\dgwp$ distance of \cite{memoli2011gromov}, while for $q=\infty$, it coincides with the $p$-ultrametric Gromov-Wasserstein distance defined in \cite{memoli2021ultrametric}.

This  formulation exhibiting one additional parameter  makes  Gromov-Wasserstein distances more amenable to analysis. Raising the distances to the $q$-th power allows for more explicit control of the difference in distances by emphasizing structural properties of $d_X$ and $d_Y$.  In this sense, of particular interest is the case of the Euclidean metric, with $q=2$, where taking squares of distances allows one to move from norms to inner products.  A construction related to the  case of $q=2$ and $p=4$ was considered in \cite{salmona2021gromov} to determine the value the Gromov-Wasserstein distances between arbitrary Gaussian distributions. A similar $(p,q)$-distortion was considered by Sturm in \cite{sturm-ss} with the distinction that the difference between the $q$-th powers of the distances is raised to the $p$-th power (as opposed to the $p/q$-th power). This implies that the $\dgwpq$ distance we consider has absolute homogeneity, while Sturm's version does not (cf. Remark \ref{rmk:relation}).

We now connect the $(p,q)$-Gromov-Wasserstein to some existing  computational approaches. Note that the computation of $\dgwpq(X,Y)$ involves optimizing the $(p,q)$-distortion over the set of all possible couplings $\gamma\in\mathcal{M}(\mu_X,\mu_Y)$. This reduces to a non-convex quadratic optimization problem \cite[Section 7]{memoli2011gromov}, which is in general NP-hard (see, e.g., \cite{peyre2016gromov, scetbon2022linear}). Nonetheless, there exist numerous computational approaches to find approximate solutions to the above problem or its variants: see \cite{peyre2016gromov, chowdhury2019gromov,chowdhury2020gromov,alvarez2018gromov,titouan2019optimal,titouan2019sliced,vayer2020fused} and references therein. Perhaps the most standard of these approaches is the use of gradient descent algorithms. In the absence of an algorithm that provably finds the global optima of this problem, practitioners often depend on heuristic initializations and find local optima through these gradient based methods. It is hence essential to assess the (sub)-optimality of these local optima. 

A particularly popular approach of assessing sub-optimality is to consider lower bounds of $\dgwp$ by using `signatures' or invariants of metric measure spaces (see \cite{memoli2011gromov,memoli2022distance}). We exhibit three invariant based lower bounds for the $(p,q)$-Gromov-Wasserstein distance between two arbitrary metric measure spaces. The three bounds are constructed from signatures related to particular invariants of and distributions on metric measure spaces. For example, the $(p,q)$-Second Lower Bound, denoted $\slbpq$, is a counterpart to the Second Lower Bound for $\dgwp$ from \cite{memoli2011gromov}, and utilizes the global distribution of distances between points of the two metric measure spaces. In  Proposition \ref{thm:hierarchy} we also establish a hierarchy of poly-time computable lower bounds for this distance in the spirit of \cite{memoli2011gromov,chowdhury2019gromov,memoli2022distance, memoli2021ultrametric}. 

As discussed above, these lower bounds aspire to be useful for determining whether the output of an algorithm is sufficiently close to the global optimum. Even though the exact values of the $\dgwpq$ Gromov Wasserstein distances may not always be  available, in practice one can compare the objective function at the computed (local) optima against our lower bounds to evaluate the performance of their algorithm. In the same spirit, lower bounds are useful in accelerating shape classification, where knowing the relative strengths of the lower bounds allows one to progressively filter out comparisons of the most distinct examples by comparing examples with successively stronger lower bounds from the hierarchy (see \cite{memoli2011gromov}).

We further derive lower and upper bounds on the $\dgwpq(X,Y)$ distance in the case where $p=4$, $q=2$ and the metric measure spaces $X,Y$ are spheres equipped with geodesic or Euclidean distances, and uniform measures. Spheres, being canonical spaces, are a natural starting point for understanding Gromov-Wasserstein distances. These can provide reasonable benchmarks for assessing the quality of a given algorithm for estimating Gromov-Wasserstein and related distances (see, e.g., \cite{lim2023gromov}). 

\begin{figure}
\includegraphics[width = \linewidth]{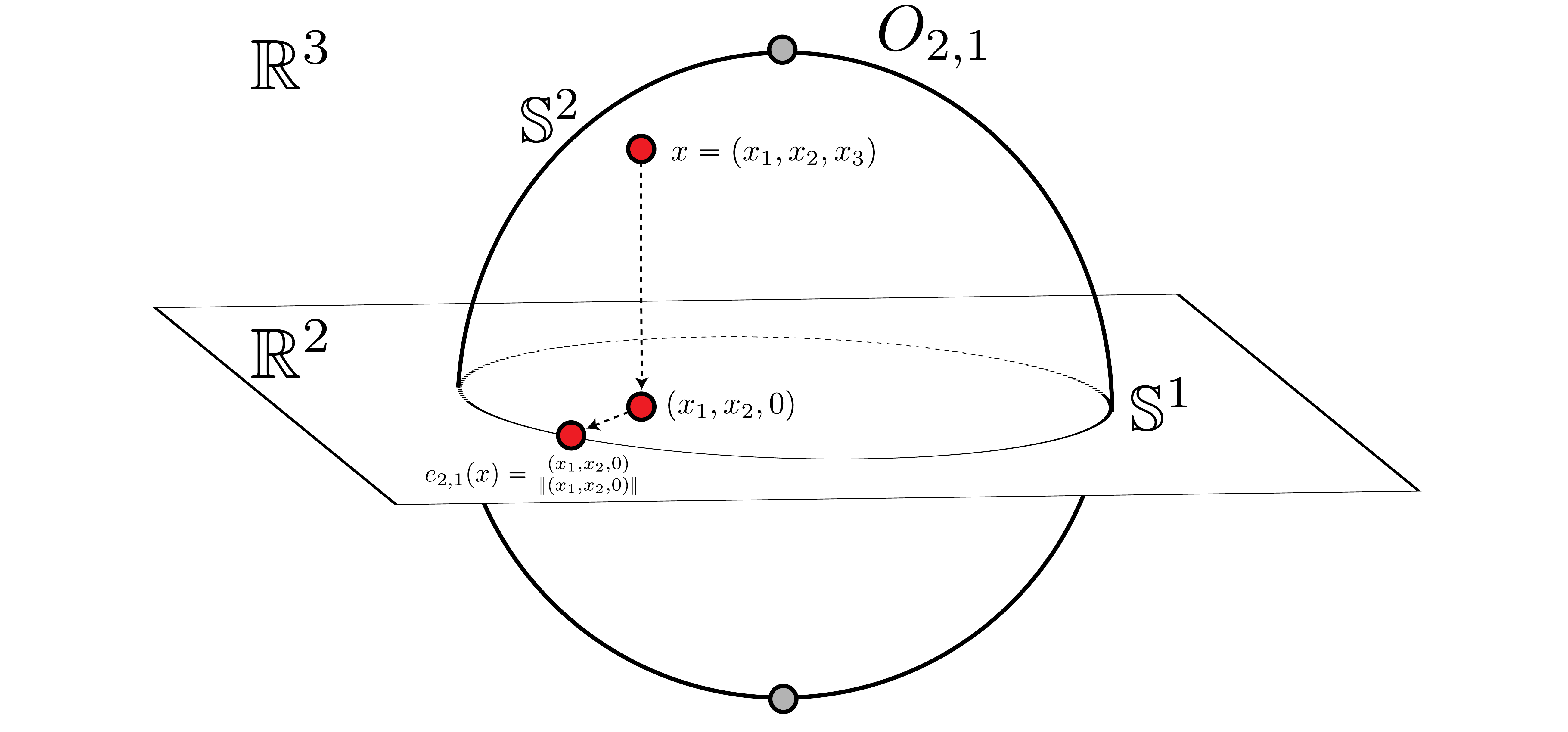}
\caption{A depiction of the equatorial map $e_{2,1}:\Sp^2\backslash O_{2,1}\to \Sp^1$. The measure zero set $O_{2,1}$ consists of the north and south poles of $\Sp^2$. See Definition \ref{def:eq-map}.} \label{fig:prj}
\end{figure}

Furthermore, with the goal of providing such benchmarks, in  Theorem \ref{thm:equ-opti} we determine the \emph{exact} value of $\dgwft(\Spe^m,\Spe^n)$, i.e., the (4,2)-Gromov-Wasserstein distance between $m$ and $n$ dimensional unit spheres equipped with the Euclidean metric and uniform measures. The optimal coupling is induced by a map $\gamma_{m,n}$ that projects the first $m$ coordinates from the $n$-dimensional sphere $(m<n)$ to the $m$-dimensional unit sphere. Inspired by the case of $m=1$ and $n=2$, we call this the \emph{equatorial} map; see Definition \ref{def:eq-map} and Figure \ref{fig:prj}.

\begin{restatable}[Main Theorem]{thm}{mainthm}
\label{thm:equ-opti}
The equatorial coupling $\gamma_{m,n}$ is an optimal coupling for $\dgwft(\Spe^m,\Spe^n)$. In particular, for $n\ge m\ge 0$, 
$$
\dgwft(\Spe^m,\Spe^n)
=
\dfrac{1}{2}
\dis_{4,2}(\gamma_{m,n},\Spe^m,\Spe^n)
=
\dfrac{1}{\sqrt{2}}
\left[    
    \dfrac{1}{m+1}+\dfrac{1}{n+1}
-\dfrac{2}{m+1}
\left(    
\dfrac{\Gamma\left(\tfrac{m+2}{2}\right)
\Gamma\left(\tfrac{n+1}{2}\right)}{
\Gamma\left(\tfrac{m+1}{2}\right)
\Gamma\left(\tfrac{n+2}{2}\right)
}\right)^2
\right]^{1/4}.
$$
\end{restatable}

Note that when $n=m+2$ the expression given in the theorem reduces to 
\[
\dgwft(\Sp^{m}_E,\Sp^{m+2}_E)
=
\dfrac{1}{2^{1/4}}
\left[
\dfrac{1}{(m+1)(m+3)}
+
\dfrac{1}{(m+2)^2(m+3)}
\right]^{\frac{1}{4}}.
\]
In particular, this means that $\dgwft(\Sp^{m}_E,\Sp^{m+2}_E)=O(m^{-1/2})$ as $m\to \infty$. These asymptotics do not depend on the fact that $n=m+2$. Indeed, as explained in Remark \ref{rem:asympt-n-m}, $\dgwft(\Sp^{m}_E,\Sp^{n}_E)=O(m^{-1/2})$ when $n=m+k_0$ for a fixed $k_0$. A different interesting regime  (see Remark~\ref{rem:asympt-fixed-m}) is the one when $n=m+k$ for fixed $m$ but $k\to\infty$; in that case we have
$$
\dgwft(\Spe^m,\Spe^{m+k})
=
\dfrac{1}{\sqrt{2}}\left[\dfrac{1}{m+1}\right]^{1/4}
+O(k^{-1/4})
$$
as $k\to\infty$.  See Section \ref{subsec:exact} for additional related results.

\subsubsection*{Ideas behind the proof of Theorem \ref{thm:equ-opti} and related work.} \label{sec:rel-wk} The proof proceeds by first observing that minimizing the $(4,2)$-distortion over all couplings $\gamma$ between the respective uniform measures on $\Spe^m$ and $\Spe^n$ is equivalent to \emph{maximizing} the functional $$J(\gamma):= \Bigg\|
\underbrace{
\int_{\Sp^m \times \Sp^n} 
xy^\top
\ \gamma(dx \times dy)}_{M_\gamma}
\Bigg\|_{\rm F}^2.$$ Then, via  the singular value decomposition of the matrix $M_\gamma$, we
identify a change of coordinates (see Section \ref{sec:change-of-coords}) which allows us to argue that one can restrict attention to couplings $\tilde{\gamma}$ for which the matrix $M_{\tilde{\gamma}}$ has the form $\begin{pmatrix}
\Lambda_\gamma &
\mathbf{0}_{(m+1)\times (n-m)}
\end{pmatrix}$ where $\Lambda_\gamma$ is a $m\times m$ diagonal matrix containing the singular values of $M_\gamma$ --- a simplification which in turn implies that one can equivalently restrict attention to maximizing the (simpler) functional $$D(\gamma) := 
\sum_{k=1}^{m+1}
\left[\int_{\Sp^m \times \Sp^n}
x_k y_k \
\gamma (dx \times dy) \right]^2.$$ The final step of the proof can be described as an intricate application of the Cauchy-Schwarz inequality to a suitably disintegrated expression of the functional $D$. 

To better describe some of the main ideas behind the proof of Theorem \ref{thm:equ-opti}  without these intricacies, we state and prove a subsidiary result (Proposition \ref{prop:gaussian}, which is particular case of \cite[Proposition 4.1]{salmona2021gromov}) about maximizing the functional $J$ in the case when $\gamma$ ranges over all couplings between standard Gaussian measures. We then show how to suitable alter the proof of Proposition \ref{prop:gaussian} in order to establish Theorem \ref{thm:equ-opti}.

 \label{page:change-coords}
The functional $J(\gamma)$ appears naturally when considering certain inner product based variants of the Gromov-Wasserstein distance \cite{vayer2018optimal,salmona2021gromov,dumont2022existence}. The change of coordinates step appeared in Vayer's PhD thesis \cite[Chapter 4]{vayer2020contribution}. It was also used in \cite[Theorem 3.2]{dumont2022existence} when studying the existence of Monge maps for the $J$ functional defined above.

There are results in the literature providing a precise description of optimal Monge maps in the context of Gromov-Wasserstein distances but none of those seems to be applicable in our setting:

\begin{itemize}
\item In \cite[Section 9.4]{sturm-ss} Sturm provides a characterization result for optimal couplings under the assumption that both measures are rotationally invariant and absolutely continuous w.r.t. Lebesgue measure \emph{in the same ambient Euclidean space}. The setting of Theorem \ref{thm:equ-opti} does not fit into the one considered by Sturm: the measures considered therein are not absolutely continuous. One might nevertheless contemplate applying Sturm's result to suitable smoothings of the uniform measures on the respective spheres. However, it does not seem possible to guarantee the smoothed measure resulting from the lower dimensional sphere to be rotationally invariant.
\item In \cite[Theorem 3.2]{dumont2022existence} Dumont et al. give a precise description of optimal Monge maps for the $J$ functional under the assumption that one of the measures is absolutely continuous w.r.t. Lebesgue measure.\footnote{Their results generalize \cite[Theorem 4.2.3]{vayer2020contribution}.} Therefore their results are also not applicable to our setting.
\item  In \cite[Proposition 4.1]{salmona2021gromov} Salmona et al. contend with the case of Gaussian measures and find the precise structure of an optimal couplings that, in the case of standard Gaussians, boils down to the  a coupling between Gaussians with similar structure to the equatorial coupling. As far as we know, it does not seem possible to apply their results in our setting (i.e. in order to establish Theorem \ref{thm:equ-opti}); see Question \ref{q:gauss} and also Section \ref{sec:gaussian}. 
\end{itemize}

\subsubsection*{Other related work}
Finally, this project is related to a recent effort to compute the precise value of the (closely related) Gromov-Hausdorff distance between spheres \cite{lim2023gromov,adams2022gromov}. In \cite{lim2023gromov}, the authors provide nontrivial upper and lower bounds for the Gromov-Hausdorff distance $\dgh(\Sp^m,\Sp^n)$ between spheres $\Sp^m$ and $\Sp^n$ (endowed with the geodesic metric) for $0 \leq m < n \leq \infty$. Some of these lower bounds were motivated by topological ideas related to a quantitative version of the Borsuk-Ulam theorem \cite{dubins1981equidiscontinuity}. Through explicit constructions of (optimal) correspondences it was proved that their lower bounds were tight in the cases of $\dgh(\Sp^0,\Sp^n)$, $\dgh(\Sp^m,\Sp^\infty)$, $\dgh(\Sp^1,\Sp^2)$, $\dgh(\Sp^1,\Sp^3)$, and $\dgh(\Sp^2,\Sp^3)$. Interestingly, the optimal correspondences achieving these distances are very different in nature from the optimal coupling achieving the exact value of $\dgwft(\Sp^m_E,\Sp^n_E)$, in the sense that these optimal correspondence are induced by highly discontinuous maps whereas the optimal coupling described above is induced by certain natural projection maps (which are continuous).

\subsection*{Acknowledgements.} This project was part of the 2022 AMS-MRC (Math Research Communities) on ``Data Science at the Crossroads of Analysis, Geometry, and Topology". S.A. and A.A. acknowledge support from NSF through grant DMS-1916439. R.A.C.E. acknowledges support from NSF MPS Ascending Fellowship Award \#2138110. F.M. was supported by the NSF through grants 
DMS-1723003, IIS-1901360, CCF-1740761, DMS-1547357, CCF-2310412 and by the BSF under grant  2020124.

\subsection*{Organization of the Paper.}
The rest of the paper is organized as follows. 

Section~\ref{sec:notation} introduces notation and terminology that will be used throughout the paper. 

Section \ref{sec:sphr-lb} is the central section of our paper. There we introduce the requisite background and supporting results used to prove Theorem \ref{thm:equ-opti}, which we also do therein. 

In Section \ref{sec:gen-lb} we recall several invariants of metric measure spaces and use them to prove lower bounds for the $(p,q)$-Gromov-Wasserstein distance.  In Section \ref{sec:lbs-spheres} we evaluate those lower bounds for spheres with their Euclidean and geodesic distances. For the former case we compare these lower bounds against the exact value provided by Theorem \ref{thm:equ-opti}.

Section \ref{sec:exp} contains a description of some experiments illustrating the result from Theorem \ref{thm:equ-opti}. In particular, our experiments provide an computational perspective and an indication of
the performance of discrete Gromov-Wasserstein solvers from \cite{flamary2021pot} and \cite{cuturi2022optimal} when estimating the Gromov-Wasserstein distance between spheres.

Section \ref{sec:disc} provides a discussion and contains several questions (Questions \ref{q:comp}, \ref{q:eq-map-opt}, \ref{q:gauss}, and \ref{q:monge}) that might suggest further research directions.

To enhance readability, the proofs of several results and other details are relegated to Appendices ~\ref{sec:proofs}, \ref{app:calc}, and \ref{app:exps}. 

\subsection{Notation and Terminology}\label{sec:notation}

We now define the main concepts used in the paper. 

Given a measurable space $(X,\Sigma_X)$, we denote the set of all probability measures on $X$ by $\mathcal{P}(X)$.

\begin{defn}
Let $(X,\Sigma_X,\mu_X)$ and $(Y,\Sigma_Y,\mu_Y)$ be measure spaces such that $\mu_X\in\mathcal{P}(X)$ and $\mu_Y\in\mathcal{P}(Y)$. A \emph{coupling} between $\mu_X$ and $\mu_Y$ is a (probability) measure $\gamma$ on the product space $X \times Y$ such that 
\[
\gamma(A \times Y) = \mu_X(A) \hspace{5mm} \text{and} \hspace{5mm} \gamma(X \times B) = \mu_Y(B)
\]
for all $A \in \Sigma_X$ and $B \in \Sigma_Y$.  We denote the set of all couplings between $\mu_X$ and $\mu_Y$ by $\mathcal{M}(\mu_X,\mu_Y)$. Note that $\mathcal{M}(\mu_X,\mu_Y)$ is never empty as it  always contains the product measure $\mu_X \otimes \mu_Y$.
\end{defn}

Let $(X,\Sigma_X, \mu_X)$ and $(Y,\Sigma_Y,\mu_Y)$ be measure spaces such that $Y=\{y_0\}$, $\mu_X\in \mathcal{P}(X)$ and $\mu_Y\in\mathcal{P}(Y)$. Then, $\mu_Y=\delta_{y_0}^Y$ and $\mathcal{M}(\mu_X,\mu_Y)$ contains exactly one coupling. That is, $\mathcal{M}(\mu_X,\mu_Y)=\{\mu_X \otimes \delta_{y_0}^Y\}$ where $\mu_Y=\delta_{y_0}^Y$ is a Dirac delta.  

Let $(X,\Sigma_X)$ and $(Y,\Sigma_Y)$ be measurable spaces and $T:X \to Y$ a measurable map. The \emph{pushforward} of a measure $\alpha$ on $(X,\Sigma_X)$ by $T$, denoted $T_\#\alpha$, is the measure on $Y$ given by
\[
T_\#\alpha(A) = \alpha(T^{-1}(A))
\]
for every $A \in \Sigma_Y$. We can then describe the set of all couplings between $\mu_X$ and $\mu_Y$ as 
\[
\mathcal{M}(\mu_X,\mu_Y)=\{ \gamma \in \mathcal{P}(X\times Y)\mid (\pi_X)_\#\gamma=\mu_X, (\pi_Y)_\#\gamma=\mu_Y \}
\]
where $\pi_X:X\times Y \to X$ and $\pi_Y:X \times Y \to Y$ are the canonical projections onto the first and second components respectively.

\begin{framed}Given a topological space $X$, unless indicated otherwise, we will assume all measures on $X$ to be Borel measures, and  will denote the Borel sigma algebra of $X$ by $\Sigma_X$. Furthermore, in this case, $\mathcal{P}(X)$ will denote the set of all Borel probability measures on $X$.
\end{framed}
\begin{defn} The \emph{support} of a Borel measure $\alpha$ on a topological space $X$  is the smallest closed subset $X_0 \subset X$ so that $\alpha(X \backslash X_0)=0,$ that is, for any $A \subset X$, if $A \cap X_0 = \emptyset$, then $\alpha(A)=0.$ We denote the support of $\alpha$ by $\mathrm{supp}[\alpha].$
\end{defn}

\begin{defn}
    Let $\alpha$ be a Borel probability measure on $\mathbb{R}$ and $r\in [1,\infty)$. Then the \emph{$r$-moment} of $\alpha$ is
\[
m_r(\alpha):= \left(\int_\mathbb{R} x^r \alpha(dx)\right)^{1/r}.
\] 
Now let $(X, d_X)$ be a metric space.
We define,
\[ \mathcal{P}_r (X) := \{ \mu \in \mathcal{P}(X) \ | \ d_X(x_0, \cdot)_\# \mu \text{ has finite $r$-moment for some $x_0 \in X$} \}.\]
In fact, the choice of $x_0$ is immaterial -- if the moment is finite for one reference point, it is finite for any reference point.
\end{defn}

\begin{defn} Let $(X,d_X)$ be a  metric space, $r\in[1,\infty]$, and $\alpha, \beta \in \mathcal{P}_r(X)$.\footnote{We define  $\mathcal{P}_\infty(X)$ as the set of those probability measures on $X$ with bounded support.} 
The \emph{$r$-Wasserstein}  distance on between $\alpha$ and $\beta$ is given by
\[
\dwrx(\alpha,\beta):= \inf_{\gamma \in \mathcal{M}(\alpha, \beta)} \bigg( \int_{X \times X} d_X^r(x,x') \ \gamma(dx \times dx') \bigg)^{1/r}
\]

\noindent for $1 \leq r < \infty$, and

\[
d^X_{\operatorname{W}\infty}(\alpha,\beta):= \inf_{\gamma \in \mathcal{M}(\alpha, \beta)} \sup_{(x,x') \in \supp[\gamma]} d_X(x,x').
\]
\end{defn}

\begin{ex}\label{ex:dWonR}
When $\alpha,\beta \in \mathcal{P}_r(\bbR)$, the $r$-Wasserstein distance on $\mathbb{R}$ (with the usual metric) can be explicitly computed as follows  (see, e.g., \cite{villani2021topics}):
\begin{equation}\label{eq:dw-R}
    \dwrr(\alpha,\beta)
=
\left(
\int_0^1 
|F_{\alpha}^{-1}(u)-F_{\beta}^{-1}(u)|^r \ du
\right)^{1/r}
\end{equation}
where $F_{\alpha}(t):= \alpha((-\infty,t])$ and $F_{\beta}(t):= \beta((-\infty,t])$ are the cumulative distributions of $\alpha$ and $\beta$, respectively, and their generalized inverses are defined as:
\begin{equation}\label{eq:def-ginv}
F_{\alpha}^{-1}(u) := \inf\{ t \in \mathbb{R} \mid F_{\alpha}(t) > u\}
\end{equation}
for $u\in[0,1]$.
\end{ex}

\begin{ex} \label{ex:wass-delta} 
For $r \in [1,\infty)$, the $r$-Wasserstein distance on the real line  between $\alpha\in \mathcal{P}_r(\bbR)$ and the Dirac delta $\delta_0$  equals the $r$-moment of $\alpha$: 
\[
\dwrr(\alpha,\delta_0) = \bigg(\int_{\bbR \times \bbR} |t-s|^r (\alpha \otimes \delta_0)(dt \times ds) \bigg)^{1/r} = \bigg(\int_{\bbR_+} t^r \alpha(dt)\bigg)^{1/r} = m_r(\alpha).
\]
\end{ex}

\begin{defn}
For each $q\in [1,\infty]$, we define $\Lambda_q:\bbR_+ \times \bbR_+ \to \bbR_+$ in the following way (cf. \cite{memoli2022metric}): 
\begin{align*}
    &\Lambda_q(a,b):=|a^q-b^q|^{\frac{1}{q}}\text{ if }q<\infty,\text{ and}\\
    &\Lambda_\infty(a,b):=\begin{cases}\max\{a,b\}&\text{if }a\neq b\\0&\text{if }a=b.\end{cases}
\end{align*}
\end{defn}

\begin{rmk}
Note that $\Lambda_1(a,b) = |a-b|$ for all $a,b\geq 0$. One of the the claims of the following proposition is that $\Lambda_q$ is a metric on $\bbR_+$ for each $q\geq 1$. 
\end{rmk}

\begin{prop}[{\cite[Lemma 2.2, Example 2.7, Proposition 2.11]{memoli2022metric}}]\label{prop:lmbdaq}
$\Lambda_q$ defines a metric on $\bbR_+$ for each $q\in[1,\infty]$, i.e., it is symmetric, non-negative, it satisfies $\Lambda_q(a,b)=0$ if and only if $a=b$, and it satisfies the triangle inequality:
$$\Lambda_q(a,b)\leq \Lambda_q(a,c)+\Lambda_q(c,b)\,\,\mbox{for all $a,b,c\geq 0$}.$$ 
Also, if $1\leq q\leq q'\leq\infty$, then $\Lambda_q\leq\Lambda_{q'}$.
\end{prop}

The fact that $(\bbR_+,\Lambda_q)$ is a metric space, enables us to consider the $p$-Wasserstein distance  $d^{(\bbR_+,\Lambda_q)}_{\operatorname{W}p}$ as a generalization of $d^{(\bbR_+,\Lambda_1)}_{\operatorname{W}p}=d^{\bbR_+}_{\operatorname{W}p}$.

\begin{rmk}[Closed-form solution for $d^{(\bbR_+,\Lambda_q)}_{\operatorname{W}p}$]\label{rmk:clfrmdWR+Lbdaq}
For $1\leq q\leq p<\infty$, we have the following equality which generalizes Example \ref{ex:dWonR}:
\[
d^{(\bbR_+,\Lambda_q)}_{\operatorname{W}p}(\alpha,\beta)
=
\left(
\int_0^1 
\big(\Lambda_q(F_{\alpha}^{-1}(u),F_{\beta}^{-1}(u))\big)^p \ du
\right)^{1/p}.
\]
While the above equality is a special case of \cite[Theorem A.4]{memoli2021ultrametric}, we include a proof here for pedagogical reasons: 
\begin{align*}
    \left(d^{(\bbR_+,\Lambda_q)}_{\operatorname{W}p}(\alpha,\beta)\right)^p&=\inf_{\gamma \in \mathcal{M}(\alpha, \beta)} \int_{\bbR_+\times\bbR_+} \left(\Lambda_q(a,b)\right)^p \ \gamma(da \times db)\\
    &=\inf_{\gamma \in \mathcal{M}(\alpha, \beta)} \int_{\bbR_+\times\bbR_+} \vert S_q(a)-S_q(b)\vert^{p/q} \ \gamma(da \times db)\\
    &=\inf_{\gamma \in \mathcal{M}(\alpha, \beta)} \int_{\bbR_+\times\bbR_+} \vert s-t\vert^{p/q} \ \big((S_q, S_q)_\#\gamma\big)(ds \times dt)\\
    &=\inf_{\gamma \in \mathcal{M}((S_q)_\#\alpha,(S_q)_\#\beta)} \int_{\bbR_+\times\bbR_+} \vert s-t\vert^{p/q} \ \gamma(ds \times dt)\\
    &=\int_0^1 \Big|F_{(S_q)_\#\alpha}^{-1}(u)-F_{(S_q)_\#\beta}^{-1}(u)\Big|^{p/q} \ du\\
    &=\int_0^1 
\big(\Lambda_q(F_{\alpha}^{-1}(u),F_{\beta}^{-1}(u))\big)^p \ du
\end{align*}
where $S_q:\bbR_+\rightarrow\bbR_+$ is the map sending $x$ to $x^q$. The fourth equality holds by \cite[Lemma 3.2]{chowdhury2019gromov}, fifth equality holds by Example \eqref{ex:dWonR}, and the last equality holds since $F_{(S_q)_\#\alpha}^{-1}=(F_{\alpha}^{-1})^q$.
\end{rmk}

The following remark makes the connection between the generalized Wasserstein distance $d^{(\bbR_+,\Lambda_q)}_{\operatorname{W}p}$ and the Wasserstein distance on $\bbR_+$ with the usual metric.

\begin{rmk}(Relationship between $ d^{(\bbR_+,\Lambda_q)}_{\operatorname{W}p}$ and $d_{\operatorname{W}p/q}^\mathbb{R}$)\label{rmk:old-new wasserstein}\label{rmk:Wpq-closed-form}  In the previous remark, we saw that when $1\leq q\leq p<\infty$,
\begin{align*}
    d^{(\bbR_+,\Lambda_q)}_{\operatorname{W}p}(\alpha,\beta)  =  \left(\int_0^1 
\big(\Lambda_q(F_{\alpha}^{-1}(u),F_{\beta}^{-1}(u))\big)^p \ du\right)^{1/p}=\left(\int_0^1 \big|F_{(S_q)_\#\alpha}^{-1}(u)-F_{(S_q)_\#\beta}^{-1}(u)\big|^{p/q} \ du \right)^{1/p}.
\end{align*}
The right hand side of the above expression can be interpreted as the $p/q$- Wasserstein distance between the measures $(S_q)_\#\alpha$ and $(S_q)_\#\beta$  on $\bbR_+$ with the usual metric ($\Lambda_1$) as follows: 
$$ d^{(\bbR_+,\Lambda_q)}_{\operatorname{W}p}(\alpha,\beta) =  \left(d_{\operatorname{W}p/q}^\bbR((S_q)_\#\alpha,(S_q)_\#\beta)\right)^{1/q}. $$
\end{rmk}

\begin{defn}
A \emph{metric measure space} is a triple $(X,d_X,\mu_X)$ where $(X,d_X)$ is a compact metric space, and $\mu_X$ is a Borel probability measure on $X$ such that $\mathrm{supp}[\mu_X]=X$. We denote the collection of all metric measure spaces by $\mathcal{G}_w.$ We will often abuse notation and write $X$ to represent the triple $(X,d_X,\mu_X) \in \mathcal{G}_w.$
\end{defn}

The next example is central to our paper.
\begin{ex}[$\Sp^n_E$ and $\Sp^n_G$] \label{ex:spheres-def} For each integer $n\geq 1$, we  consider the $n$-dimensional unit sphere  $\Sp^n\subset \mathbb{R}^{n+1}$ as a metric measure space by equipping it with the uniform measure and the geodesic or Euclidean metric. For example, when endowed with its geodesic distance, the usual $n$-dimensional unit sphere gives rise to $(\Sp^n,d_n,\mu_n) \in \Gcalw$, where $d_n(x,x') :=\arccos(\langle x,x'\rangle) $ for $x,x'\in\Sp^n$. 
 We henceforth write $\Sp^n_E$ and $\Sp^n_G$  to denote the spheres equipped with the Euclidean and geodesic metrics, respectively, as metric measure spaces.
 
We also consider $\Sp_G^0$, the $0$-dimensional sphere consisting of two points at distance $\pi$ and, similarly, $\Sp^0_E$ consists of two points at distance $2$. In both cases we view these 0-dimensional spheres as mm-spaces by endowing them with the uniform measure (on two points). Note that $\diam(\Sp^n_G) = \pi$ and $\diam(\Sp^n_E)=2$ for all integers $n\geq 0$.
\end{ex}

\begin{defn}[$p$-diameter]\label{def:diamp}
The \emph{diameter} $\diam(A)$ of bounded  subset $A$ of a metric space $(X,d_X)$ is defined as $$\diam(A) := \sup_{x,x'\in A}d_X(x,x').$$

 Let $(X,d_X,\mu_X) \in \mathcal{G}_w$. The \emph{$p$-diameter} of $X$ for $p\in [1,\infty]$ is:
\[
\diamp(X):= \left(\int_X \int_X d^p_X(x,x') \ \mu_X(dx) \ \mu_X(dx')\right)^{1/p} \,\,\mbox{for $1\le p<\infty$}
\] and 
\[
\diam_{\infty}(X):= \diam(\supp[\mu_X]).
\]
\end{defn}

\begin{defn}[$(p,q)$-distortion] Let $(X,d_X,\mu_X)$ and $(Y,d_Y,\mu_Y)$ be metric measure spaces and let $\gamma \in \mathcal{M}(\mu_X,  \mu_Y)$. Then, for each $p,q\in [1,\infty]$, the \emph{$(p,q)$-distortion} of the coupling $\gamma$ is defined as:
\begin{equation*}\label{eq:def-distort}
     \dispq(\gamma) := \bigg(  \int_{X\times Y}\int_{X\times Y} \big(\Lambda_q(d_X(x,x'),d_Y(y,y'))\big)^p \ \gamma(dx \times dy) \ \gamma( dx' \times dy')\bigg)^{1/p} 
\end{equation*}

\noindent for $1\le p<\infty$, and 

\[
\dis_{\infty,q}(\gamma) := \sup_{(x,y),(x',y')\in\supp[\gamma]} \Lambda_q(d_X(x,x'),d_Y(y,y')).
\]
\end{defn}

\begin{ex}\label{ex:exdisonepoint} Consider the coupling $\gamma= \{\mu_X \otimes \delta_{y_0}^Y\} \in \Mcal(\mu_X,\mu_Y)$ where $Y=\{y_0\}$ is the one point metric measure space. Then, for all $p,q\in [1,\infty]$, one can verify that $\dispq(\gamma)=\mathrm{diam}_p(X)$ as follows:

\begin{align*} 
     \dispq(\{\mu_X \otimes \delta_{y_0}^Y\}) 
     = & \ \bigg( \int_{X\times Y}\int_{X\times Y} \big(\Lambda_q(d_X(x,x'),d_Y(y_0,y_0))\big)^p \ \gamma(dx \times dy) \ \gamma( dx' \times dy')\bigg)^{1/p} \\
     = & \ \bigg( \int_{X\times Y}\int_{X\times Y} | d_X^q(x,x')|^{p/q} \ \gamma(dx \times dy) \ \gamma( dx' \times dy')\bigg)^{1/p}\\
     = & \ \bigg( \int_{X \times X} d_X^p(x,x') \ \mu_X(dx) \ \mu_X( dx')\bigg)^{1/p} \\
     = & \ \mathrm{diam}_p(X)
\end{align*}
for the $p<\infty$ case, and the $p=\infty$ case can be checked in a similar way.
\end{ex}

\begin{ex}\label{ex:dis-42} Let $p=4$ and $q=2$ and $\gamma \in \Mcal(\mu_X,\mu_Y)$, then we have:
\begin{align*} 
(\disft(\gamma))^4 & =
\left( \int_{X\times Y} \int_{X\times Y} \left( d^2_X(x, x') - d_Y^2(y,y') \right)^{2} \ \gamma(dx \times dy) \ \gamma(dx'\times dy') \right) \\
& = \int_{X \times X} d^4_X(x,x')\ \mu_X(dx)\,\mu_X(dx') + \int_{Y \times Y} d^4_Y(y,y)\ \mu_Y(dy)\,\mu_Y(dy')\\
& - 2\int_{X\times Y} \int_{X\times Y} d^2_X(x,x') \, d^2_Y(y,y') \ \gamma(dx\times dy) \ \gamma(dx'\times dy'). 
\end{align*}
\end{ex}

\begin{rmk}\label{rmk:max-dis42}
Note that the marginals $\mu_X$ and $\mu_Y$ determine the first two terms in Example \ref{ex:dis-42} (in fact the sum of the first two terms is $(\diam_4(X))^4+(\diam_4(Y))^4$) and thus,
\[
\mbox{$\gamma$ minimizes $\disft(\gamma)$ $\Leftrightarrow$ $\gamma$ maximizes }
\int_{X\times Y} \int_{X\times Y} d^2_X(x,x') \, d^2_Y(y,y') \ \gamma(dx\times dy) \ \gamma(dx'\times dy').
\]
The equivalence above will prove instrumental in Section~\ref{sec:sphr-lb} of our paper where we prove the optimality of a coupling for achieving the Gromov-Wasserstein distance between spheres.
Passing to the distance squared allows us to unfold Euclidean distances into expressions that depend solely on inner products.
Our proof of the theorem depends on the favorable interplay between these inner products and linear maps.
In fact, our introduction of the broader family of $(p,q)$-Gromov-Wasserstein distances was motivated by this ease of analysis in the case $p = 4,\ q = 2$.
\end{rmk}

\begin{defn}[$(p,q)$-Gromov-Wasserstein distance] Let $(X,d_X,\mu_X)$ and $(Y,d_Y,\mu_Y)$ be metric measure spaces. Let $\gamma \in \mathcal{M}(\mu_X, \mu_Y)$. The \emph{$(p,q)$-Gromov-Wasserstein distance} between $\mu_X$ and $\mu_Y$ is given by one-half of the infimum of the $(p,q)$-distortion:
    \begin{equation*}\label{eq:def-gwpq}
        \dgwpq(X,Y) := 
        \dfrac{1}{2}
        \inf_{\gamma \in \Mcal(\mu_X, \mu_Y)}  \dispq(\gamma). 
    \end{equation*}
\noindent where $p,q\in [1,\infty]$.
\end{defn}

\begin{rmk}
In the case $q=1$ we recover the $p$-Gromov-Wasserstein distance $\dgwp$ from \cite{memoli2011gromov}. 
\begin{equation*}\label{eq:def-gwpq1}
\dgwpone(X,Y) = 
\dfrac{1}{2}
\inf_{\gamma \in \Mcal(\mu_X, \mu_Y)} 
\dis_{p,1}(\gamma) = \dgwp(X,Y).
  \end{equation*}
In the case $q=\infty$ we recover the $p$-ultrametric Gromov-Wasserstein distance $u_{\mathrm{GW},p}$ from \cite{memoli2021ultrametric}. 
\begin{equation*}\label{eq:def-gwpqinfty}
\dgwpinf(X,Y) = 
\dfrac{1}{2}
\inf_{\gamma \in \Mcal(\mu_X, \mu_Y)} 
\dis_{p,\infty}(\gamma) =\dfrac{1}{2}u_{\mathrm{GW},p}(X,Y).
\end{equation*}
\end{rmk}

\begin{ex} \label{ex:dgwpq-onepoint}
Let $X,Y \in \Gcalw$ where $Y=\{y_0\}$. It follows from Example \ref{ex:exdisonepoint},
  \[
    \dgwpq(X,Y) = \frac{1}{2}\dispq(\{\mu_X \otimes \delta_{y_0}^Y\}) = \frac{1}{2}\diamp(X) \\
    \]
for all $p,q\in [1,\infty]$.
\end{ex}

The following theorem shows that the $(p,q)$-Gromov-Wasserstein distance is a well defined metric on $\mathcal{G}_w$. This is a generalization of both Theorem 5.1 in \cite{memoli2011gromov}, which shows that the original $p$-Gromov-Wasserstein distance $d_{\mathrm{GW},p}$ is a metric on $\mathcal{G}_{w}$, and of Theorem 3.10 in \cite{memoli2021ultrametric}, which shows that the ultrametric $p$-Gromov-Wasserstein distance $u_{\mathrm{GW},p}$ is a $p$-metric on the collection of compact ultrametric spaces.

\begin{thm}\label{thm:props}
The $(p,q)$-Gromov-Wasserstein distance, $\dgwpq$, is a metric on the collection of isomorphism classes of $\mathcal{G}_w$ for all $p,q\in [1,\infty]$. Furthermore, $\dgwpq\leq \dgwppqq$ whenever $p \leq p'$ and $q\leq q'$.       
\end{thm}
We defer the proof of this theorem to  Section~\ref{subsec:pfth1-1}.

\begin{rmk}\label{rmk:relation} In \cite[Section 9]{sturm-ss} Sturm considers a two parameter family of distances, $\Delta\!\!\!\!\Delta_{p,q}$, which is closely related to but differs from $\dgwpq$. A precise relationship is, $$\frac{1}{2}\Delta\!\!\!\!\Delta_{p/q,q}(X,Y) =\big(d_{{\operatorname{GW}}p,q}(X,Y)\big)^{q},$$ for $X,Y\in\mathcal{G}_w$ and $p,q\in [1,\infty)$. See  Remark \ref{rmk:sturm-comp}. 

Also, in contrast with $d_{{\operatorname{GW}}p,q}$, $\Delta\!\!\!\!\Delta_{p,q}$ is not homogeneous: if for $\lambda\geq 0$, $\lambda X$ denotes the metric measure space $(X,\lambda\,d_X,\mu_X)$ (resp. for $\lambda Y)$, $\Delta\!\!\!\!\Delta_{p,q}(\lambda X,\lambda Y)=\lambda^q\Delta\!\!\!\!\Delta_{p,q}(X,Y)$ whereas $d_{{\operatorname{GW}}p,q}(\lambda X,\lambda\,Y)=\lambda \, d_{{\operatorname{GW}}p,q}(X,Y)$.
\end{rmk}

\section{The $(p,q)$-Gromov-Wasserstein Distance Between Spheres}\label{sec:sphr-lb}
Despite the increasing number of applications, the precise value of the Gromov-Wasserstein distance is only known for a few cases \cite{memoli2011gromov,salmona2021gromov}. In this section, we compute the exact value of $\dgwft(\Sp^m_E,\Sp^n_E)$ for arbitrary $m$ and $n$.

\subsection{The Equatorial Coupling}

So far, we have provided lower bounds for the $(p,q)$-Gromov-Wasserstein distance between spheres equipped with the geodesic distance. We now provide a strategy for providing upper bounds. In the following two sections, we will consider the equatorial coupling (defined below) and show that when $p=4$ and $q=2$, the equatorial coupling is optimal for the case of spheres with Euclidean distance.

\smallskip
Assuming $n>m$ we will implicitly use the (isometric) embedding $\bbR^{m+1}\hookrightarrow \bbR^{n+1}$ given by $$(x_1,\ldots,x_{m+1}) \mapsto (x_1,\ldots,x_{m+1},0,\ldots,0).$$
\begin{defn}[Projection and Equatorial map]\label{def:eq-map}
For all $n>m$, we define the
 \emph{projection map} $\pi_{n+1,m+1}:\bbR^{n+1}\rightarrow\bbR^{m+1}$ in the following way:
\begin{align*}
    \pi_{n+1,m+1}:\bbR^{n+1}&\rightarrow\bbR^{m+1}\\
    (x_1,\dots,x_{n+1})&\longmapsto (x_1,\dots,x_{m+1}).
\end{align*}
Note that $\pi_{n+1,m+1}$ as a measurable map from $\bbR^{n+1}$ to $\bbR^{m+1}$.

The \emph{equatorial map} $e_{n,m}:\Sp^n\backslash{O_{n,m}}\rightarrow\Sp^m$ is defined in the following way:

\begin{align*}
    e_{n,m}:\Sp^n\backslash{O_{n,m}}&\longrightarrow\Sp^m\\
    (x_1,\dots,x_{n+1})&\longmapsto\frac{\pi_{n+1,m+1}(x_1,\dots,x_{n+1})}{\Vert \pi_{n+1,m+1}(x_1,\dots,x_{n+1}) \Vert} = \frac{(x_1,\dots,x_{m+1})}{\Vert (x_1,\dots,x_{m+1}) \Vert},
\end{align*}
where $O_{n,m}:=\{x\in \Sp^{n}\subset \mathbb{R}^{n+1}|x_1=\dots=x_{m+1}=0\}$. Note that, since $\mu_n(O_{n,m})=0$, one can extend the domain of $e_{n,m}$ to whole $\Sp^n$ while preserving  measurability. Thus, through a slight abuse of notation, we will view $e_{n,m}$ as a measurable map from $\Sp^n$ to $\Sp^m$; cf. Figure \ref{fig:prj}.  
\end{defn}

We then have the following claim whose proof we omit for brevity.
\begin{claim}\label{claim:eq-coupling} For all $n>m$, the equatorial map $e_{n,m}:\Sp^n \to \Sp^m$  induces a coupling  $\gamma_{m,n} \in \Mcal(\mu_m,\mu_n)$, where $\mu_m$ and $\mu_n$ are the uniform measures on $\Sp^m$ and $\Sp^n$ respectively, and $\gamma_{m,n}$ is given by: 
\begin{equation}\label{eq:eq-coupling}
\gamma_{m,n}:=
(e_{n,m}, \mathrm{id}_{\Sp^n})_\#\mu_n.
\end{equation}
\end{claim}

  We call $\gamma_{m,n}\in \Mcal(\mu_m,\mu_n)$ as in Claim \ref{claim:eq-coupling} \emph{the equatorial coupling}. 

\begin{rmk}
Since $\gamma_{m,n}\in \Mcal(\mu_m,\mu_n)$, it follows trivially that
    $$
    \dgwpq(\Sp_{\bullet}^m,\Sp_{\bullet}^n)\le 
    \dfrac{1}{2}
    \dispq(\gamma_{m,n},\Sp_{\bullet}^m,\Sp_{\bullet}^n).
    $$
\end{rmk}

\begin{ex}($\dis_{4,2}(\gamma_{0,1},\Spg^0,\Spg^1)$)\label{ex:dis42g01} By Remark~\ref{rmk:max-dis42} and Example~\ref{ex:dlb-geo}, we have
\begin{align*}
    &\big(\dis_{4,2}(\gamma_{0,1},\Spg^0,\Spg^1)\big)^4\\
    =&~
    (\diam_4(\Spg^0))^4+(\diam_4(\Spg^1))^4
    -2\int 
    \big(d_0(e_{1,0}(y),e_{1,0}(y'))\big)^2\,
    \big(d_1(y,y')\big)^2\,
    \mu_1(dy)\mu_1(dy')\\
    =&~
    \dfrac{\pi^4}{2}+\dfrac{\pi^4}{5}
    -2\int_{\Sp^1}\int_{\Sp^1}
    (d_0(e_{1,0}(y),e_{1,0}(y')))^2(d_1(y,y'))^2
    \mu_1(dy)\mu_1(dy')\\
    =&~
    \dfrac{\pi^4}{2}+\dfrac{\pi^4}{5}
    -2\times\dfrac{\pi^4}{4}
    =\dfrac{\pi^4}{5}
\end{align*}
where the value of the integral in the last line follows from the calculation in  Example \ref{ex:dis42g01-app}. Hence
$$
\dis_{4,2}(\gamma_{0,1},\Spg^0,\Spg^1)
=\left(\dfrac{1}{5}\right)^{1/4}\pi
\approx 2.101.
$$
This implies that 
$$
\dgwft(\Spg^0,\Spg^1)
\le \dfrac{1}{2}
\dis_{4,2}(\gamma_{0,1},\Spg^0,\Spg^1)
\approx 1.050.
$$
\end{ex}

\begin{ex}\label{ex:dis42g12}
    Following Example~\ref{ex:dis42g01}, one can do analogous calculations in the case of $\Spg^1$ and $\Spg^2$ to obtain
    \begin{align*}
    &\big(\dis_{4,2}(\gamma_{1,2},\Spg^1,\Spg^2)\big)^4\\
    =&~
    (\diam_4(\Spg^1))^4+(\diam_4(\Spg^2))^4
    -2\int 
    (d_1(e_{2,1}(y),e_{2,1}(y')))^2(d_2(y,y'))^2
    \mu_2(dy)\mu_2(dy')\\
    =&~
    \dfrac{\pi^4}{5}+
    24-6\pi^2+\dfrac{\pi^4}{2}
    -2\int_{\Sp^1}\int_{\Sp^1}
    (d_1(e_{2,1}(y),e_{2,1}(y')))^2(d_2(y,y'))^2
    \mu_1(dy)\mu_1(dy')\\
    \approx &~
    \dfrac{\pi^4}{5}+
    24-6\pi^2+\dfrac{\pi^4}{2}
    -2\times 14.159
    \approx 
    4.651
\end{align*}
where we compute the integral in the second last step using numerical integration along with the values of 4-diameters computed in  Example~\ref{ex:dlb-geo}. This immediately implies that
$$
\dgwft(\Spg^1,\Spg^2)
\le \dfrac{1}{2}
\dis_{4,2}(\gamma_{1,2},\Spg^1,\Spg^2)
\approx 0.734.
$$
\end{ex}

\subsection{Exact Determination of $\dgwft(\Sp^m_E,\Sp^n_E)$}\label{subsec:exact}

In this section, we will establish that the equatorial coupling $\gamma_{m,n}$ is a minimizer of the $(4,2)$-distortion $$\disft(\cdot,\Sp^m_E,\Sp^n_E):\mathcal{M}(\mu_m,\mu_n)\to \bbR_+$$ amongst all couplings  between $\mu_m$ and $\mu_n$. Our first result is the following lemma, which exactly computes the $(4,2)$-distortion under the equatorial coupling.

\begin{lem}\label{lem:GW-eu-eq} The $(4,2)$-distortion of the equatorial coupling between spheres $\Sp^m_E$ and $\Sp^n_E$ respectively equipped with their Euclidean distance and uniform measure with $n \ge m$ is 
\[
\disft(\gamma_{m,n},\Sp^m_E,\Sp^n_E) =
    \left[
    \dfrac{4}{m+1}+\dfrac{4}{n+1}
    -\dfrac{8}{m+1}
\left(    \dfrac{\Gamma\left(\tfrac{m+2}{2}\right)
\Gamma\left(\tfrac{n+1}{2}\right)}{
\Gamma\left(\tfrac{m+1}{2}\right)
\Gamma\left(\tfrac{n+2}{2}\right)
}\right)^2\right]^{1/4}.
\]
\end{lem}

We defer the proof of this lemma to Section~\ref{subsec:pf-lems}. We now present the main result of this section which shows the optimality of  the equatorial map:

\mainthm*

\begin{rmk}\label{rmk:gh_v_gw}
The fact that the equatorial coupling is optimal for $\dgwft$ is in sharp contrast to what takes place at the level of the closely related Gromov-Hausdorff distance, where  cognates of the equatorial coupling are far from being optimal \cite{lim2023gromov}.
\end{rmk}

\begin{rmk}[$m=0$] \label{rmk:m=0}
In this interesting special case, Theorem~\ref{thm:equ-opti} implies:
\begin{align*}
\dgwft(\Sp^{0}_E,\Sp^{m+1}_E)
=&~\dfrac{1}{\sqrt{2}}
\left[
\dfrac{n+2}{n+1}
-\dfrac{2}{\pi}
\left(
\dfrac{\Gamma(\frac{n+1}{2})}{\Gamma(\frac{n+2}{2})}
\right)^2
\right]^{1/4}.
\end{align*}
In particular,
\[
\dgwft(\Sp^{0}_E,\Sp^1_E)
=\dfrac{1}{\sqrt{2}}
\left(
\dfrac{3}{2}
-\dfrac{8}{\pi^2}
\right)^{1/4}
\approx 0.644
\,\,
\text{ and }
\,\,
\dgwft(\Sp^{0}_E,\Sp^2_E)
=\dfrac{1}{\sqrt{2}}
\left(
\dfrac{5}{6}
\right)^{1/4}\approx 0.676.
\]
\end{rmk}

\begin{rmk}[$n=m+1$] \label{rmk:ud}
In this interesting special case, Theorem~\ref{thm:equ-opti} implies:
\begin{align*}
\dgwft(\Sp^{m}_E,\Sp^{m+1}_E)
=&~\dfrac{1}{\sqrt{2}}
\left[
\dfrac{2m+3}{(m+1)(m+2)}
-\left(\dfrac{2}{m+1}\right)^3
\left(
\dfrac{\Gamma(\frac{m+2}{2})}{\Gamma(\frac{m+1}{2})}
\right)^4
\right]^{1/4}.
\end{align*}
In particular,
\[
\dgwft(\Sp^{0}_E,\Sp^1_E)
=\dfrac{1}{\sqrt{2}}
\left(
\dfrac{3}{2}
-\dfrac{8}{\pi^2}
\right)^{1/4}
\approx 0.644
\,\,
,
\,\,
\dgwft(\Sp^{1}_E,\Sp^2_E)
=\dfrac{1}{\sqrt{2}}
\left(
\dfrac{5}{6}
-\dfrac{\pi^2}{16}
\right)^{1/4}\approx 0.482.
\]
and
\[
\dgwft(\Sp^{2}_E,\Sp^3_E)
=
\dfrac{1}{\sqrt{2}}
\left(
\dfrac{7}{12}
-\dfrac{8}{27}
\left(\dfrac{16}{\pi^2}\right)
\right)^{1/4}
\approx
0.400.
\]
\end{rmk}

\begin{rmk}[$n=m+2$]\label{rmk:neqmp2} In some cases, it is possible to simplify the formula given by Theorem~\ref{thm:equ-opti}. For example, when $n=m+2$, the quantity in Theorem~\ref{thm:equ-opti} simplifies to the following \emph{explicit} formula:
$$
\dgwft(\Sp^{m}_E,\Sp^{m+2}_E)
=
\dfrac{1}{2^{1/4}}
\left[
\dfrac{1}{(m+1)(m+3)}
+
\dfrac{1}{(m+2)^2(m+3)}
\right]^{\frac{1}{4}}
.
$$
This implies that $\dgwft(\Sp^{m+2}_E,\Sp^m_E)=O(m^{-{1/2}})$ as $m\to \infty$. We compute some exact values below:
$$
\dgwft(\Sp^{0}_E,\Sp^2_E)=\left(\dfrac{5}{24}\right)^{1/4}\approx 0.676
\,\,
\text{;}
\,\,
\dgwft(\Sp^{1}_E,\Sp^3_E)=\left(\dfrac{11}{144}\right)^{1/4}\approx 0.526.
$$
\end{rmk}

\begin{rmk}[Asymptotics for large $m$ and $n$]\label{rem:asympt-n-m} It is clear from Theorem~\ref{thm:equ-opti} that $\dgwft(\Spe^m,\Spe^n)\to 0$ as $m,n\to\infty$. More precisely, note that by Stirling approximation, we have
$$
\dfrac{\Gamma(\frac{m+2}{2})}{\Gamma(\frac{m+1}{2})}=\sqrt{\dfrac{m+1}{2}}(1+O(m^{-1}));
\,\,\dfrac{\Gamma(\frac{n+1}{2})}{\Gamma(\frac{n+2}{2})}
=\sqrt{\dfrac{2}{n+1}}(1+O(n^{-1})),
$$
which implies that
\begin{align*}
    \dgwft(\Spe^m,\Spe^n)
    =&~\dfrac{1}{\sqrt{2}}
    \left[
    \dfrac{1}{m+1}-\dfrac{1}{n+1}
    +O(m^{-2})
    \right]^{1/4}\\
    =&~\dfrac{1}{\sqrt{2}}
    \left[
    \dfrac{n-m}{(m+1)(n+1)}
    +O(m^{-2})
    \right]^{1/4},
\end{align*}
as $m\to\infty$. Thus, if $n-m=O(1)$, we have $\dgwft(\Spe^m,\Spe^n)=O(m^{-1/2})$ as $m\to\infty$.
\end{rmk}

\begin{rmk}[Asymptotics for fixed $m$, large $n$]\label{rem:asympt-fixed-m} As above, for large $k$, Stirling approximation yields 
$$
\dfrac{\Gamma\left(\frac{m+k+1}{2}\right)}
{\Gamma\left(\frac{m+k+2}{2}\right)}
=\sqrt{\dfrac{2}{m+k+1}}\left(1+O(k^{-1})\right).
$$
Theorem~\ref{thm:equ-opti} then implies that
$$
\big(\dgwft(\Spe^m,\Spe^{m+k})\big)^4
-\dfrac{1}{4(m+1)}
=\dfrac{1}{4(m+k+1)}
\left[1-\dfrac{4}{m+1}
\left(
\dfrac{\Gamma(\frac{m+2}{2})}{\Gamma(\frac{m+1}{2})}
\right)^2
\right]
+O(k^{-2}).
$$
Thus, for a fixed $m$ we have
$$
\dgwft(\Spe^m,\Spe^{m+k})
=
\dfrac{1}{\sqrt{2}}\left[\dfrac{1}{m+1}\right]^{1/4}
+O(k^{-1/4})
$$
as $k\to\infty$.

\end{rmk}

\subsection{The proof of Theorem~\ref{thm:equ-opti}.}
We divide the proof into  several steps.

\subsubsection{Preliminaries.} \label{sec:prelimn-proof}
By property of Euclidean distances, we have
\begin{align*}
    \|x-x'\|^2
    =&~2\left(
    1-\langle x,x'\rangle\right)
    \quad
    \text{and}
    \quad
    \|y-y'\|^2
    =~2\left(
    1-\langle y,y'\rangle\right).
\end{align*}
Consider any coupling $\gamma\in\mathcal{M}(\mu_m,\mu_n)$. By the definition of distortion from Equation~\eqref{eq:def-distort}, when $d_X$ and $d_Y$ are both the Euclidean distances, one has 
\begin{align*}\label{eq:eu-dis}
    \dis^4_{4,2}(\gamma) 
    =&~  
    \int_{\Sp^m\times \Sp^n}\int_{\Sp^m\times\Sp^n} 
    (\|x-x'\|^2
    -\|y-y'\|^{2} )^2
    \ \gamma(dx \times dy) \ \gamma(dx' \times dy') \\
    =&~
     \int_{\Sp^m\times\Sp^n}\int_{\Sp^m\times\Sp^n} 
    | 2\langle x, x'\rangle 
    -2\langle y, y'\rangle|^{2} 
    \ \gamma(dx \times dy) \ \gamma(dx' \times dy')\\
    =&~
    4\iint_{\Sp^m\times\Sp^m}\langle x, x'\rangle^2 \ \mu_m(dx) \ \mu_m(dx')+4\iint_{\Sp^n\times\Sp^n}\langle y, y'\rangle^2 \ \mu_n(dy) \ \mu_n(dy')\\
    &\quad-8\int_{\Sp^m\times\Sp^n}\int_{\Sp^m\times\Sp^n}\langle x, x'\rangle\langle y, y'\rangle \ \gamma(dx \times dy) \ \gamma(dx' \times dy').\numberthis
\end{align*}

Thus for any coupling $\gamma\in\mathcal{M}(\mu_m,\mu_n)$ we have
\begin{equation}\label{eq:eu-dis-2}
    \dis^4_{4,2}(\gamma)
    =4\iint
    \langle x, x'\rangle^2 \ \mu_m(dx) \ \mu_m(dx')+4\iint_{\Sp^n\times\Sp^n}\langle y, y'\rangle^2 \ \mu_n(dy) \ \mu_n(dy')-8J(\gamma)
\end{equation}
where we define
\begin{equation}\label{eq:defJmu}
J(\gamma):=
\int
\int
\langle x, x'\rangle\langle y, y'\rangle 
\ \gamma(dx \times dy) \ \gamma(dx' \times dy').
\end{equation}
Since the first two terms on the right hand side of \eqref{eq:eu-dis-2} do not depend on the coupling $\gamma$, we have the following equivalent optimization problem: 
$$ 
\gamma 
\text{ minimizes}
\,\, \dis_{4,2}^4(\gamma)
\Leftrightarrow
\gamma
\text{ maximizes}
\,\, J(\gamma) 
$$
where both optimizations are over the space of couplings $\gamma\in\mathcal{M}(\mu_m,\mu_n)$. In the rest of the proof we therefore focus on \emph{maximizing} $J(\gamma)$.

\subsubsection{A change of coordinates}\label{sec:change-of-coords} 
In this section we prove a lemma which permits simplifying the functional $J$ defined above. See the discussion on page \pageref{page:change-coords}.

 \begin{lem}\label{lem:change-coords} Let $\alpha \in \mathcal{P}(\bbR^{m+1})$ and $\beta  \in \mathcal{P}(\bbR^{n+1})$ where $n\geq m$ be two rotationally invariant measures with barycenters coinciding with the respective origins. Consider the functional $J:\mathcal{M}(\alpha,\beta)\to \bbR$ defined above. Then, 
 $$\max_{\gamma\in \mathcal{M}(\alpha,\beta)} J(\gamma) = \max_{\gamma\in \mathcal{M}(\alpha,\beta)} D(\gamma) $$
 where 
 $$ D(\gamma) := 
\sum_{k=1}^{m+1}
\left[\int 
x_k y_k \
\gamma (dx \times dy) \right]^2 .$$
 \end{lem}

\begin{proof}
Applying the  linearity of the integral and the identity  ${\rm trace}(AB)={\rm trace}(BA)$ for conformable matrices\footnote{I.e. the matrices can be multiplied.} $A,B$, we  compute,
\begin{align*}\label{eq:Jmu-Fnorm}
J(\gamma)
=&~
\int  
\int 
\langle x,x'\rangle
\langle y,y'\rangle \
\gamma (dx \times dy) \ \gamma (dx' \times dy') \\
=&~
\int  
\int 
{\rm  tr}\left[
(x')^{\top}xy^{\top}y' 
\right] \
\gamma (dx \times dy) \ \gamma (dx' \times dy') \\
=&~
\int
\int 
{\rm tr}\left[ 
x y^{\top}y'(x')^{\top}
\right]
\gamma (dx \times dy) \ \gamma (dx' \times dy') \\
=&~
{\rm tr}
\left[
\int   
\int 
xy^\top 
(x'(y')^{\top})^\top \
\gamma (dx \times dy) \gamma (dx' \times dy')
\right]\\
=&~
{\rm tr}
\left[
\int  
xy^\top \
\gamma(dx \times dy)
\left(
\int  
x' (y')^\top \
\gamma (dx' \times dy')
\right)^\top
\right] \\
=&~
\Bigg\|
\underbrace{
\int  
xy^\top
\ \gamma(dx \times dy)}_{M_\gamma}
\Bigg\|_{\rm F}^2 
\end{align*}

We now manipulate the matrix $M_\gamma\in \bbR^{(m+1)\times (n+1)}$ in order to simplify the optimization problem; see the discussion about related work on page \pageref{sec:rel-wk}. Hereon, we write $\II_d$ to mean the identity matrix of size $d$.

Consider the (possibly non-unique) \emph{singular value decomposition} $$ M_{\gamma}=P_{\gamma}\Delta_{\gamma}Q_{\gamma}^{\top}
$$ where
\begin{itemize}
\item $\Delta_{\gamma}\in \bbR^{(m+1)\times (m+1)}$ is a diagonal matrix containing the singular values of $M_\gamma$,
\item $P_\gamma\in \bbR^{(m+1)\times (m+1)}$ and $Q_\gamma\in \bbR^{(n+1)\times (m+1)}$ satisfy
$P_\gamma^{\top}P_\gamma
=P_\gamma P_\gamma^{\top}=\II_{m+1}
$
and
\item 
$Q_\gamma^{\top}Q_\gamma=\II_{m+1}$, i.e. $P_\gamma$ is orthonormal and $Q_\gamma$ is semi-orthonormal. 
\end{itemize}
We now define
$$
U_{\gamma}:=P_{\gamma}^{\top}\in \bbR^{(m+1)\times (m+1)}
\text{ and }
V_{\gamma}:=\begin{pmatrix}
Q_{\gamma} &
Q_{\gamma}^{\perp}
\end{pmatrix}^{\top}\in \bbR^{(n+1)\times(n+1)}
$$
where $Q_{\gamma}^{\perp}\in \bbR^{(n+1)\times(n-m)}$ is any semi-orthornormal matrix, i.e., $(Q_{\gamma}^{\perp})^{\top}Q_{\gamma}^{\perp}=\II_{n-m}$, which also satisfies $Q_{\gamma}^{\top}Q_{\gamma}^{\perp}=\mathbf{0}_{(m+1)\times(n-m)}$. Note that by construction $U_\gamma$ and $V_\gamma$ are orthonormal, i.e., $U_\gamma^{\top}U_\gamma=U_\gamma U_\gamma^{\top}=\II_{m+1}$ and $V_\gamma^{\top}V_\gamma=V_\gamma V_\gamma^{\top}=\II_{n+1}$.

Recall that the marginals of $\gamma$ are  $\mu_m$ and $\mu_n$. Let $U_{\gamma}\in\bbR^{(m+1)\times (m+1)}$ and $V_{\gamma}\in\bbR^{(n+1)\times (n+1)}$ be the two orthonormal matrices  defined above and $T_{U_{\gamma}}$ and $T_{V_{\gamma}}$ be the linear maps they induce by left multiplication (i.e. $T_{U_{\gamma}} : \mathbb{R}^{m + 1} \to \mathbb{R}^{m+1}$ is defined as $v \mapsto U_{\gamma}v$, with $T_{V_{\gamma}}$ defined similarly).
By the assumed symmetry of the measures $\alpha$ and $\beta$, we may pushforward $\gamma$ through the associated maps and  still obtain a coupling between $\alpha$ and $\beta$, that is:
\[
(T_{U_{\gamma}}, T_{V_{\gamma}})_{\#} \gamma \in \mathcal{M}(\alpha,\beta).
\]

Now, we define 
$\Tilde{\gamma} := (T_{U_{\gamma}}, T_{V_{\gamma}})_{\#} \gamma$, and see that
\begin{align*}
M_{\Tilde{\gamma}}=&~
\int 
x y^\top \
\Tilde{\gamma} (dx \times dy)\\
=&~
U_\gamma
\left(\int 
 x y^\top 
\gamma (dx \times dy)\right)
V_\gamma^{\top}\\
=&~
U_\gamma
M_\gamma
V_\gamma^{\top}\\
=&
\begin{pmatrix}
\Delta_\gamma &
\mathbf{0}_{(m+1)\times (n-m)}
\end{pmatrix}.
\end{align*}

Since the Frobenius norm of a matrix is simply the Euclidean norm of its singular values,
\[
J(\gamma)=
\left\|
M_\gamma
\right\|_{\rm F}^2
=\left\| \Delta_\gamma \right\|_{\rm F}^2
=\left\|
M_{\Tilde{\gamma}}
\right\|_{\rm F}^2
=J(\tilde{\gamma}).
\]
That is, for any optimal coupling $\gamma$, there exists another optimal coupling $\Tilde{\gamma}$ for which $M_{\Tilde{\gamma}}$ is of the form $\begin{pmatrix}
\Delta_\gamma &
\mathbf{0}_{(m+1)\times (n-m)}
\end{pmatrix}$ for a diagonal matrix $\Delta_{\gamma}\in \bbR^{(m+1)\times(m+1)}$. 

We can then write:
\begin{align*}
J(\gamma)=J(\Tilde{\gamma})
=\left\|
M_{\Tilde{\gamma}}
\right\|_{\rm F}^2
=\sum_{k,l}\left[
\int 
x_k y_l \
\Tilde{\gamma} (dx \times dy)
\right]^2
=\sum_{k=1}^{m+1}\left[
\int 
x_k y_k \
\Tilde{\gamma} (dx \times dy)
\right]^2
\end{align*}
where the last equality follows since $M_{\Tilde{\gamma}}=\begin{pmatrix}
\Delta_\gamma &
\mathbf{0}_{(m+1)\times (n-m)}
\end{pmatrix}$ for a diagonal matrix $\Delta_{\gamma}\in \bbR^{(m+1)\times(m+1)}$.
Thus, maximizing $J(\gamma)$ is equivalent to: 
\begin{equation}\label{eq:maxJ-diag}
\text{maximize }
D(\gamma) = 
\sum_{k=1}^{m+1}
\left[\int 
x_k y_k \
\gamma (dx \times dy) \right]^2
\text{ over }
\gamma\in\mathcal{M}(\alpha,\beta).
\end{equation}
That is, any optimizer of $D$ is an optimizer of $J$ and any optimizer of $J$ can be pushed forward via a rotation to an optimizer of $D$. 
\end{proof}

\subsubsection{Optimizing $J(\gamma)$ over couplings between standard Gaussians.} \label{sec:gaussian}

In this section, we focus on optimizing $J(\gamma)$ for standard Gaussian marginals. Despite the close connections between the standard Gaussian measure and the uniform measure on the sphere, there are fundamental differences among the two in terms of optimizing $J(\gamma)$ over all possible couplings. In particular, the techniques of this section will be, strictly speaking, applicable only to Gaussian measures and not to the uniform measures. \textbf{The general method of proof will nevertheless pave the way for our proof of Theorem \ref{thm:equ-opti}}; see Remark \ref{rmk:non-tight}.

We will use the notation $\eta_{d}$ to denote the standard Gaussian measure on $\bbR^{d}$ (so that $\eta_1$ will denote the standard Gaussian measure on $\bbR)$. In the usual notation for Normal distributions $\eta_{d} = N(\mathbf{0},\II_{d})$. 

Using the projection map $\pi_{n+1,m+1}$ from Definition \ref{def:eq-map}, we  define the following  coupling between standard Gaussians:

\begin{defn}\label{def:proj-coupling}
    For all $n\ge m$, the projection map $\pi_{n+1,m+1}:\bbR^{n+1}\to\bbR^{m+1}$ induces a coupling $\gamma^{\rm gauss}_{m+1,n+1}\in\mathcal{M}(\eta_{m+1},\eta_{n+1})$  given by:
    \[
    \gamma^{\rm gauss}_{m+1,n+1}:= (\pi_{n+1,m+1}, \mathrm{id}_{\bbR^{n+1}})_\#\eta_{n+1}.
    \]
 \end{defn}

\begin{rmk}\label{rem:relation-couplings}
Note that we can recover the equatorial coupling from (\ref{eq:eq-coupling}) as follows $\gamma_{m,n} = (f_{m+1},f_{n+1})_\# \gamma_{m+1,n+1}^\mathrm{gauss}$ where $f_{m+1}:\bbR^{m+1}\backslash \{\textbf{0}\}\to \Sp^m$ is the central projection map: $x\mapsto \tfrac{x}{\|x\|}$. 
\end{rmk}

When the marginals of $\gamma$ are standard Gaussian measures $\eta_{m+1}$ and $\eta_{n+1}$, the optimization problem in \eqref{eq:maxJ-diag} can be solved by relaxing the optimization into an optimization over the coordinate wise pushforwards of $\gamma$. This leads to the following proposition.

\begin{prop}\label{prop:gaussian}
    Suppose $n\geq m$. Then,
    \[
    \max_{\gamma\in \mathcal{M}(\eta_{m+1},\eta_{n+1})}J(\gamma)=J(\gamma^{\rm gauss}_{m+1,n+1})=m+1.
    \]

\end{prop}

\begin{rmk}
     Note that \cite[Proposition 4.1]{salmona2021gromov} gives a more general claim than Proposition \ref{prop:gaussian} and consequently requires a much more sophisticated method of proof (cf. \cite[Lemma 3.2]{salmona2021gromov}). 
\end{rmk}

\begin{proof}[Proof of Proposition~\ref{prop:gaussian}]
    By Lemma \ref{lem:change-coords}, we can equivalently maximize the functional $D(\gamma)=\displaystyle\sum_{k=1}^{m+1}\left[\int x_k\,y_k\gamma(dx\times dy)\right]^2$ over all couplings. To proceed with this, write
\begin{align}
    \sup_{\gamma\in\mathcal{M}(\eta_{m+1},\eta_{n+1})} D(\gamma)
    \le&~\sum_{k=1}^{m+1}\sup_{\gamma\in\mathcal{M}(\eta_{m+1},\eta_{n+1})}
    \left[\int
    x_k y_k \
    \gamma (dx \times dy) \right]^2\label{eq:non-tight}\\
    =&~ (m+1)
    \sup_{\gamma\in\mathcal{M}(\eta_{m+1},\eta_{n+1})}
    \left[\int
    x_1 y_1 \
    \gamma (dx \times dy) \right]^2.\label{eq:gaussian-upbd}
    \end{align}
Since the optimization on the right hand side depends only on $x_1,y_1$, one can then optimize over the first coordinate pushforwards of $\gamma$. To be more precise, if $\varphi_1^{m+1}:\mathbb{R}^{m+1}\to \mathbb{R}$ and $\varphi_1^{n+1}:\mathbb{R}^{n+1}\to \mathbb{R}$ denote the respective projections onto the first coordinate, consider $\gamma_1:=(\varphi_{1}^{m+1},\varphi_{1}^{n+1})_{\#}\gamma$. It follows by Lemma 3.2 of \cite{chowdhury2019gromov} that $$(\varphi_{1}^{m+1},\varphi_{1}^{n+1})_{\#}(\mathcal{M}(\eta_{m+1},\eta_{n+1}))=\mathcal{M}((\varphi_{1}^{m+1})_{\#}(\eta_{m+1}),(\varphi_{1}^{n+1})_{\#}(\eta_{n+1})) = \mathcal{M}(\eta_1,\eta_1),$$ where the last step follows from the fact that the pushforward through a (one dimensional) coordinate projection of the standard Normal in $\bbR^d$ is a standard normal on $\bbR$.  
    Thus, the optimization on the right hand side above can be  equivalently written as
    \begin{align*}\label{eq:dgamma-ubd-gaussian-calc}
    \sup_{\gamma\in\mathcal{M}(\eta_{m+1},\eta_{n+1})}
    \left[\int_{\mathbb{R}^{m+1} \times \mathbb{R}^{n+1}}
    x_1 y_1 \
    \gamma (dx \times dy) \right]^2
    =&~
    \sup_{\gamma_1\in\mathcal{M}(\eta_1,\eta_1)}
    \left[\int_{\mathbb{R} \times \mathbb{R}}
    x_1 y_1 \
    \gamma_1 (dx_1 \times dy_1) \right]^2\\
    \le&~
    \left[
    \int_{\mathbb{R}}
    x_1^2\
    \eta_1(dx_1) 
    \right]
    \left[
    \int_{\mathbb{R}}
    y_1^2 \
    \eta_1(dy_1) 
    \right]\\
    =&~
    \left[
    \int_{\mathbb{R}}
    z^2
    \ 
    \dfrac{1}{\sqrt{2\pi}}\exp(-z^2/2)\,dz
    \right]^2=1. \numberthis
    \end{align*}
    The first inequality follows by applying the Cauchy-Schwarz inequality. The last equality uses the well-known computation of the second moment of a one dimensional standard Gaussian measure. Plugging \eqref{eq:dgamma-ubd-gaussian-calc} into \eqref{eq:gaussian-upbd} shows that
    \begin{equation}\label{eq:dgamma-upbd-gauss}
        \sup_{\gamma\in\mathcal{M}(\eta_{m+1},\eta_{n+1})} D(\gamma)\le (m+1).
    \end{equation}
    To finish the proof of Proposition~\ref{prop:gaussian} we now note that
    \begin{align*}
    \sup_{\gamma\in\mathcal{M}(\eta_{m+1},\eta_{n+1})} D(\gamma)
    \ge&~
    D(\gamma^{\rm gauss}_{m+1,n+1})\\
    =&~\sum_{k=1}^{m+1}
    \left[\int_{\bbR^{m+1}\times \bbR^{n+1}}
    x_k\,y_k \ \gamma^{\rm gauss}_{m+1,n+1}(dx\times dy)
    \right]\\
    =&~
    \sum_{k=1}^{m+1}
    \left[\int_{\bbR^{m+1}}
    x_k^2 \ \eta_{m}(dx)
    \right]\\
    =&~
    \sum_{k=1}^{m+1}
    \left[\int_{\bbR^{m+1}}
    x_k^2\ \dfrac{1}{(2\pi)^{(m+1)/2}}\exp\left(-\sum_{k=1}^{m+1}x_k^2/2\right)dx_1\,\ldots dx_{m+1}
    \right]\\
    =&~
    \sum_{k=1}^{m+1}
    \left[\int_{\bbR}
    x_k^2\, \dfrac{1}{\sqrt{2\pi}}\exp(-x_k^2/2) \, dx_k
    \right]=\sum_{k=1}^{m+1}1=(m+1).
    \end{align*}
    Here the first equality follows by definition of $D(\gamma)$, the second follows by the definition of $\gamma^{\rm gauss}_{m+1,n+1}$, the third one is due to the definition $\eta_{m+1}:=N(\mathbf{0},\II_{m+1})$, the fourth  and the fifth inequalities follow by standard Gaussian integral computations. Together with \eqref{eq:dgamma-upbd-gauss} this finishes the proof of Proposition~\ref{prop:gaussian}.
\end{proof}

\begin{rmk}\label{rmk:non-tight}
    The strategy used in Proposition \ref{prop:gaussian} for finding the optimal coupling for Gaussian marginals, does not, however, work for uniform measures on the sphere because the inequality in Equation~\eqref{eq:non-tight} is \emph{not tight for uniform measures on $\Sp^m$} whenever $m>0$. This is because whenever $m>0$, the uniform measure on $\Sp^m$ is not a product measure over its coordinates --- and thus the problem of maximizing $D(\gamma)$ cannot be solved by a coordinate wise approach: Optimizing in the first coordinate leads to constraints on the feasible set of the optimization in the second coordinate, and so on. Note that this is in contrast with the case where $\gamma$ has  standard Gaussian measures as marginals (considered in Section~\ref{sec:gaussian}), which are indeed product distributions.  See \Cref{q:gauss}.
 \end{rmk}

We now describe a variant of the  approach used in the proof of Proposition \ref{prop:gaussian} which takes into consideration the required dependence between the coordinates.

\subsubsection{The conclusion of the proof of Theorem~\ref{thm:equ-opti}.}

Recall from the calculations in Section \ref{sec:prelimn-proof} that minimizing $\dis_{4,2}$ over all couplings between the uniform measures $\mu_m$ and $\mu_n$ leads to maximizing the functional $J$ defined in (\ref{eq:defJmu}). Note that the uniform measures on spheres $\mu_m\in\mathcal{P}(\Sp^m)$ and $\mu_n\in\mathcal{P}(\Sp^n)$  are elements of $\mathcal{P}(\bbR^{m+1})$ and $\mathcal{P}(\bbR^{n+1})$, respectively. Hence, we can invoke Lemma \ref{lem:change-coords}  and equivalently maximize the functional $D$ over all such couplings. Before tackling this, we need some preparations.

\smallskip
To simplify subsequent computations, we define the projections,
$y=(y_A,y_B)^{\top}
\text{ where }
y_A\in\bbR^{m+1},\ \text{ and }
y_B\in\bbR^{n-m}$.
Fixing $\gamma$, we use the projection definition to define the decomposition of $\Sp^n$ as a union of products of smaller spheres, 
\[
A_t := \{ y \in \Sp^n : \| y_A \| = t \} = \left( t \cdot \Sp^m \right) \times \left( \sqrt{1 - t^2} \cdot \Sp^{n - m-1} \right),\,\,\mbox{$t\in[0,1]$}
\]
so that $\bigcup_{t=0}^{1} A_t = \Sp^n$. 
Let the measure $\nu \in \mathcal{P}([0,1])$ be the pushforward of $\gamma$ by $(x, y) \mapsto \|y_A\|$. Then, by the Disintegration Theorem \cite[Theorem 5.3.1]{ambrosio2005gradient}, there is a measure-valued map $t\mapsto\gamma_t$ from $[0,1]$ to $\mathcal{P}(\Sp^m\times\Sp^n)$ such that: 
\begin{enumerate}
\item[(1)] $t\mapsto\gamma_t(B)$ is measurable for all Borel set $B\subseteq\Sp^m\times\Sp^n$,  
\item[(2)] $\gamma=\int_0^1\gamma_t\,\nu(dt)$, and 
\item[(3)] $\supp[\gamma_t]\subseteq A_t \times \Sp^m$ (so we will view $\gamma_t$ as a probability measure on $A_t \times \Sp^m$ for each $t\in [0,1]$).
\end{enumerate}

Marginalizing this disintegration over its first factor, $\Sp^m$, we derive a disintegration of the uniform measure $\mu_n$ according to the map $y \mapsto \| y_A \|$, which we denote by $\overline{\gamma}_{t}$.
This new marginal disintegration is, in particular, defined such that for all Borel subsets $B \subseteq \Sp^n$, $\overline{\gamma}_{t} (B) := {\gamma}_t( B \times \Sp^m)$.
To check that this is indeed a disintegration, let $\varphi : \Sp^n \to \mathbb{R}$ be a measurable function, and then, since $\gamma$ marginalizes to $\mu_n$, we have 
\begin{equation}\label{eq:cond-exp}
\begin{split}
    \int_{\Sp^n} \varphi(y) \ \mu_n (dy) 
    &= \int_{\Sp^m \times \Sp^n} \ \varphi(y) \ \gamma (dx \times dy) \\
    &= \int_0^1 \int_{\Sp^m \times A_t} \varphi(y) \ {\gamma}_t(dx \times dy)\, \nu (dt) \\
    &= \int_0^1 \int_{A_t} \varphi(y) \ \overline{\gamma}_{t} (dy) \,\nu (dt).
\end{split}
\end{equation}
Since $\overline{\gamma}_{t}$ is a disintegration of $\mu_n$, it has a symmetry informed by the symmetry of $\mu_n$.
In particular, for any $\varphi : \Sp^n \to \mathbb{R}$ and any $U \in O(m + 1)$ and $V \in O(n-m)$,
\begin{align*}
\int_0^1 \int_{A_t} \varphi(y_A, y_B) \ ((T_U, T_V)_\# \overline{\gamma}_{t})(dy_A \times dy_B) \,\nu(dt)
&= \int_0^1 \int_{A_t} \varphi(U y_A, V y_B) \ \overline{\gamma}_{t}(dy_A \times dy_B) \,\nu(dt) \\
&= \int_{\Sp^n} \varphi(U y_A, V y_B) \ \mu_n(dy_A \times dy_B)  \\
&=\int_{\Sp^n} \varphi(y_A, y_B) \ (T_U,T_V)_\#\mu_n(dy_A \times dy_B)\\
&=
\int_{\Sp^n} \varphi(y_A, y_B) \ \mu_n(dy_A \times dy_B)  \\
&= \int_0^1 \int_{A_t} \varphi(y_A, y_B) \ \overline{\gamma}_{t}(dy_A \times dy_B) \,\nu(dt).
\end{align*}

So, for any $U \in O(m+1)$ and $V \in O(n-m)$, $(T_U, T_V)_\# \overline{\gamma}_{t} = \overline{\gamma}_{t}$ for almost every $t$.
The only probability measure on $A_t$ that satisfies these conditions is the product of uniform measures on both factors $y_A$ and $y_B$ (a.e.). 
Marginalizing over $y_B$, we denote the induced measure on $y_A$ as $\mu_{t \cdot \Sp^m}$. By the above argument $\mu_{t\cdot \Sp^m}$ is the uniform measure over $t\cdot\Sp^m$.

The disintegration described above allows for the  computation: 
\begin{align*}\label{eq:Dmu-upbd}
D(\gamma)
=&~
\sum_{k=1}^{m+1}
\left[\int_{\Sp^m \times \Sp^n}
x_k y_k \
\gamma (dx \times dy) \right]^2
\\ 
=&~
\sum_{k=1}^{m+1}
\left[
\int_0^1 t\cdot \int_{\Sp^m \times A_t}
\frac{x_k y_k}{t} \
{\gamma}_t (dx \times dy) \ \nu(dt) \right]^2
\\ 
\le&~
\sum_{k=1}^{m+1}
\left[
\int_0^1 t\cdot \left(
\int_{\Sp^m \times A_t}
\frac{y_k^2}{t^2} \ {\gamma}_t(dx \times dy) \right)^{1/2}
\cdot
\left(
\int_{\Sp^m \times A_t}
x_k^2 \ {\gamma}_t(dx \times dy) \right)^{1/2}
\ \nu(dt) \right]^2 \tag{$\ast$}
\\ 
=&~
\sum_{k=1}^{m+1}
\left[
\int_0^1 t\cdot \left(
\int_{\Sp^m}
y_k^2 \ \mu_m(dx) \right)^{1/2}
\cdot
\left(
\int_{\Sp^m \times A_t}
x_k^2 \ {\gamma}_t(dx \times dy) \right)^{1/2}
\ \nu(dt) \right]^2
\\ 
=&~
\frac{1}{m+1}
\sum_{k=1}^{m+1}
\left[
\int_0^1 t\cdot 
\left(
\int_{\Sp^m \times A_t}
x_k^2 \ {\gamma}_t(dx \times dy) \right)^{1/2}
 \ \nu(dt) \right]^2.
\end{align*}
The inequality is an application of the Cauchy-Schwarz inequality, and the subsequent equality is justified by the following computation:
\begin{equation}\label{eq:pf-equ-0}
\int_{\Sp^m \times A_t} \frac{y_k^2}{t^2} \ \mu_t (dx \times dy)
=
\int_{A_t} \frac{y_k^2}{t^2} \ \overline{\gamma}_t (dy)
=
\int_{t \cdot \Sp^m} \frac{y_k^2}{t^2}  \ \mu_{t \cdot \Sp^m}(dy_A)
=
\int_{\Sp^m} \tilde{y}_k^2 \ \mu_m (d\tilde{y}),
\end{equation}
where the first two equalities are given by integrating out $x$ and $y_B$ respectively. 
The final equality follows from a change of variables by the map $y \mapsto t^{-1} \cdot y$, using the fact that $\mu_{t\cdot\Sp^m}$ is a uniform measure on $t\cdot\Sp^m$.

We would like to pass the summation inside the square to apply the Cauchy-Schwarz inequality, so we will write out the squared integral as a product of integrals over independent variables:
\begin{align*}\label{eq:Dmu2}
&~D(\gamma) \\
&\le \frac{1}{m+1}
\sum_{k=1}^{m+1}
\left[
    \int_0^1 t\cdot 
    \left(
    \int_{\Sp^m \times A_t}
        x_k^2 \ \gamma_t(dx \times dy) \right)^{1/2} \hspace{-1em}
    \nu(dt) 
\right]^2\\
&=
\frac{1}{m+1}
\sum_{k=1}^{m+1}
\left[
    \int_0^1 
        t\cdot \left(
            \int_{\Sp^m \times A_t}
                x_k^2 
            \gamma_t(dx \times dy) 
        \right)^{1/2} \hspace{-1em}
    \nu(dt) \cdot
    \int_0^1 t'\cdot 
        \left(
            \int_{\Sp^m \times A_{t'}}
                {x'_k}^2 
            \gamma_{t'}(dx' \times dy') 
        \right)^{1/2} \hspace{-1em}
    \nu(dt') 
\right] \\
&=
\frac{1}{m+1}
\int_0^1 \int_0^1 
    tt'\cdot 
    \sum_{k=1}^{m+1}
    \left(
        \int_{\Sp^m \times A_t}
            x_k^2 \ 
        \gamma_t(dx \times dy) 
    \right)^{1/2}
    \left(
        \int_{\Sp^m \times A_{t'}}
            {x'_k}^2 \ 
        \gamma_{t'}(dx' \times dy') 
    \right)^{1/2}
\nu(dt') \nu(dt) \\
&\le \frac{1}{m+1}
\int_0^1 \int_0^1 
    tt'\cdot 
    \left(
        \sum_{k=1}^{m+1}
        \int_{\Sp^m \times A_t}
            x_k^2 \ 
        \gamma_t(dx \times dy) 
    \right)^{1/2}
    \left(
        \sum_{k=1}^{m+1}
        \int_{\Sp^m \times A_{t'}}
            {x'_k}^2 \ 
        \gamma_{t'}(dx' \times dy') 
    \right)^{1/2} \hspace{-1em}
\nu(dt') \nu(dt) \tag{$\ast \ast$}
\\
&=\frac{1}{m+1}
\int_0^1 \int_0^1 tt'\cdot 
\left(
    \int_{\Sp^m \times A_t}
    1 \ \gamma_t(dx \times dy) 
\right)^{1/2}
\left(
\int_{\Sp^m \times A_{t'}}
    1 \ \gamma_t(dx \times dy) 
\right)^{1/2} \hspace{-1em}
\nu(dt') \nu(dt) \\
&= \frac{1}{m+1} 
\left( 
    \int_{\Sp^m \times \Sp^n} 
        \| y_A \| 
    \gamma (dx \times dy)
\right)^2
= \frac{1}{m+1}
\left(
    \int_{\Sp^n} 
        \| y_A \| 
    \mu_n (dy)
\right)^2.
\end{align*}

We appeal to a well known characterization of $\mu_n$ to compute this integral. More precisely, if $Z_1,\dots,Z_{n+1}$ are independent $N(0,1)$ distributed random variables, the law of  
$$
y=\dfrac{(Z_1,Z_2,\dots,Z_{n+1})}{\,\,(Z_1^2+\dots+Z_{n+1}^2)^{1/2}}
$$
is given by $\mu_n$, which follows by the spherical symmetry of the standard multivariate Gaussian distribution. Then by definition of $y_A$, we have
\[
\|y_A\|^2\stackrel{d}{=}
\dfrac{Z_1^2+Z_2^2+\dots+Z_{m+1}^2}{Z_1^2+Z_2^2+\dots+Z_{n+1}^2}
\sim {\rm Beta}\left(
\dfrac{m+1}{2},\dfrac{n-m}{2}
\right),
\]
i.e., the Beta distribution with parameters $\frac{m+1}{2}$ and  $\frac{n-m}{2}$. See, e.g., Theorem 5.8.4 and Section 8.2 of \cite{degroot2012probability}. The ``$\stackrel{d}{=}$" symbol denotes an equality in distribution. Note that if $X\sim {\rm Beta}(a,b)$, then 
\begin{align*}
  \EE(\sqrt{X})
  =&~
  \dfrac{1}{\beta(a,b)}
  \int_0^1
  \sqrt{x}x^{a-1}
  (1-x)^{b-1}dx
  =
  \dfrac{1}{\beta(a,b)}
  \int_0^1
  x^{a-1/2}
  (1-x)^{b-1}dx
  =
  \dfrac{\beta(a+1/2,b)}{\beta(a,b)}
\end{align*}  
where $\beta(a,b):=\int_0^1 x^{a-1}(1-x)^{b-1}dx$. Thus,
\[
\int_{\Sp^n}\|y_A\| \ \mu_n(dx)
=\dfrac{\beta((m+2)/2,(n-m)/2)}{\beta((m+1)/2,(n-m)/2)}
=\dfrac{
\Gamma\left(\tfrac{m+2}{2}\right)
\Gamma\left(\tfrac{n+1}{2}\right)}{\Gamma\left(\tfrac{m+1}{2}\right)
\Gamma\left(\tfrac{n+2}{2}\right)
}.
\]
We now return to compute the three terms on the right hand side of \eqref{eq:eu-dis} one by one. First, by the spherical symmetry and the fact that $\mu_m$ is the uniform measure on $\Sp^m$, we observe that $x'\mapsto\int_{\Sp^m} \langle x, x'\rangle^2 \ \mu_m(dx)$ is constant for all $x'\in\Sp^m$. In particular,
$$
\iint_{\Sp^m\times\Sp^m}\langle x, x'\rangle^2 \ \mu_m(dx) \ \mu_m(dx')
=\int_{\Sp^m}\langle x, e_1\rangle^2 \ \mu_m(dx)
=\cdots=
\int_{\Sp^m}
\langle x, e_{m+1}\rangle^2 
\ \mu_m(dx).$$
Hence
$$
(m+1)\iint_{\Sp^m\times\Sp^m}\langle x, x'\rangle^2 \ \mu_m(dx) \ \mu_m(dx')
=\sum_{i=1}^{m+1}
\int_{\Sp^m}\langle x, e_i\rangle^2 \ \mu_m(dx)
=\int_{\Sp^m}\sum_{i=1}^{m+1}
\langle x, e_i\rangle^2 \ \mu_m(dx)=1$$
Therefore, 
$$
\iint_{\Sp^m\times\Sp^m}
\langle x, x'\rangle^2 \ \mu_m(dx) \ \mu_m(dx')
=\frac{1}{m+1}
\text{ and }
\iint_{\Sp^n\times\Sp^n}\langle y, y'\rangle^2 \ \mu_n(dy) \ \mu_n(dy')=\frac{1}{n+1}.
$$

It follows that any coupling $\gamma\in\mathcal{M}(\mu_m,\mu_n)$ satisfies 
\begin{align*}
(\disft(\gamma))^4
=&~4\left(
    \dfrac{1}{m+1}+\dfrac{1}{n+1}
\right)
-8J(\gamma)\\
\le&~
4\left(
\dfrac{1}{m+1}+\dfrac{1}{n+1}
\right)
-
\dfrac{8}{m+1}
\left(
    \int_{\Sp^n} \|y_A\| \ \mu_n (dx)
\right)^2\\
=&~
4\left(
\dfrac{1}{m+1}+\dfrac{1}{n+1}
\right)
-\dfrac{8}{m+1}
\left(
\dfrac{
\Gamma\left(\tfrac{m+2}{2}\right)
\Gamma\left(\tfrac{n+1}{2}\right)}{\Gamma\left(\tfrac{m+1}{2}\right)
\Gamma\left(\tfrac{n+2}{2}\right)
}
\right)^2\\
=&~
\big(\disft(\gamma_{m,n},\Sp^m_E,\Sp^n_E)\big)^4,
\end{align*}
via Lemma~\ref{lem:GW-eu-eq}, thus showing that the equatorial map is optimal. \qed

\begin{rmk}
By analyzing the equality conditions for the Cauchy-Schwarz inequality in the above proof, one obtains a proof of Theorem \ref{thm:equ-opti} 
without relying on the explicit computation of Lemma \ref{lem:GW-eu-eq}.
The equatorial coupling achieves equality in each inequality, so the computation of the bound of the distortion of $\gamma$ is also a computation of the distortion of the equatorial coupling and a proof that it is optimal.
The two inequalities occurring in the proof are:
\begin{enumerate}
    \item (Cauchy-Schwarz in Equation~\eqref{eq:Dmu-upbd}): holds with equality if, conditional on $\|y_A\| = t$, 
    $$
    x_k=C_1\dfrac{y_k}{\|y_A\|}
    $$
    for a $C_1$ constant possibly dependent on $\|y_A\|$ but not on  $k$.
    \item (Cauchy-Schwarz in Equation~\eqref{eq:Dmu2}): holds with equality if for almost every $t, t'$ there exists a constant $C_2$ such that for all $k \le n + 1$:
    \[
    \left(
        \int_{\Sp^m\times A_t}
            x_k^2 \ 
        \gamma_t(dx \times dy) 
    \right)^{1/2}=C_2\left(
        \int_{\Sp^m \times A_{t'}}
            x_k^2 \ 
        \gamma_{t'}(dx \times dy) 
    \right)^{1/2}.
    \]
\end{enumerate} 
The equatorial map satisfies both these conditions, so it is necessarily optimal.
\end{rmk}

\section{General Lower Bounds}\label{sec:gen-lb} 

In this section we will describe a number of different functions $\lbpq:\Gcalw \times \Gcalw \rightarrow \bbR_+$ which will become lower bounds for the $(p,q)$-Gromov-Wasserstein distance. Lower bounds for the $p$-Gromov-Wasserstein distance have been previously discussed \cite{memoli2011gromov,chowdhury2019gromov,memoli2022distance}. In \cite{memoli2011gromov} three lower bounds for the $p$-Gromov-Wasserstein distance called the First, Second, and Third Lower Bounds (denoted $\mathrm{FLB}_p, \mathrm{SLB}_p$ and $\mathrm{TLB}_p$ respectively) were constructed from certain invariants of metric measure spaces. Two of these lower bounds were based on the global and local distributions of distances. $\mathrm{SLB}_p$ was constructed using the Wasserstein distance on the real line between the global distributions of distances, and $\mathrm{TLB}_p$ was constructed using the local distribution of distances. 

In Section~\ref{sec:diam-lb} we consider a generalization of the lower bound based on the $p$-diameter of a metric measure space introduced in \cite{memoli2011gromov}. Note that we do not consider a generalization of $\mathrm{FLB}_p$ introduced in \cite{memoli2011gromov}, which is based on the $p$-eccentricity function associated to a metric measure space $X \in \mathcal{G}_w$ that assigns to each point in $X$ a value reflecting a notion of average distance to all other points in the space. In Sections~\ref{sec:second-lb} and \ref{sec:third-lb}, we construct $\tlbpq$ and $\slbpq$ using the local distributions of distances and global distributions of distances respectively that depend on the parameter $q$. For the choice $q=1$, our bounds $\tlbpq$ and $\slbpq$ agree with $\mathrm{TLB}_p$ and $\mathrm{SLB}_p$. Finally, in Proposition~\ref{thm:hierarchy} we give a hierarchy of our lower bounds for the setting of the $(p,q)$-Gromov-Wasserstein distance.

\subsection{Invariants.}\label{sec:invars}
We first recall some invariants of metric measure spaces which we will utilize in our construction of lower bounds for the $(p,q)$-Gromov-Wasserstein distance.

\begin{defn}[Global distribution of distances of a metric measure space] Let $(X,d_X,\mu_X) \in \Gcalw$. The \emph{global distribution of distances} associated to $X$ is the function,
\[H_X:[0,\diam(X)] \to [0,1] 
\,\,\text{given by}\,\,
t \mapsto \mu_X \otimes \mu_X(\{(x,x')\in X\times X| d_X(x,x') \leq t\}).
\]

\end{defn}

\begin{defn}[Local distribution of distances of a metric measure space] Let $(X,d_X,\mu_X) \in \Gcalw$. The \emph{local distribution of distances} associated to $X$ is the function,
\[h_X: X\times [0,\diam(X)] \to [0,1] 
\,\,\text{given by}\,\,
(x,t) \mapsto \mu_X(\{x'\in X| d_X(x,x') \leq t\}).
\]
\end{defn}

\begin{rmk}\label{rem:p-diameter and moment} It is described in \cite[Remark 5.4]{memoli2011gromov} that all $p$-diameters of $(X,d_X,\mu_X) \in \Gcalw$ can be recovered from its global distribution of distances as follows:
\[
\diamp(X) = m_p(dH_X)=\bigg(\int_0^\infty t^p dH_X(dt) \bigg)^{1/p}.
\]
\end{rmk}

The local distribution of distances generalizes the global one and we can relate the global and local distributions of distances by noting that 
\[h_X(x,t) = 
\mu_X\big(\overline{B_X(x,t)}\big)
\] 
where $\overline{B_X(x,t)}$ is the closed ball centered at $x$ with radius $t$. Then we have from \cite[Remark 5.8]{memoli2011gromov}, that
\begin{equation}\label{eq:rel-gdd-1}
 H_X(t) = \int_X \int_{\overline{{B_X(x,t)}}} \ \mu_X(dx') \ \mu_X(dx) = \int_X h_X(x,t) \ \mu_X(dx)
\quad
\text{for }
t\in [0,\diam(X)].   
\end{equation}

\begin{ex}\label{ex:gdd-geo}
    The global distance distribution function for $\Sp^n_G$ (for $n\ge 1$) is:
\[
    H_{\Sp^n_G}(t) = h_{\Sp^n_G}(t) = \frac{1}{\sqrt{\pi}} \frac{\Gamma(\frac{n+1}{2})}{\Gamma(\frac{n}{2})} \int_0^t \sin^{n-1}(s) \ ds 
\]

\noindent where $t \in [0,\pi]$. This follows by the fact that $\mu_n$ is the uniform measure on $\Sp^n$ and basic spherical geometry; see e.g. \cite[Chapter 1]{atkinson2012spherical}.  
\end{ex}

\begin{ex}\label{ex:gdd-eu}
From the previous example, and the fact that $\Vert x-x'\Vert=2\sin\left(\frac{d_n(x,x')}{2}\right)$ for all $x,x'\in\Sp^n\subset \mathbb{R}^{n+1}$, we obtain that the global distance distribution function for $\Sp^n_E$  is:
    \[ H_{\Sp_{E}^n}(t) = \frac{1}{\sqrt{\pi}} \frac{\Gamma(\frac{n+1}{2})}{\Gamma(\frac{n}{2})}\int_0^{2\arcsin\left(t/2\right)}
        \sin^{n-1}(s) \ ds = \frac{2^n}{\sqrt{\pi}} \frac{\Gamma(\frac{n+1}{2})}{\Gamma(\frac{n}{2})}\int_0^{t/2} s^{n-1}\left(\sqrt{1-s^2}\right)^{n-2} \ ds
    \]   

    \noindent where $t \in [0,2]$.
\end{ex}

\subsection{Diameter Lower Bound}\label{sec:diam-lb}

\begin{defn}[$(p,q)$-Diameter Lower Bound]\label{def:dlbpq}  The $(p,q)$-Diameter Lower Bound for $X,Y \in \Gcalw$, denoted $\dlbpq$, for $p,q\in [1,\infty]$ is:
\[
\dlbpq(X,Y):=\Lambda_q(\diamp(X),\diamp(Y))\stackrel{(*)}{=}\bigg|\left(\diamp(X)\right)^q-\left(\diamp(Y)\right)^q\bigg|^{1/q}
\]
where $(*)$ holds when $q\in[1,\infty)$.
\end{defn}

\begin{rmk} In general, by the triangle inequality for $\dgwpq$ (cf. Theorem \ref{thm:props}) and by Example \ref{ex:dgwpq-onepoint}, we always have $\dgwpq(X,Y) \geq \frac{1}{2}|\diamp(X)-\diamp(Y)|$ for all $p,q\in [1,\infty]$. The lower bound $\dlbpq$ depends on both $p$ and $q$, and provides a better lower bound for $\dgwpq$ since for all $q\geq 1$, it can be shown that $\dlbpq(X,Y) \geq |\diamp(X)-\diamp(Y)|.$
\end{rmk}

\subsection{Second Lower Bound}\label{sec:second-lb}

Here we consider a general lower bound for the $(p,q)$-Gromov-Wasserstein distance between two metric measure spaces $X$ and $Y$ based on the distribution of distances. For the case $q=1$, it is known (see \cite[Proposition 6.2]{memoli2011gromov}) that the $p$-Gromov-Wasserstein distance between $X$ and $Y$ is bounded below by the Wasserstein distance between the global distribution of distances of $X$ and $Y$ on the real line. We describe a function which we call $(p,q)$-Second Lower Bound, denoted $\mathrm{SLB}_{p,q}$, which is an analogue of this result for $q\geq 1$. We define $\mathrm{SLB}_{p,q}(X,Y)$ for $X, Y \in \mathcal{G}_w$ as follows:

\begin{defn}[$(p,q)$-Second Lower Bound] \label{def:slbpq} 
The $(p,q)$-Second Lower Bound for $X,Y \in \mathcal{G}_w$, denoted $\text{SLB}_{p,q}(X,Y)$, for $p,q\in [1,\infty]$, is:
\[\mathrm{SLB}_{p,q}(X,Y):=
 d^{(\bbR_+,\Lambda_q)}_{\operatorname{W}p}(dH_X,dH_Y).
\]
\end{defn}
For $X \in \mathcal{G}_w$, $dH_X$ is the unique measure on $\bbR_+$ defined by $dH_X([a,b]) := H_X(b)-H_X(a)$ for all $a \leq b$. It can be checked that $dH_X= (d_X)_\#\mu_X\otimes\mu_X$.

Note that Remark~\ref{rmk:Wpq-closed-form} relates the $d^{(\bbR_+,\Lambda_q)}_{\operatorname{W}p}$ distance to the usual Wasserstein distance between suitably transformed measures. The closed form solution of $d^{(\bbR_+,\Lambda_q)}_{\operatorname{W}p}$ ensures that the SLB can be computed very efficiently.

\subsection{Third Lower Bound}\label{sec:third-lb}

In analogy with the third lower bound from \cite{memoli2011gromov}, we consider the local distribution of distances and construct what we call the $(p,q)$-Third Lower Bound, denoted $\tlbpq$. For $X \in \Gcalw$, recall that to the local distribution of distances of $X$, $h_X(x,\cdot)$, we associate the unique measure on $\bbR_+$, $dh_X(x)$, where $dh_X(x)=(d_X(x,\cdot))_{\#}\mu_X$. 

\begin{defn}[$(p,q)$-Third Lower Bound] \label{def:tlbpq} The $(p,q)$-Third Lower Bound, denoted $\tlbpq$, for $p,q\in [1,\infty]$ and $X,Y \in \Gcalw$ is:
\[
\tlbpq(X,Y):=\inf_{\gamma \in \Mcal(\mu_X,\mu_Y)} \bigg(\int_{X \times Y} \big(d^{(\bbR_+,\Lambda_q)}_{\operatorname{W}p}(dh_X(x),dh_Y(y))\big)^p \gamma(dx \times dy)\bigg)^{1/p}
\]
for $1\le p <\infty$, and
\[
\mathrm{TLB}_{\infty,q}(X,Y):=\inf_{\gamma \in \Mcal(\mu_X,\mu_Y)}\sup_{(x,y)\in\supp[\gamma]} d^{(\bbR_+,\Lambda_q)}_{\operatorname{W}\infty}(dh_X(x),dh_Y(y)).
\]
\end{defn}

Note that the closed form solution of $d^{(\bbR_+,\Lambda_q)}_{\operatorname{W}p}$ from Remark~\ref{rmk:Wpq-closed-form} allows one to efficiently compute the TLB.

\subsection{The complete Hierarchy of Lower Bounds} 
Hierarchies of lower bounds have been considered in \cite{memoli2007use, chowdhury2019gromov, memoli2022distance}. A key aspect of \cite{memoli2011gromov} was providing a hierarchy between the aforementioned lower bounds, $\mathrm{FLB}_p, \mathrm{SLB}_p$ and $\mathrm{TLB}_p$, that showed $\dgwp \geq \mathrm{TLB}_p \geq \mathrm{FLB}_p$ and $\dgwp \geq \mathrm{SLB}_p$. 
In \cite{chowdhury2019gromov} they considered lower bounds in the setting of Gromov-Wasserstein between networks. In particular, they considered the associated push forwards of the First, Second, and Third Lower Bounds from \cite{memoli2011gromov} into the real line denoted, $\bbR$-$\mathrm{FLB}_p$, $\bbR$-$\mathrm{SLB}_p$ and $\bbR$-$\mathrm{TLB}_p$ and showed that $\mathrm{FLB}_p \geq \bbR$-$\mathrm{FLB}_p$, $\mathrm{SLB}_p \geq \bbR$-$\mathrm{SLB}_p$, and $\mathrm{TLB}_p \geq \bbR$-$\mathrm{TLB}_p$.

We note here that \cite{chowdhury2019gromov} and \cite{memoli2011gromov} did not provide a complete hierarchy between their lower bounds, where by incomplete we mean only partial relationships between some of the bounds were given. Proposition 2.8 of \cite{memoli2022distance} bridged this gap by giving a hierarchy of lower bounds that related the Third and Second Lower Bounds of \cite{memoli2011gromov} to one another and thus strengthened the original hierarchy results from \cite{memoli2011gromov} by showing $\dgwp \geq \mathrm{TLB}_p \geq \mathrm{SLB}_p \geq \mathrm{FLB}_p.$ Proposition \ref{thm:hierarchy} below generalizes \cite[Proposition 2.8]{memoli2022distance} to the setting of the $(p,q)$-Gromov-Wasserstein distance.

\begin{prop}
    \label{thm:hierarchy}
For all $X,Y\in\mathcal{G}_w$ and all $p,q\in [1,\infty]$ we have
$$
2\,\dgwpq(X,Y)\geq \tlbpq(X,Y)\geq \slbpq(X,Y)
\geq \dlbpqmin(X,Y)
$$
where $p\wedge q$ denotes $\min\{p,q\}$.
\end{prop}

We defer the proof of this proposition to  Section~\ref{subsec:thmhierarchy}. See also Example \ref{ex:dlb-ctrex}.

\subsection{Lower Bounds for Spheres.}\label{sec:lbs-spheres}
In this section we consider the hierarchy of lower bounds for the Gromov-Wasserstein distance between spheres equipped with the geodesic distance and Euclidean distance. Let $\Sp^m_\bullet$ represent the $m$-sphere equipped with the geodesic or Euclidean metric (See Example~\ref{ex:spheres-def}).  

\subsubsection{Diameter Lower Bound for Spheres.} \label{subsec:dlbsp}
Recall from Remark~\ref{rem:p-diameter and moment} that the $p$-diameter is related to the $p$-moment of the global distance distribution as follows:
\begin{align*}
    \diam_p(\Sp_\bullet^m) = m_{p}(dH_{\Sp_\bullet^m})
    = \bigg(\int_0^1(H^{-1}_{\Sp_\bullet^m}(u))^{p}du\bigg)^{1/p},
\end{align*}
where $H_{\Sp_\bullet^m}^{-1}$ is the generalized inverse of $H_{\Sp_\bullet^m}$ (see Example~\ref{ex:dWonR}).
The diameter lower bound is defined as \[\dlbpq(\Sp_\bullet^m,\Sp_\bullet^n) =\Lambda_q(\diam_p(\Sp_\bullet^m),\diam_p(\Sp_\bullet^n)).\]

\subsubsection{Second and Third Lower Bound for Spheres}\label{subsec:slbsp}
The local distribution of distances is equal to the  global distribution of distances in the case of spheres. This implies that the third lower bound is equal to the second lower bound, that is:
\begin{align*} \tlbpq(\Sp_\bullet^m,\Sp_\bullet^n) 
& = \inf_{\gamma \in \Mcal(\mu_{m},\mu_n)} \bigg(\int_{\Sp_\bullet^m \times \Sp_\bullet^n} \left(d^{(\bbR_+,\Lambda_q)}_{\operatorname{W}p}(dh_{\Sp_\bullet^m}(x),dh_{\Sp_\bullet^n}(y))\right)^{p} \gamma(dx \times dy)\bigg)^{1/p} \\
& = \bigg(d^{(\bbR_+,\Lambda_q)}_{\operatorname{W}p}(dH_{\Sp_\bullet^m},dH_{\Sp_\bullet^n})  \bigg)^{1/p}\\
& =\slbpq(\Sp_\bullet^m,\Sp_\bullet^n).
\end{align*}
It follows that we have the following hierarchy of lower bounds of $\dgwpq$ for spheres when $p\geq q$: 
\[2\dgwpq(\Sp_\bullet^m,\Sp_\bullet^n) \ge \tlbpq(\Sp_\bullet^m,\Sp_\bullet^n) = \slbpq(\Sp_\bullet^m,\Sp_\bullet^n) \ge \dlbpq(\Sp_\bullet^m,\Sp_\bullet^n). \]

We can compute the second lower bound (equivalently the third lower bound) between $\Sp_\bullet^m$ and $\Spn_\bullet$ as follows: 

\[  \mathrm{SLB}_{p,q}(\Sp_\bullet^m,\Sp_\bullet^n) = d^{(\bbR_+,\Lambda_q)}_{\operatorname{W}p}(dH_{\Sp_\bullet^m},dH_{\Sp_\bullet^n})   = \bigg(\int_0^1 \big|(H_{\Sp_\bullet^m}^{-1}(u))^q - (H_{\Sp_\bullet^n}^{-1}(u))^q\big|^{p/q} \ du \bigg)^{1/p}
\]
where the second equality holds by Remark \ref{rmk:clfrmdWR+Lbdaq}.

\subsubsection{Lower Bounds for $\dgwft$ between Spheres}\label{subsec:lbspex}

We now provide example computations for lower bounds of the $(4,2)$-Gromov-Wasserstein distance between spheres of dimension $0,1$ and $2$, when equipped with the geodesic distance. We make use of the formulas from Section~\ref{subsec:dlbsp} and Section~\ref{subsec:slbsp}.  To facilitate presentation, we defer the detailed calculations to Appendix~\ref{app:calc}  and only present the final values in Table~\ref{table:GW-geo} and Table~\ref{table:GW euc}.

\begin{rmk}(Hierarchy of lower bounds and $\dgwft$ for $\Spg^0,\Spg^1,\Spg^2$)
    It follows from Example~\ref{ex:dlb-geo} and Example~\ref{ex:slb-geo} that the hierarchy of lower bounds for the $(4,2)$-Gromov-Wasserstein distance are given in Table~\ref{table:GW-geo}. 
\begin{table}[ht]
\begin{tabular}{|cl|c|cl|c|}
\hline
\multicolumn{2}{|c|}{\multirow{2}{*}{Spheres}} & \multirow{2}{*}{$\frac{1}{2}\operatorname{DLB}_{4,2}$}   & \multicolumn{2}{c|}{\multirow{2}{*}{
$\frac{1}{2}\operatorname{SLB}_{4,2}
=\frac{1}{2}\operatorname{TLB}_{4,2}$
}}   & \multirow{2}{*}{$d_{\operatorname{GW}4,2}$}    \\
\multicolumn{2}{|c|}{}                         &                        & \multicolumn{2}{c|}{}                       &                        \\ \hline
\multicolumn{2}{|c|}{\multirow{2}{*}{$\Spg^0\,\mbox{vs.}\,\Spg^1$}}        & \multirow{2}{*}{0.801} & \multicolumn{2}{c|}{\multirow{2}{*}{0.918}} & \multirow{2}{*}{$\le 1.050$} \\
\multicolumn{2}{|c|}{}                         &                        & \multicolumn{2}{c|}{}                       &                        \\ \hline
\multicolumn{2}{|c|}{\multirow{2}{*}{$\Spg^1\,\mbox{vs.}\,\Spg^2$}}        & \multirow{2}{*}{0.431} & \multicolumn{2}{c|}{\multirow{2}{*}{0.461}} & \multirow{2}{*}{$\le 0.734$} \\
\multicolumn{2}{|c|}{}                         &                        & \multicolumn{2}{c|}{}                       &                        \\ \hline
\end{tabular}
\medskip
\caption{This table demonstrates the lower bound hierarchy from Proposition~\ref{thm:hierarchy} in the case of the (4,2)-Gromov-Wasserstein distance between spheres equipped with the geodesic distance. The values of the lower bounds are computed in Section~\ref{sec:sphere-lb-geo}, while the upper bounds of the (4,2)-Gromov-Wasserstein distances are  computed using the equatorial coupling (see Claim~\ref{claim:eq-coupling} and Examples~\ref{ex:dis42g01} and \ref{ex:dis42g12}).} 
\label{table:GW-geo}
\end{table}
\end{rmk}

\begin{rmk}(Hierarchy of lower bounds and $\dgwft$ for $\Spe^0,\Spe^1,\Spe^2$)
Similarly to the case of $\Spg$, calculations for bounding $\dgwft$ can be done for spheres equipped with the Euclidean distance using the global distribution of distances for $\Spe^m$ given in Example~\ref{ex:gdd-eu}.

     In the next section, we give an exact determination of the $(4,2)$-Gromov-Wasserstein distance between spheres equipped with the Euclidean distance. So, the values of the  lower bounds can be compared against the exact value of $(4,2)$-Gromov-Wasserstein distance between Euclidean spheres as shown in Table~\ref{table:GW euc}. 

     \smallskip
     In Table~\ref{table:GW euc} note how the exact values significantly exceed those provided by the lower bounds. 
\begin{table}[ht]
\begin{tabular}{|cl|c|cl|c|}
\hline
\multicolumn{2}{|c|}{\multirow{2}{*}{Spheres}} & \multirow{2}{*}{$\frac{1}{2}\operatorname{DLB}_{4,2}$}   & \multicolumn{2}{c|}{\multirow{2}{*}{
$\frac{1}{2}\operatorname{SLB}_{4,2}
=\frac{1}{2}\operatorname{TLB}_{4,2}$
}}   & \multirow{2}{*}{$d_{\operatorname{GW}4,2}$}    \\
\multicolumn{2}{|c|}{}                         &                        & \multicolumn{2}{c|}{}                       &                        \\ \hline 
\multicolumn{2}{|c|}{\multirow{2}{*}{$\Spe^0\,\mbox{vs.}\,\Spe^1$}}        & \multirow{2}{*}{0.308} & \multicolumn{2}{c|}{\multirow{2}{*}{0.488}} & \multirow{2}{*}{0.644} \\
\multicolumn{2}{|c|}{}                         &                        & \multicolumn{2}{c|}{}                       &                        \\ \hline
\multicolumn{2}{|c|}{\multirow{2}{*}{$\Spe^1\,\mbox{vs.}\,\Spe^2$}}        & \multirow{2}{*}{0.187} & \multicolumn{2}{c|}{\multirow{2}{*}{0.276}} & \multirow{2}{*}{0.482} \\
\multicolumn{2}{|c|}{}                         &                        & \multicolumn{2}{c|}{}                       &                        \\ \hline
\end{tabular}
\medskip
\caption{This table demonstrates the lower bound hierarchy from Proposition~\ref{thm:hierarchy} in the case of the (4,2)-Gromov-Wasserstein distance between spheres with the Euclidean metric. The values of the lower bounds are computed using the elements in Section~\ref{sec:sphere-lb-euc}, while the exact values of the (4,2)-Gromov-Wasserstein distances are  computed using Theorem~\ref{thm:equ-opti} (see Remark \ref{rmk:ud}).}
\label{table:GW euc}
\end{table}
\end{rmk}

\section{Experimental Illustration}\label{sec:exp}

The explicit computations of the $(4, 2)$-Gromov-Wasserstein distance between
Euclidean spheres provide a helpful tool for benchmarking common optimal
transport solvers and packages. The goal of this section is to benchmark various sampling methods and the
number of samples required to obtain accurate estimates of the
Gromov-Wasserstein distance while also ascertaining the accuracy of the
various solvers when taking into account Theorem \ref{thm:equ-opti}.

The authors are aware of two 
Python implementations of optimal transport GW solvers: Python Optimal Transport
(POT) \cite{flamary2021pot} and  Optimal Transport Tools (OTT)  \cite{cuturi2022optimal}. 
These packages implement two of the common methods for
computing the Gromov-Wasserstein distance:
Conditional Gradient Descent (implemented by POT)  \cite{vayer2018optimal} and Sinkhorn Projections with entropic regularization
(implemengted by both OTT and POT) \cite{peyre2016gromov,peyre2019computational}.

In our experiments we used the Conditional Gradient Descent (CGD) solver from POT and the Sinkhorn solver from OTT (with regularizarion parameter $0.01$).\footnote{The reason we used the Sinkhorn solver from OTT instead of the one from POT is that the former
    appears to be faster for ``smaller scale problems'' according to their
     documentation: 
    \url{https://ott-jax.readthedocs.io/en/latest/tutorials/OTT\_\%26\_POT.html}.}  All experiments and their results are
available in the Github repository \cite{dan-github}.

We run two sorts of experiments. First, we examine how the number of
samples relates to the choice of: the 
solver, the subsampling method, and the weights. Second, we fix the number of samples and vary the dimensionality of the
spheres. The former are described next whereas the latter results are presented in Appendix \ref{app:exps}.

\subsection*{Experiment with varying number of sample points}

In this experiment, we fix the dimensions of both spheres to the values $(m,n)=(1,2),(1,3)$ and $(2,3)$ and vary the number
of samples we draw from each one. For each number of 
sample points between $10$ and $200$ (in increments of $10$), we run $20$ trials
of each combination of sampling method, weight procedure, and GW solver (see below). The maximal size $200$ was chosen so as to keep a reasonable computational burden. 

See Figures~\ref{fig:varied-points-2-1},
\ref{fig:varied-points-3-1}, and \ref{fig:varied-points-3-2} where the label POT is used to indicate the CGD solver and OTT is used to indicate the Sinkhorn solver. The plotted
lines are the mean values estimated from the $20$ trials, while the shaded areas
represent the central $80\%$ of samples. Dotted lines correspond to the ``true" values established by Theorem \ref{thm:equ-opti}:
\begin{itemize}
\item[(1,2):] $\dgwft(\Sp^{1}_E,\Sp^2_E)
=\dfrac{1}{\sqrt{2}}
\left(
\dfrac{5}{6}
-\dfrac{\pi^2}{16}
\right)^{1/4}\approx \mathbf{0.482}$; see Remark \ref{rmk:ud}.
\item[(1,3):] $\dgwft(\Sp^{1}_E,\Sp^3_E)=\left(\dfrac{11}{144}\right)^{1/4}\approx \mathbf{0.526}$; see Remark \ref{rmk:neqmp2}.
\item[(2,3):] $\dgwft(\Sp^{2}_E,\Sp^3_E)
=
\dfrac{1}{\sqrt{2}}
\left(
\dfrac{7}{12}
-\dfrac{8}{27}
\left(\dfrac{16}{\pi^2}\right)
\right)^{1/4}
\approx
\mathbf{0.400}$; see Remark \ref{rmk:ud}.
\end{itemize}

We implement two sampling strategies:

\begin{enumerate}
\item \textbf{Random}: We draw the desired number of uniform samples via the well known method from \cite{muller1959note}, which consists of normalizing standard Gaussian samples.

\item \textbf{Farthest Point Sampling (FPS)}: We first sample $10^6$ points from the sphere uniformly at random via \cite{muller1959note} and we then select the desired number of subsamples via the FPS method \cite{gonzalez1985clustering,eldar1997farthest}.\footnote{In a nutshell,  given a finite metric space $(X,d_X)$ and a poitive integer $N\geq 2$, the FPS method  selects the first as a random point from $X$. The second point will be any point at maximal distance fromt the first selected point. The third point will be any point at maximal distance form the first two points and so on.} 
\end{enumerate}

We implement two different procedures for assigning weights to the samples. Given a finite sample  $P\subset \Sp^m$:

\begin{enumerate}
    \item \textbf{Voronoi}: This consists on assigning to each point $p\in P$  an estimate of the total mass of the Voronoi cell on the sphere corresponding to $p$. To estimate this, we construct a set $S$ consisting of $10^6$ uniformly points on the sphere and assign to $p$ the proportion of points from $S$ the that are closer to it than to any other point
        in $P$.
    
    \item \textbf{Uniform}: We simply give uniform weights $|P|^{-1}$ to all points $p\in P$. 
\end{enumerate}

\begin{figure}[ht]
    \centering
    \begin{tabular}{cc}
        \includegraphics[width=0.45\linewidth]{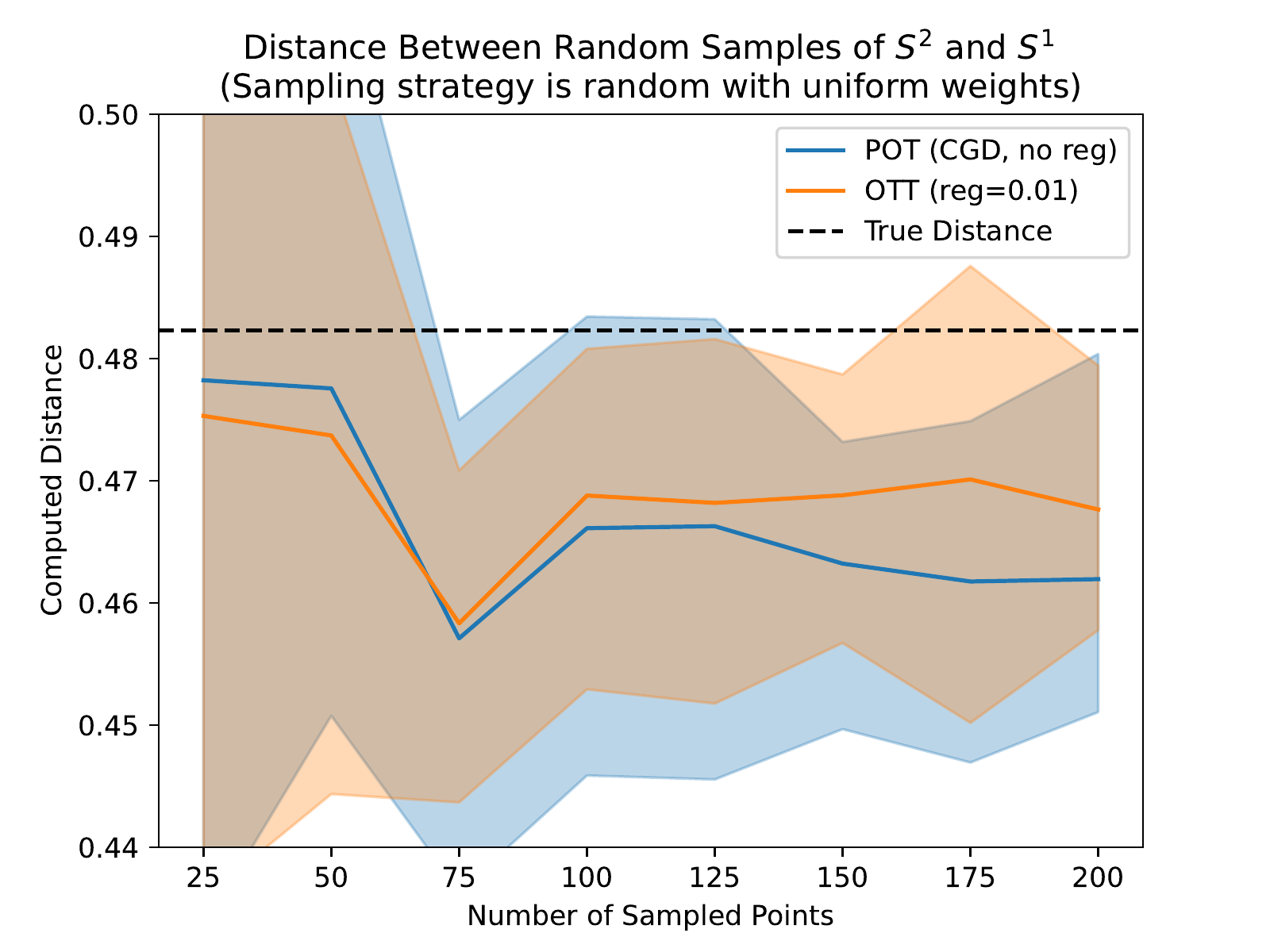} &
        \includegraphics[width=0.45\linewidth]{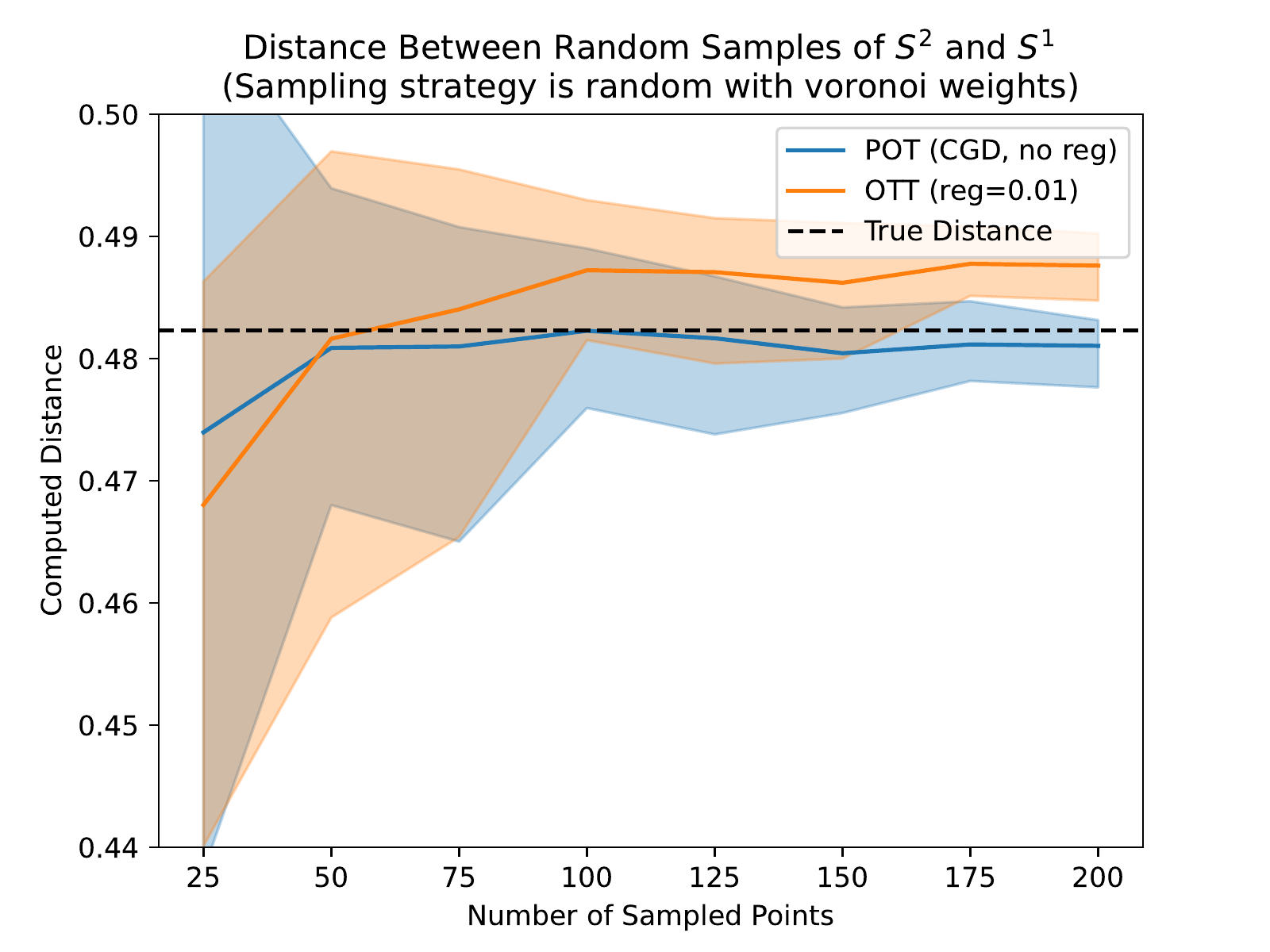} \\
        \includegraphics[width=0.45\linewidth]{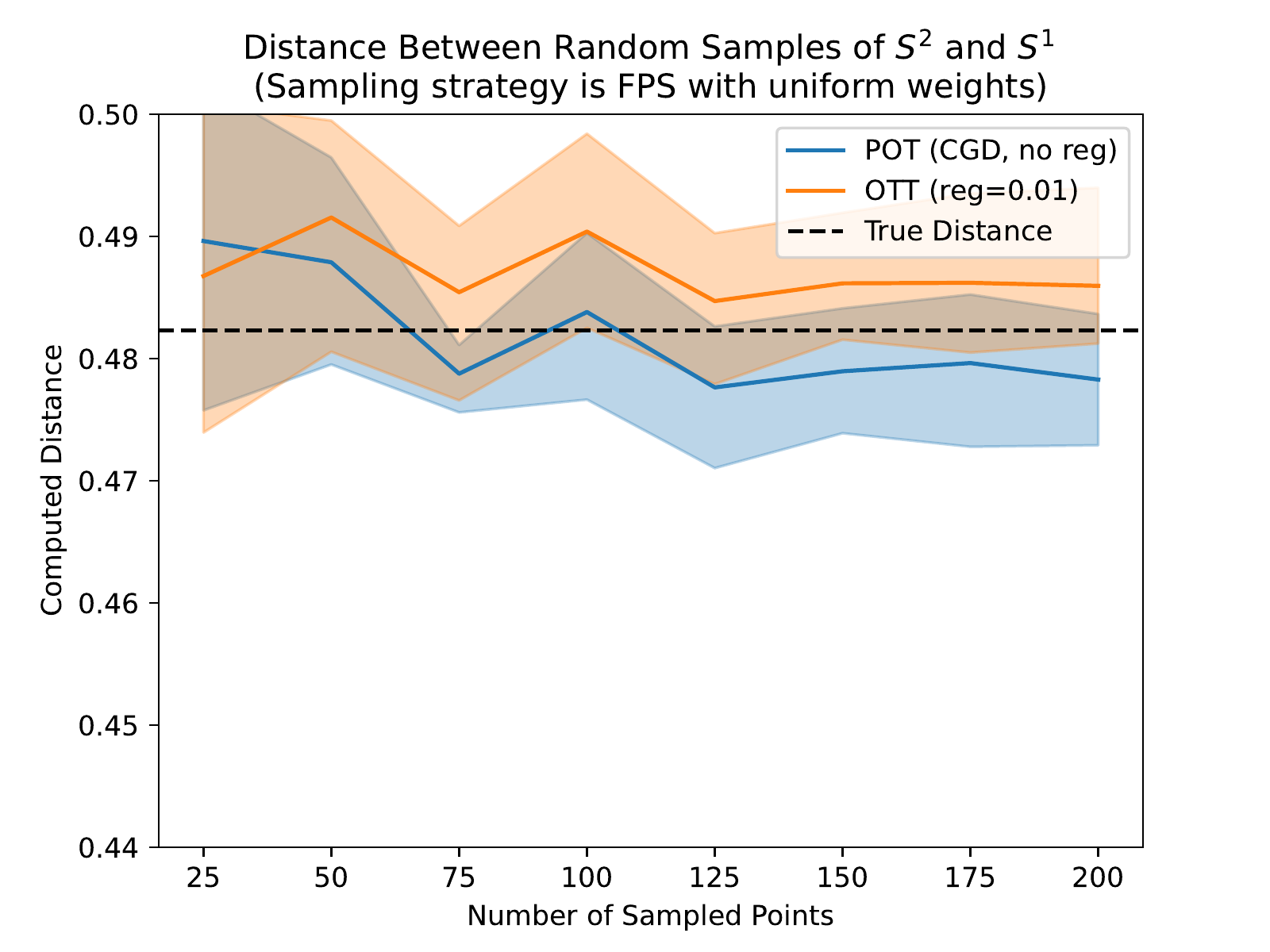} &
        \includegraphics[width=0.45\linewidth]{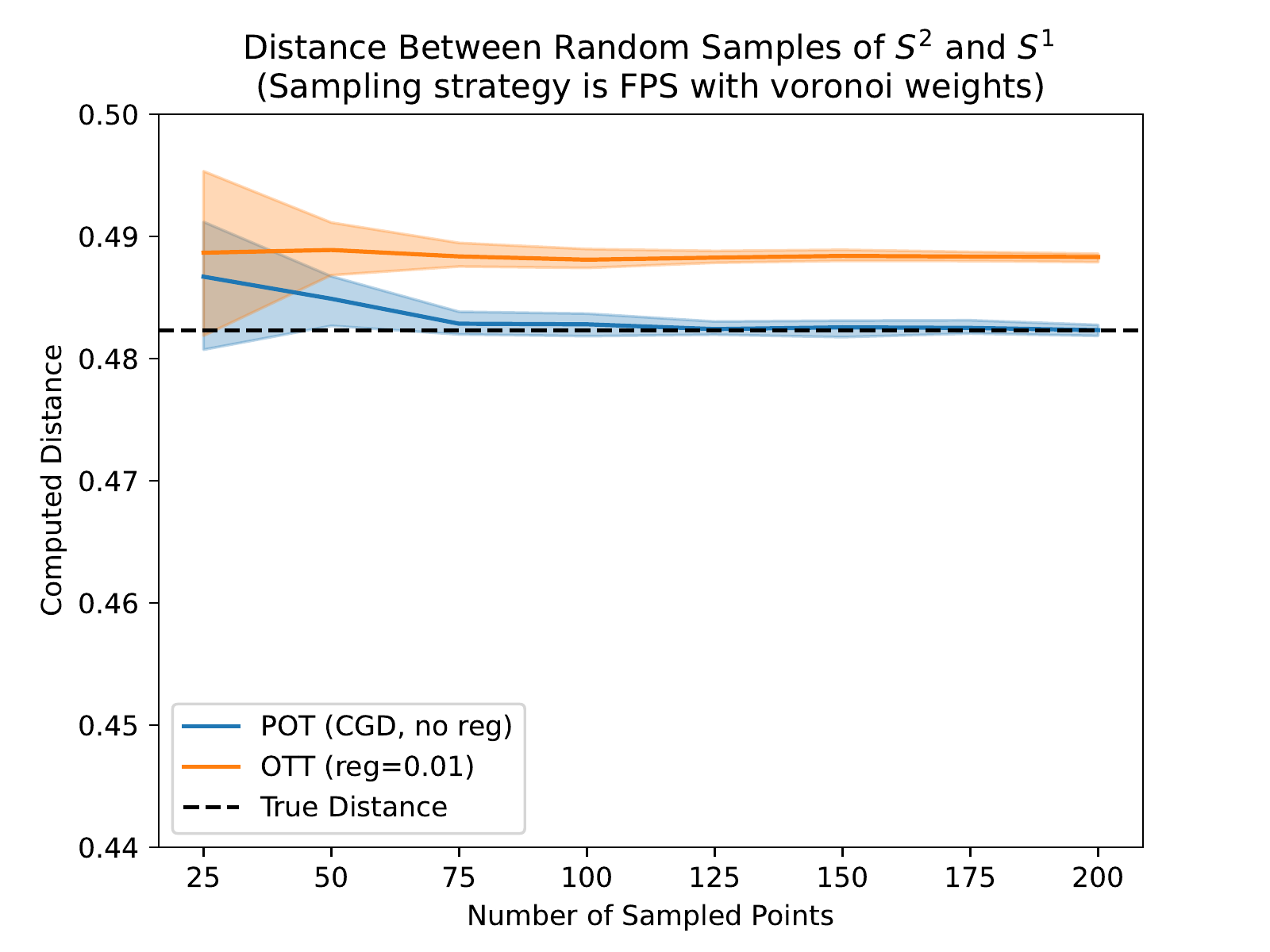} 
    \end{tabular}
   \caption{Estimating the Gromov-Wasserstein distance between $\mathbb{S}^2_E$ and $\mathbb{S}^1_E$. }\label{fig:varied-points-2-1}
\end{figure}

\begin{figure}[ht]
    \centering
    \begin{tabular}{cc}
        \includegraphics[width=0.45\linewidth]{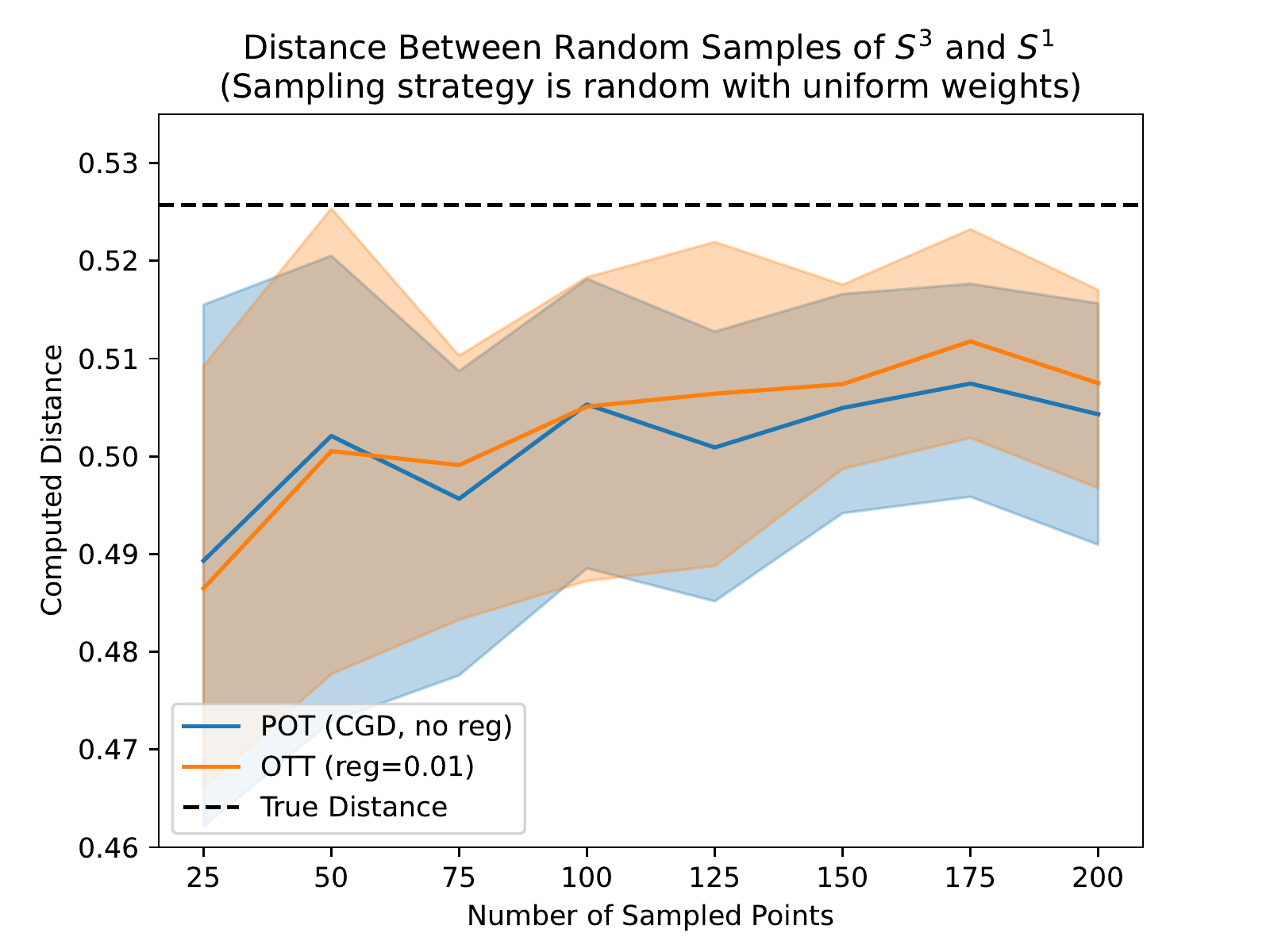} &
        \includegraphics[width=0.45\linewidth]{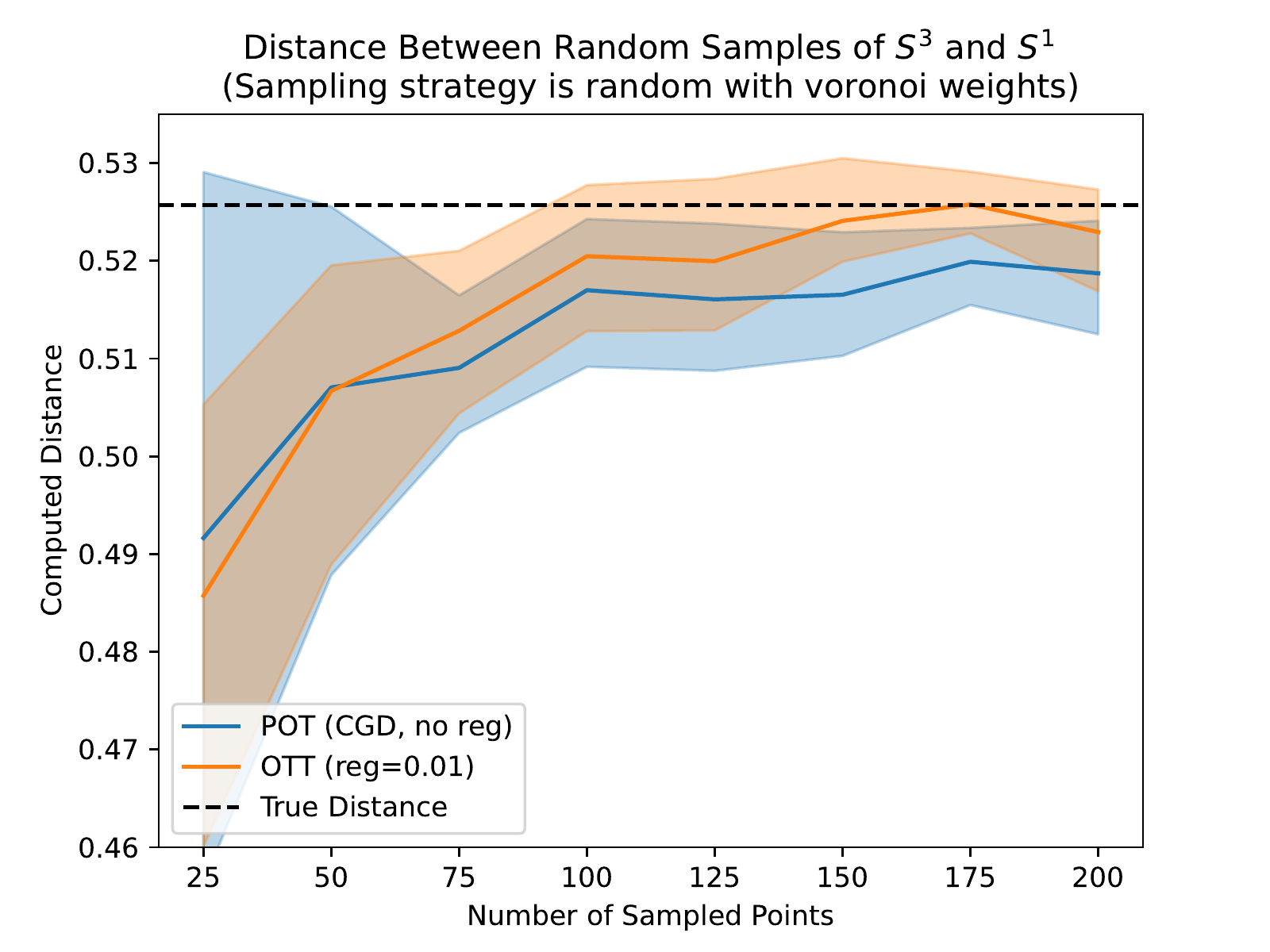} \\
        \includegraphics[width=0.45\linewidth]{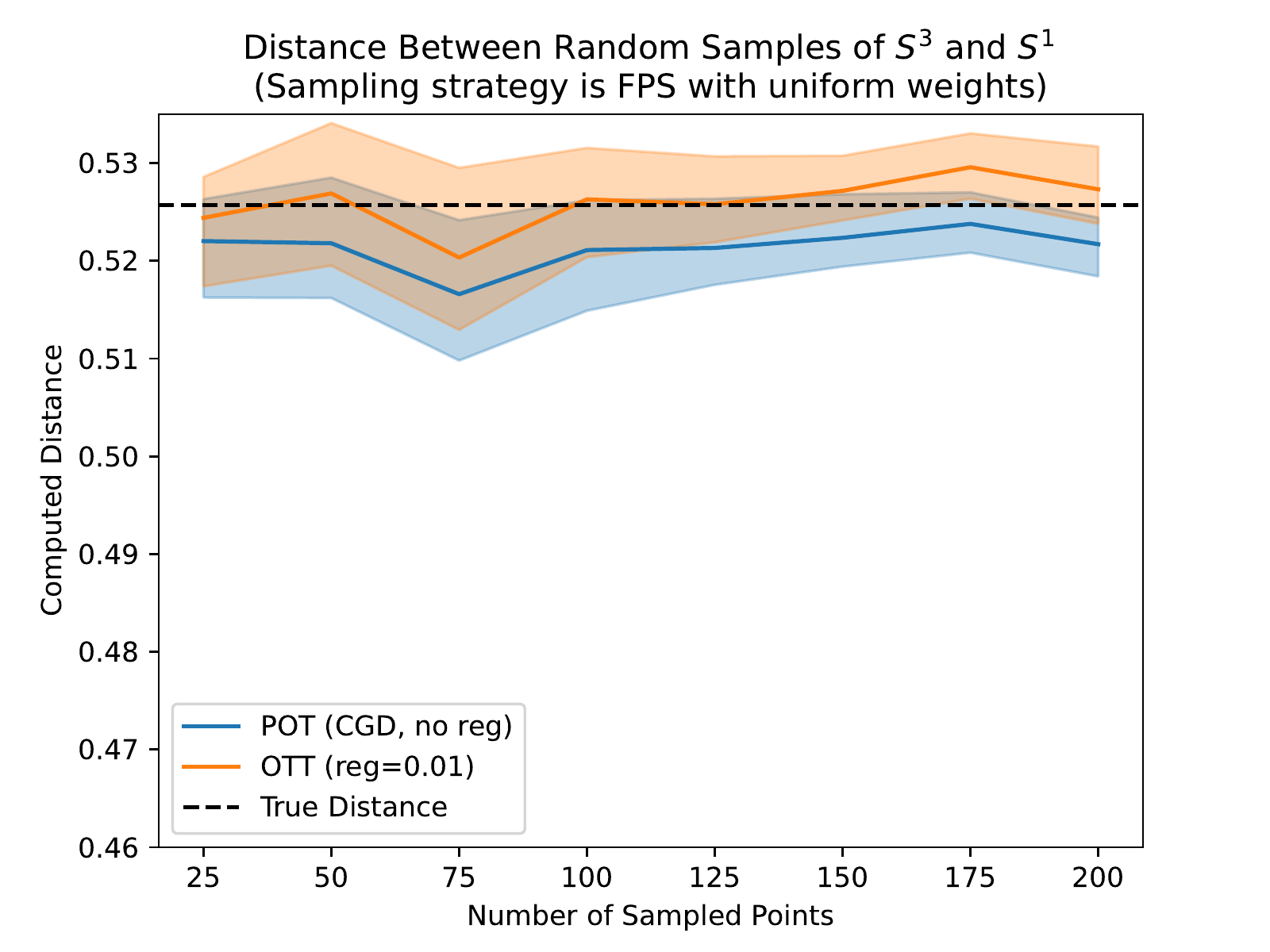} &
        \includegraphics[width=0.45\linewidth]{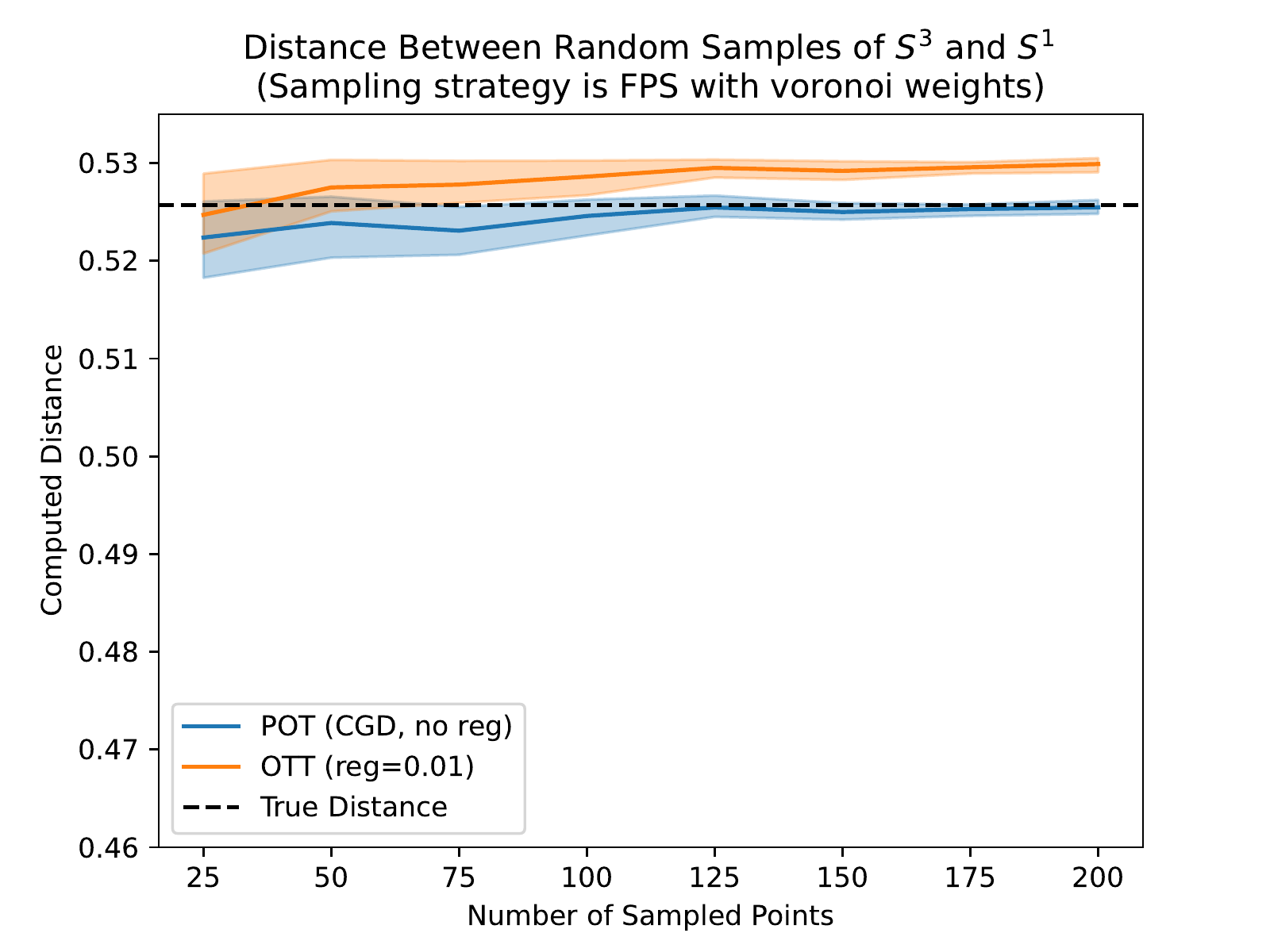} 
    \end{tabular}
    \caption{Estimating the Gromov-Wasserstein distance between $\mathbb{S}^3_E$ and $\mathbb{S}^1_E$. }\label{fig:varied-points-3-1}
\end{figure}

\begin{figure}[ht]
    \centering
    \begin{tabular}{cc}
        \includegraphics[width=0.45\linewidth]{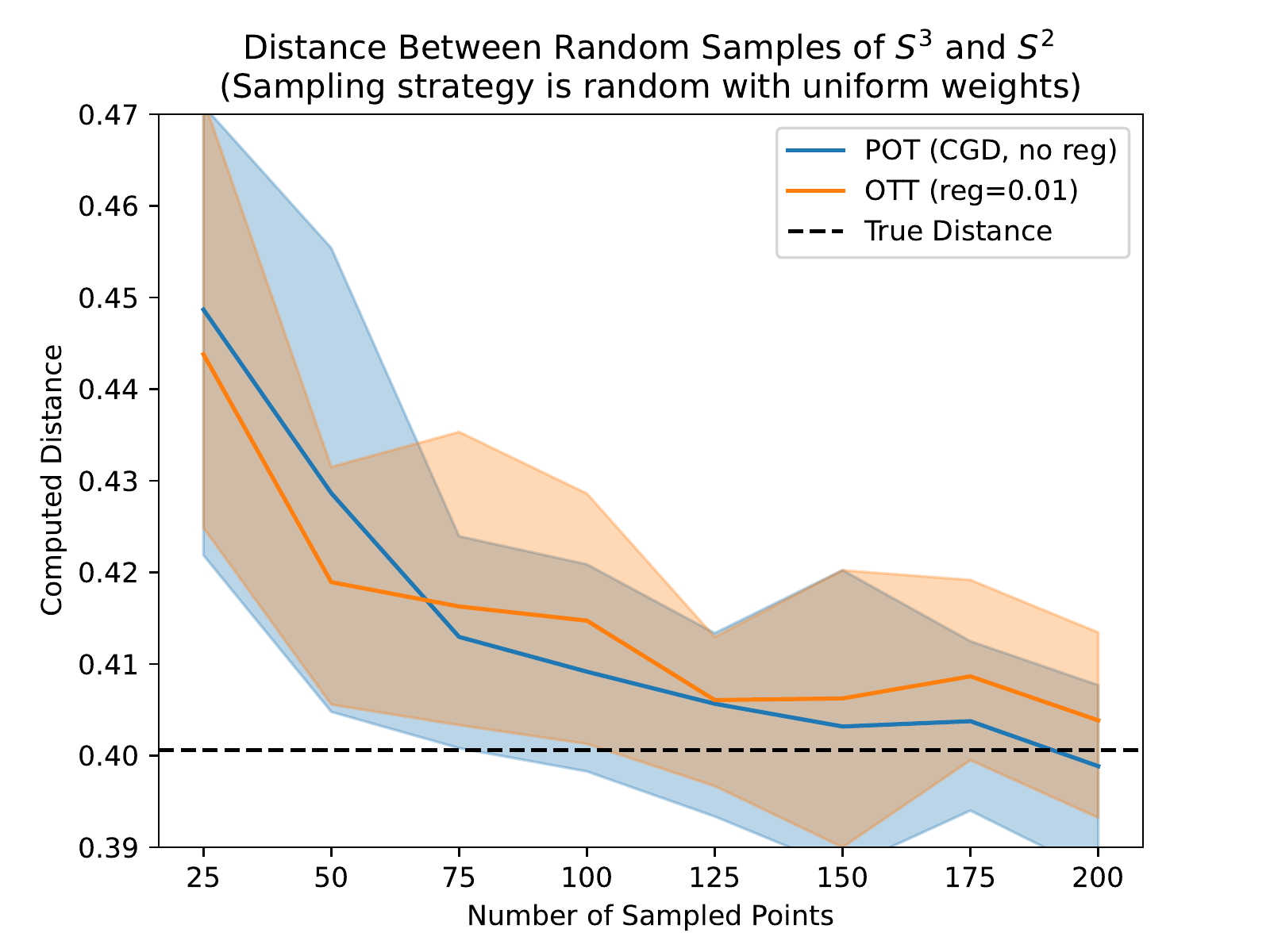} &
        \includegraphics[width=0.45\linewidth]{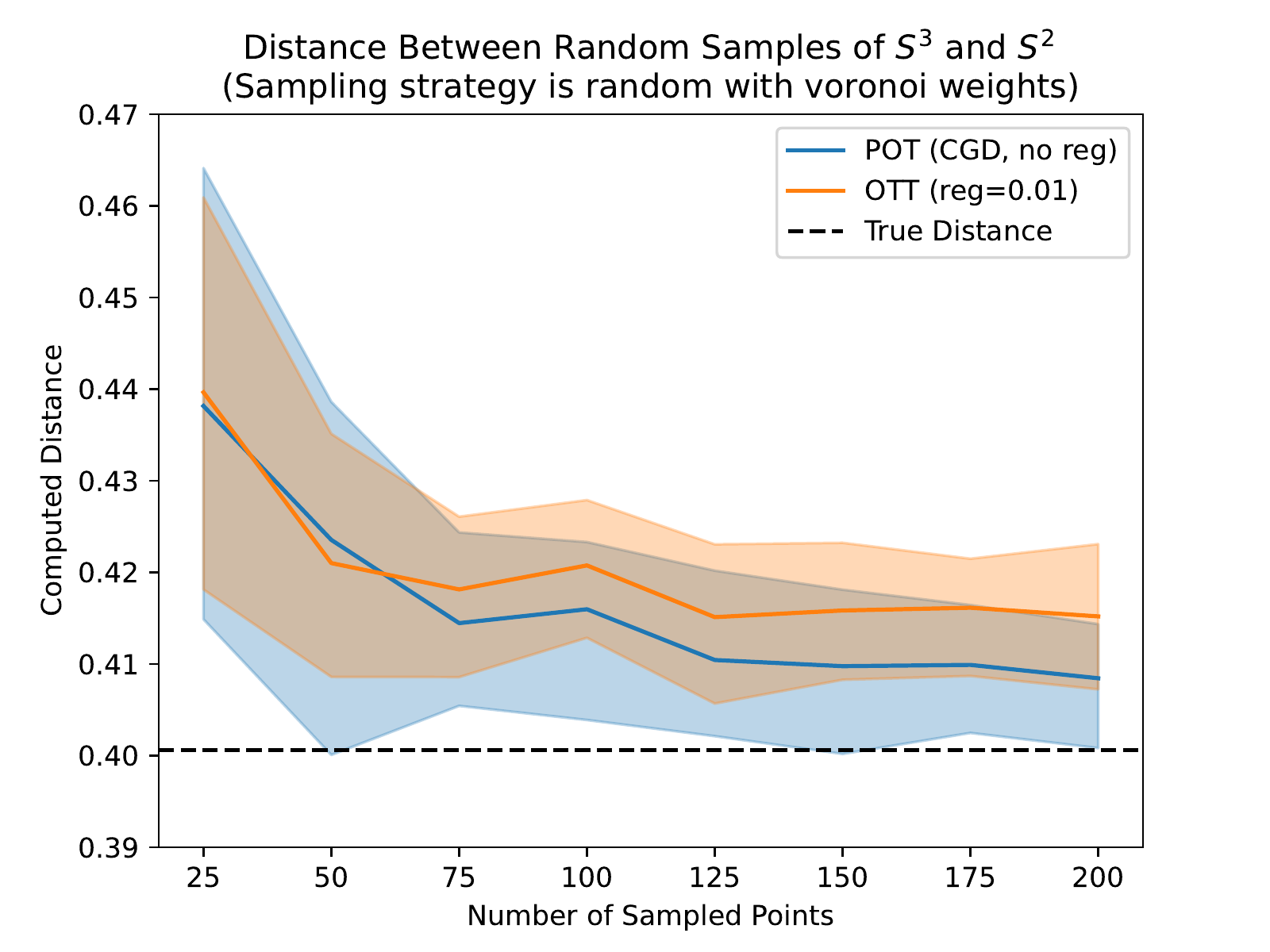} \\
        \includegraphics[width=0.45\linewidth]{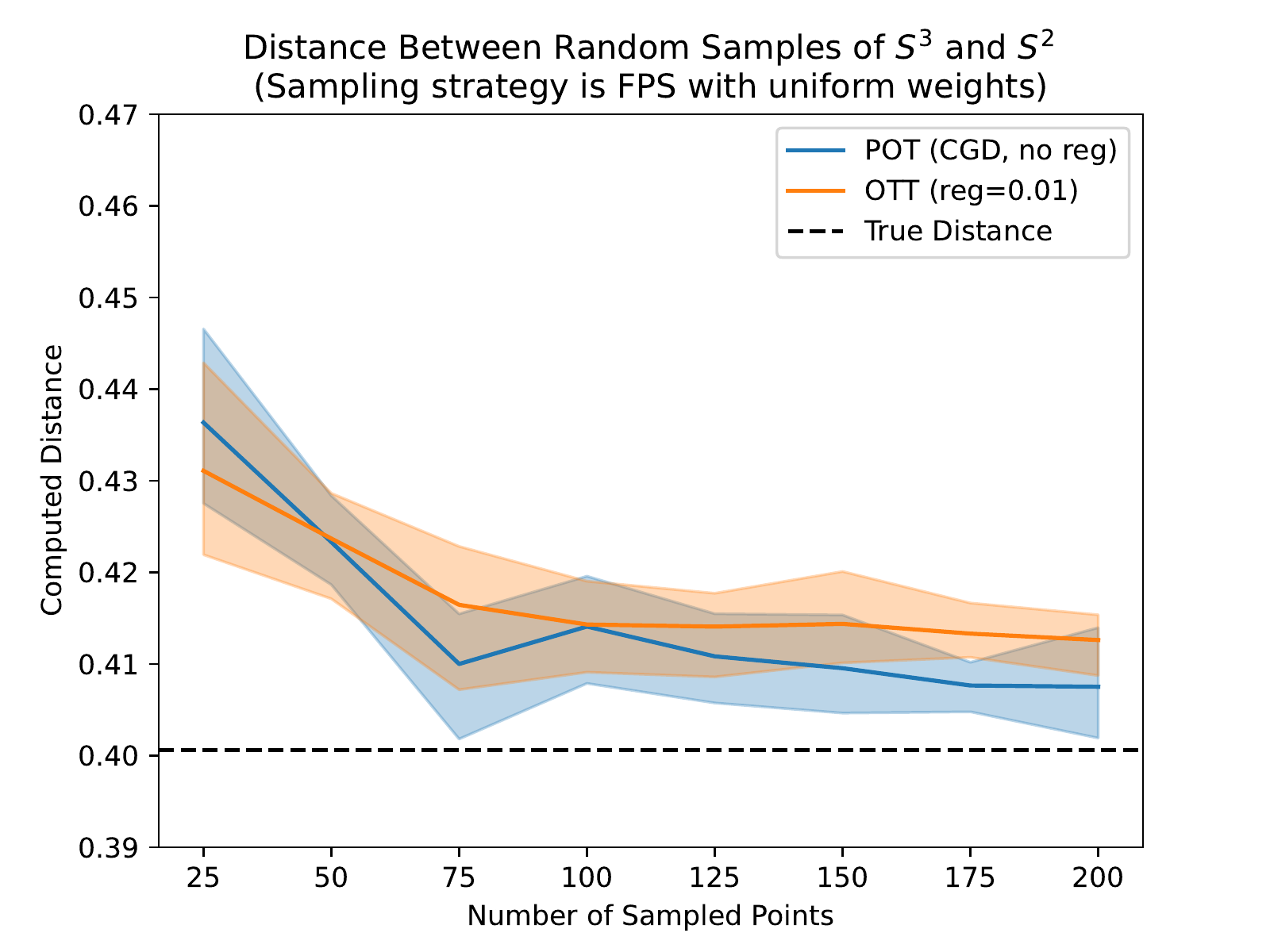} &
        \includegraphics[width=0.45\linewidth]{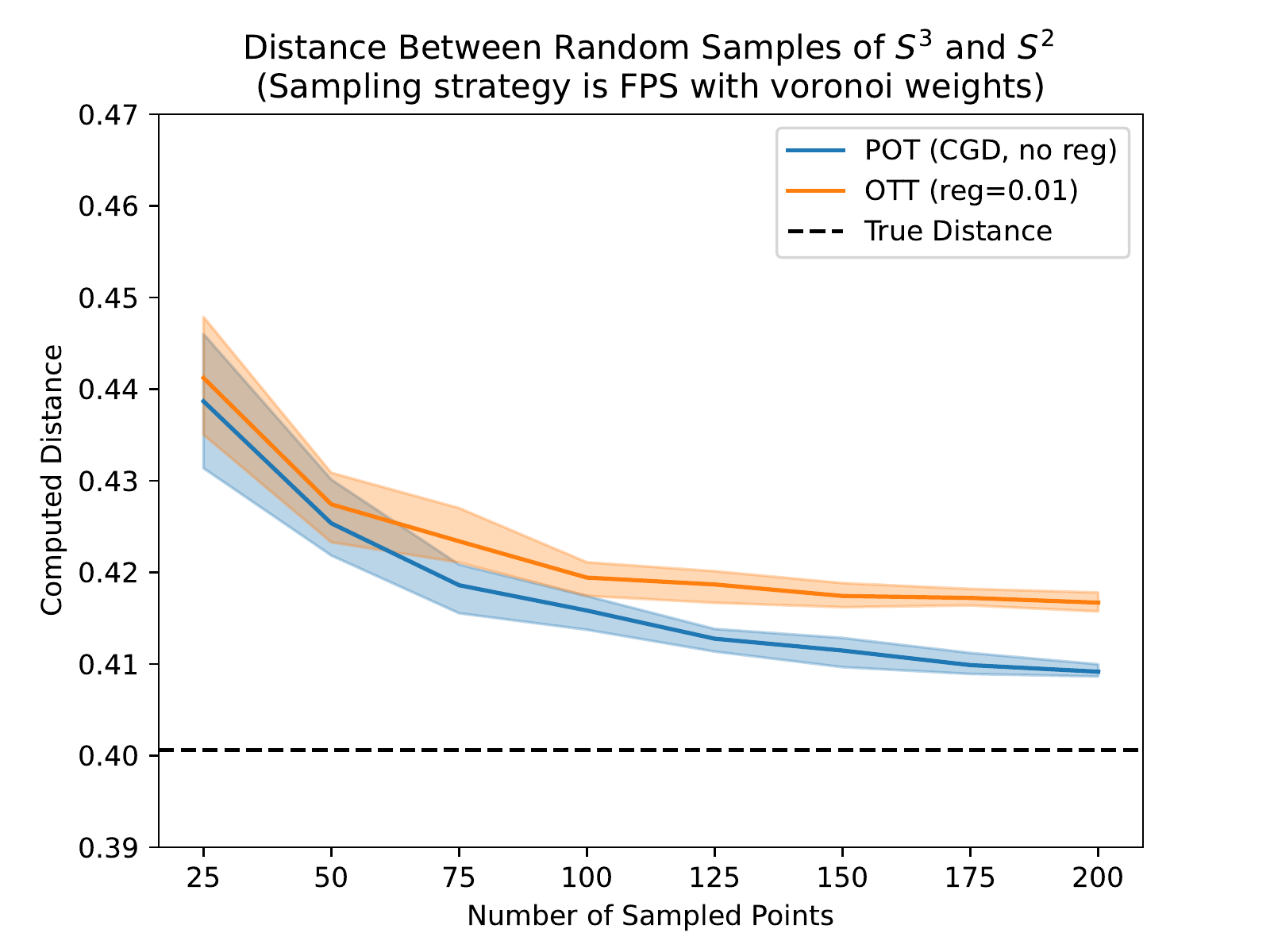} 
    \end{tabular}
    \caption{Estimating the Gromov-Wasserstein distance between $\mathbb{S}^3_E$ and $\mathbb{S}^2_E$. }\label{fig:varied-points-3-2}
\end{figure}

\subsubsection*{Observations.} 

Below it will be convenient to refer to the different combinations of procedures and solvers via the specification of the triple (Sampling, Weights, Solver) where Sampling $\in\{\text{Random, FPS}\}$, Weights $\in\{\text{Uniform, Voronoi}\}$ and Solver $\in\{\text{POT, OTT}\}$.

Figures~\ref{fig:varied-points-2-1},
\ref{fig:varied-points-3-1}, and \ref{fig:varied-points-3-2} suggest the following observations: 

\begin{itemize}

\item In all figures, the general trend is that FPS sampling overperforms Random sampling and that Voronoi weights overperform Uniform weights. These are expected as Voronoi weights are known to be optimal in the sense of quantization of measures \cite[Lemma D.6]{lipman2013conformal}, and FPS sampling is expected to provide a quasi-optimal  sampling of a metric space \cite{gonzalez1985clustering,kritika}.

\item  The combination (FPS, Voronoi, POT) produced the best results in all cases. In the case $\Sp^1$ vs. $\Sp^2$ and in the case $\Sp^1$ vs. $\Sp^3$ it provided excellent results equal to 200 points. In those cases, the Sinkhorn solver (OTT) exhibited some bias, as is expected from the the fact it uses entropic regularization  \cite{rioux2023entropic}. The case of $\Sp^2$ vs. $\Sp^3$ suggests that a dense sampling might be  necessary to approach the true value of the distance. 

\item For the case $\Sp^1$ vs $\Sp^2$  it is remarkable that with  few samples the plots of (Random, Voronoi, POT), (FPS, Uniform, POT), (FPS, Voronoi, POT) are already quite close to the true distance value. This especially the case for (FPS, Voronoi, POT).

\item In the case $\Sp^2$ vs. $\Sp^3$, in all likelihood due to the fact that 200 points is expected to be insufficient to effectively sample $\Sp^3$, most combinations exhibited some degree of error.

\end{itemize}

\section{Conclusions and Perspectives.}\label{sec:disc}

Our results provide one more infinite class of shapes for which we know the exact value of the Gromov-Wasserstein distance. Besides their intrinsic theoretical interest, our results also provide a benchmark against which the standard solvers for the Gromov-Wasserstein distance can be compared. 

\medskip
We now collect a number of questions.

\subsection*{Some questions.}

The fact that we have considered an extra parameter ($q$) in our construction of the $(p,q)$-Gromov-Wasserstein distance,  and the fact that it is known that for $p=q=\infty$, the resulting distance admits a polynomial time algorithm \cite{memoli2021ultrametric}, suggests posing the following question.

\begin{question}\label{q:comp}
    Are there classes $\mathcal{C}\subset \mathcal{G}_w$ of metric measures spaces (or networks, as in \cite{chowdhury2019gromov}) and particular choices of $p$, $q$, such that there exists a polynomial time algorithms for computing  $\dgwpq(X, Y)$ for $X,Y\in\mathcal{C}$?
\end{question}

In light of Theorem \ref{thm:equ-opti} and Remark \ref{rmk:gh_v_gw}, it is natural to ask at what level of generality the equatorial coupling is optimal.

\begin{question}\label{q:eq-map-opt}

In particular, we would like to know:
\begin{itemize}
    \item Is the equatorial coupling optimal for $\dgwpq(\Sp^m_E, \Sp^n_E)$ for other values of $p$ and $q$ beyond $(p,q) = (4,2)$?
    \item Are there values of $p, q$ so that the equatorial coupling is optimal for the $(p,q)$-Gromov-Wasserstein distance between spheres with their geodesic distances (as opposed to their Euclidean distances) $\dgwpq (\Sp^m_G, \Sp^n_G)$?
\end{itemize}
\end{question}

Since  the uniform measure $\mu_m$ (resp.\ $\mu_n$) on $\Sp^m$ (resp.\ $\Sp^n$)  can be obtained as the pushforward of the standard Gaussian measure on $\mathbb{R}^{m+1}$ (resp.\ $\mathbb{R}^{n+1}$)  under the central projection map and since,  by Remark \ref{rem:relation-couplings}, the equatorial coupling can be analogously recovered from $\gamma_{m+1,n+1}^\mathrm{gauss}$, one may wonder:

\begin{question}\label{q:gauss}
   Can one directly invoke \cite[Proposition 4.1]{salmona2021gromov} or Proposition \ref{prop:gaussian},  establishing the optimality of the  coupling $\gamma_{m+1,n+1}^\mathrm{gauss}$ for (a certain variant of) the Gromov-Wasserstein distance between Gaussian measures, to obtain a different proof of Theorem \ref{thm:equ-opti}?
\end{question}

Finally we note that in \cite[Theorems 3.2 and 3.6]{dumont2022existence}, the authors show the existence of a Monge map that induces an optimal coupling for both  the ``inner product" Gromov-Wasserstein distance and the quadratic (i.e. $p=2) $ Gromov-Wasserstein distance between two measures $\nu_m\in\mathcal{P}(\bbR^{m+1})$ and $\nu_n\in \mathcal{P}(\bbR^{n+1})$ with $n\ge m$ and $\nu_n$ absolutely continuous with respect to the Lebesgue measure on $\bbR^{n+1}$. For example in the  setting of \cite[Theorem 3.2]{dumont2022existence}, which is the closest to ours, they find that there exists an optimal coupling that is obtained through a Monge map which is the gradient of a convex function --- in a manner similar to the celebrated Brenier's theorem in optimal transport. As we pointed out on page \pageref{sec:rel-wk}, their results do not  apply in our setting, since the uniform distribution on a sphere $\Sp^d_E$ is singular with respect to the Lebesgue measure in $\bbR^{d+1}$. However, in Theorem~\ref{thm:equ-opti},  we found that an optimal coupling for the uniform distribution between spheres is in fact generated by the Monge map $e_{n,m}$ (see Definition~\ref{def:eq-map} and Claim~\ref{claim:eq-coupling}). Interestingly, $e_{n,m}:\bbS^n\to \bbS^m$ can be written as 
$e_{n,m}=T_0(\pi_{n,m}(x))$ where $T_0:\bbR^{m+1}\to\Sp^m$ is the gradient of the convex function $g:\bbR^{m+1}\to \bbR$ defined as $g(y):=(y_1^2+\dots+y_{m+1}^2)^{1/2}$.

\smallskip
This leads us to the following question:
\begin{question}\label{q:monge}
     Does the conclusion from Theorem 3.2 and 3.6 of \cite{dumont2022existence}, i.e., the existence of Monge maps that minimize the Gromov-Wasserstein distance between two metric measure spaces, holds for more general classes of measures?
\end{question}

\newcommand{\etalchar}[1]{$^{#1}$}

\appendix

\section{Relegated Proofs}\label{sec:proofs}

\subsection{Proof of Theorem \ref{thm:props}}\label{subsec:pfth1-1}

\begin{lem}\label{lem:inc}
Let $(X,d_X,\mu_X), (Y,d_Y,\mu_Y) \in \mathcal{G}_w$ be fixed and let  $\gamma \in \Mcal(\mu_X,\mu_Y)$. Then $${\operatorname{dis}_{p,q}}(\gamma) \leq {\operatorname{dis}_{p',q'}}(\gamma)$$ for all $1 \leq p \leq p' \leq \infty$ and  $1\leq q\leq q' \leq\infty$.
\end{lem}
\begin{proof}
First of all, observe that from the fact that $p'\geq p$ it follows that
\[
\dispq(\gamma) = || \Lambda_q(d_X,d_Y) ||_{L^p(\gamma \otimes \gamma)} 
\leq  || \Lambda_q(d_X,d_Y) ||_{L^{p'}(\gamma \otimes \gamma)}
= \disprimepq(\gamma).
\]

Also, note that
\begin{align*}
(\disprimepq(\gamma))^{p'} & = \int_{X \times Y} \int_{X \times Y} \bigg(\Lambda_q(d_X(x,x'),d_Y(y,y'))\bigg)^{p'} \ \gamma(dx \times dy) \ \gamma(dx' \times dy') \\
& \leq \int_{X \times Y} \int_{X \times Y} \bigg(\Lambda_{q'}(d_X(x,x'),d_Y(y,y'))\bigg)^{p'} \ \gamma(dx \times dy) \ \gamma(dx' \times dy')\\
& = (\operatorname{dis}_{p',q'}(\gamma))^{p'}
\end{align*}
where the inequality in the second line holds since $\Lambda_q \leq \Lambda_{q'}$ by Proposition \ref{prop:lmbdaq}. Combining the above, it follows that $\dispq(\gamma) \leq \disprimepq(\gamma)\leq  \disprimepprimeq(\gamma)$ which proves the lemma. \end{proof}

\begin{proof}[Proof of Theorem \ref{thm:props}]
The claim that $\dgwpq\leq \dgwppqq$ for all $1 \leq p \leq p' \leq \infty$ and  $1\leq q\leq q' \leq \infty$ follows immediately from Lemma \ref{lem:inc} above.

We now prove that $\dgwpq(X,Y)=0$ implies that $X\cong Y$. Suppose $\dgwpq(X,Y)=0.$ Then,
\[
0 = \dgwpq(X,Y) \geq \dgwpone(X,Y) = \dgwp(X,Y) \geq 0
\]
where the first inequality follows from monotonicity of $\dgwpq$. By \cite[Theorem 5.1 (a)]{memoli2011gromov}  $\dgwpq(X,Y)=0$ implies that $X \cong Y$. 

Finally, we establish the triangle inequality for $\dgwpq$ as follows. Fix arbitrary $(X,d_X,\mu_X)$, $(Y,d_Y,\mu_Y)$, and $(Z,d_Z,\mu_Z)$ in $\mathcal{G}_w$. Let $\varepsilon>0$ be an arbitrary real number. Then, one can choose couplings $\mu_{XZ}\in\mathcal{M}(\mu_X,\mu_Z)$ and $\mu_{ZY}\in\mathcal{M}(\mu_Z,\mu_Y)$ such that $$\frac{1}{2}|| \Lambda_q(d_X,d_Z) ||_{L^p(\mu_{XZ} \otimes \mu_{XZ})}<\dgwpq(X,Z)+\varepsilon$$ and $$\frac{1}{2}|| \Lambda_q(d_Z,d_Y) ||_{L^p(\mu_{ZY} \otimes \mu_{ZY})}<\dgwpq(Z,Y)+\varepsilon.$$ Next, by the gluing lemma \cite[Lemma 7.6]{villani2021topics}, there exists a probability measure  $\omega$ on $X\times Z\times Y$ such that  $(\pi_{XZ})_\#\omega=\mu_{XZ}$ and $(\pi_{ZY})_\#\omega=\mu_{ZY}$ where $\pi_{XZ}:X\times Z\times Y\rightarrow X\times Z$ and $\pi_{ZY}:X\times Z\times Y\rightarrow Z\times Y$ are the canonical projections. Now, let $\mu_{XY}:=(\pi_{XY})_\#\omega$. Then,
\begin{align*}
    \dgwpq(X,Y)&\leq\frac{1}{2}|| \Lambda_q(d_X,d_Y) ||_{L^p(\mu_{XY} \otimes \mu_{XY})}=\frac{1}{2}|| \Lambda_q(d_X,d_Y) ||_{L^p(\omega \otimes \omega)}\\
    &\leq\frac{1}{2}\left(|| \Lambda_q(d_X,d_Z) ||_{L^p(\omega \otimes \omega)}+|| \Lambda_q(d_Z,d_Y) ||_{L^p(\omega \otimes \omega)}\right)\,\\
    &=\frac{1}{2}|| \Lambda_q(d_X,d_Z) ||_{L^p(\mu_{XZ} \otimes \mu_{XZ})}+\frac{1}{2}|| \Lambda_q(d_Z,d_Y) ||_{L^p(\mu_{ZY} \otimes \mu_{ZY})}\\
    &=\dgwpq(X,Z)+\dgwpq(Z,Y)+2\varepsilon.
\end{align*}
The second inequality follows from an application of the triangle inequality for $\Lambda_q$: $$\Lambda_q(d_X(x,x'),d_Y(y,y'))\le 
\Lambda_q(d_X(x,x'),d_Z(z,z'))
+
\Lambda_q(d_Z(z,z'),d_Y(y,y'))
$$
for $(\omega\otimes\omega)$-a.e.\ $(x,x',y,y',z,z')$. This is possible since $\Lambda_q$ is a metric on $\bbR_+$ by Proposition~\ref{prop:lmbdaq}. Since the choice of $\varepsilon$ is arbitrary, one can establish the required triangle inequality.
\end{proof}

\begin{rmk}\label{rmk:sturm-comp}
Since $\Delta\!\!\!\!\Delta_{p/q,q}$ is a metric on the collection of isomorphism classes of $\mathcal{G}_w$ whenever $p\geq q$  (cf. \cite[Corollary 9.3]{sturm-ss}) and by Remark \ref{rmk:relation} $\dgwpq$ is the $q$-snowflake transform of $\Delta\!\!\!\!\Delta_{p/q,q}$ multiplied by the constant $2^{-1/q}$, one can conclude that $\dgwpq$ is also a metric on the collection of isomorphism classes of $\mathcal{G}_w$ for $p\geq q$. This provides an alternative proof  of a claim in Theorem \ref{thm:props} for the case when $p\geq q$. We note that our statement in Theorem \ref{thm:props} and the proof above do not have this restriction.
\end{rmk}

\subsection{Proof of Proposition \ref{thm:hierarchy} and an Example}\label{subsec:thmhierarchy}

\begin{ex}\label{ex:dlb-ctrex} We now provide an example  showing that 
    $\dlbpq$ is not always a lower bound for $\dgwpq$ in the case where $p < q$. We set $p = 1$, and, for some $\alpha \in [1/2,1]$,   
    \begin{align*}
        X = \{x_1, x_2\}, &\text{ with } d_X(x_1, x_2) = 1 \text{ and } 
        \mu_X(\{x_1\}) = \alpha, \ 
        \mu_X(\{x_2\}) = 1 - \alpha, \\
        Y = \{y_1 , y_2\}, &\text{ with }  d_Y(y_1, y_2) = 1 \text{ and } \mu_Y(\{y_1\}) = 1/2, \ \mu_Y(\{y_2\}) = 1/2.
    \end{align*}
    Then,
    \begin{align*}
        {\rm DLB}_{1, q} (X, Y) &= \left| 
        ({\rm diam}_1 (X))^q - 
        ({\rm diam}_1 (Y))^q \right|^{1/q} 
        =~
        \left| (2 \alpha (1 - \alpha))^q  - (1 / 2)^q \right|^{1/q}.
    \end{align*}
    On the other hand, 
    \begin{align*}
        &~2 \,d_{{\rm GW}_{1, q}} \\
        &\le {\rm dis}_{1, q} (\gamma) \\
        &= 2 (\gamma(x_1,y_1) \cdot \gamma(x_1,y_2) + \gamma(x_2,y_1) \cdot \gamma(x_2,y_2) + \gamma(x_1,y_1) \cdot \gamma(x_2,y_1) + \gamma(x_1,y_2) \cdot \gamma(x_2,y_2)) \\
        &= 2 (\alpha - 1/2)(1 - \alpha) + (\alpha - 1/2) = 4\alpha - 2 \alpha^2 - 3/2
    \end{align*}
    where $\gamma$ is the coupling measure between $\mu_X$ and $\mu_Y$ described in Figure \ref{fig:DLB_cex}. Selecting $q  = 4$, and $\alpha = 3/4$ (in fact this is a counterexample for any $q>2.5$ and any $\alpha\in(\frac{1}{2},1)$), these evaluate to:
    \begin{align*}
    {\rm DLB}_{1, 4}(X, Y) &= \left| (3/8)^4 - (1/2)^4 \right|^{1/4} \approx 0.45 \\
    2\,d_{{\rm GW}_{1, 4}}(X, Y) &\le 3 - 18/16 - 3/2 = 3/8 = 0.375.
    \end{align*}
    Thus, the diameter lower bound does not hold in general when $p < q$.
    
    \begin{figure}
    \begin{tikzpicture}
        \node at (-1.2, -1.5) {$X$};
        \node[draw, circle, inner sep=0pt, minimum size=1.22 * 20pt] (x_1) at (-0.4, -0.8) {\small $x_1$};
        \node[draw, circle, inner sep=0pt, minimum size=0.707 * 20pt] (x_2) at (-0.4, -2.2) {\small $x_2$};

        \node at (1.5, 1) {$Y$};
        \node[draw, circle] (y_1) at (0.8, 0.2) {\small $y_1$};
        \node[draw, circle] (y_2) at (2.2, 0.2) {\small $y_2$};

        \draw (x_1) -- (x_2) node[midway, left] {$1$};
        \draw (y_1) -- (y_2) node[midway, above] {$1$};

        \node (g_ac) at (0.8, -1) {\small $1/2$};
        \node (g_ad) at (2.3, -1) {\small $\alpha - 1/2$};
        \node (g_bc) at (0.8, -2.) {\small $0$};
        \node (g_bd) at (2.2, -2.) {\small $1 - \alpha$};
        
        \draw[dashed] (0.2, -1.5) -- (3.1, -1.5);
        \draw[dashed] (1.5, -0.4) -- (1.5, -2.5);
        \draw (0.2, -0.4) rectangle (3.1, -2.5);
        \node (g) at (3.5, -1.5) {$\gamma$};
    \end{tikzpicture}
    \caption{We have $\mu_X(\{x_1\}) = \alpha,\ \mu_X(\{x_2\}) = 1 - \alpha,\ \mu_Y(\{y_1\}) = \mu_Y(\{y_2\}) = 1/2$ for some $\alpha\in[1/2,1]$. In this scenario, we construct an example where $\dlbpq$ is \emph{not} a lower bound for $\dgwpq$ when $p < q$. The $(p,q)$ distortion under the coupling $\gamma$, illustrated in the square, is used to derive an upper bound on $2\,\dgwpq(X,Y)$. See Example~\ref{ex:dlb-ctrex} for more details.}
\label{fig:DLB_cex}
    \end{figure}
\end{ex}

We will need the following lemmas to prove Proposition \ref{thm:hierarchy}.

\begin{lem}[{\cite[Remark 5.8]{memoli2011gromov}}]\label{lem:ggdd-gldd}
Suppose $X \in \Gcalw$ is given. Let $\lddx(x)$ be the unique probability measure on $\bbR$ associated to the local distributions of distances $h_{X}(x,\cdot)$ and $\gddx$ be the unique probability measure associated to the global distribution of distances $H_{X}$. Then, we have the following:
\[
\int_X \lddx(x)\,\mu_X(dx) = \gddx.
\]
\end{lem}

\begin{lem}\label{lem:optcplmble} Suppose $X,Y\in \Gcalw$ and $p,q\in[1,\infty)$ are given. Then, there is a measure-valued map $(x,y)\mapsto\nu_{x,y}$ from $X\times Y$ to $\mathcal{P}(\bbR_+\times\bbR_+)$ such that 
\begin{enumerate}
    \item[(1)] $(x,y)\mapsto\nu_{x,y}(B)$ is measurable for every Borel set $B\subseteq\bbR_+\times\bbR_+$, 
    \item[(2)] $\nu_{x,y}$ belongs to $\mathcal{M}(dh_X(x),dh_Y(y))$ for each $(x,y)\in X\times Y$, and
    \item[(3)] $\,d^{(\bbR,\Lambda_q)}_{\operatorname{W}p}(dh_X(x),dh_Y(y))=\left(\int_{\bbR_+\times\bbR_+}\big(\Lambda_q(a,b)\big)^p\,\nu_{x,y}(da\times db)\right)^\frac{1}{p}.$
\end{enumerate}
\end{lem}

While the proof of the above lemma is similar to that of  Claim 1 in \cite[ pg.69]{memoli2021ultrametric}, we provide it in Section~\ref{subsec:pf-lems}  for completeness. 

\begin{proof}[Proof of Proposition \ref{thm:hierarchy}]
First, consider the $p<\infty$ case. We divide the proof into the proofs of each inequality.

\begin{proof}[Proof of $2\,\dgwpq \geq \tlbpq$.]
Fix an arbitrary coupling $\gamma\in\mathcal{M}(\mu_X,\mu_Y)$. Recall that,
$$(\dispq(\gamma))^p=\int_{X\times Y}\int_{X\times Y}\big(\Lambda_q(d_X(x,x'),d_Y(y,y'))\big)^p\,\gamma(dx'\times dy')\,\gamma(dx\times dy).$$
Furthermore, observe that for each $x\in X$ and $y\in Y$, we have
\begin{align*}
    &~\int_{X\times Y}\big(\Lambda_q(d_X(x,x'),d_Y(y,y'))\big)^p\,\gamma(dx'\times dy')
    \\
    =&~\int_{\bbR_+\times\bbR_+}\big(\Lambda_q(a,b)\big)^p\,(d_X(x,\cdot)\times d_Y(y,\cdot))_\# \ \gamma(da\times db)\\
    \geq 
    &~\big(d^{(\bbR,\Lambda_q)}_{\operatorname{W}p}(dh_X(x),dh_Y(y))\big)^p
\end{align*}
where the inequality holds since $dh_X(x)=(d_X(x,\cdot))_\#\mu_X$ and $dh_Y(y)=(d_Y(y,\cdot))_\#\mu_Y$, so $(d_X(x,\cdot), d_Y(y,\cdot))_\# \gamma$ is a coupling between $dh_X(x)$ and $dh_Y(y)$. This implies that $$\dispq(\gamma)\geq\left(\int_{X\times Y}\big(d^{(\bbR,\Lambda_q)}_{\operatorname{W}p}(dh_X(x),dh_Y(y))\big)^p\,\gamma(dx\times dy)\right)^\frac{1}{p}.$$
Since the choice of $\gamma$ is arbitrary, infimizing over $\gamma \in \mathcal{M}(\mu_X, \mu_Y)$ establishes the required inequality. \end{proof} 

\begin{proof}[Proof of $\tlbpq \geq \slbpq$.]
First, consider the case $q<\infty$.

Fix an arbitrary coupling $\gamma\in\mathcal{M}(\mu_X,\mu_Y)$. By Lemma \ref{lem:optcplmble}, there is a measurable choice $(x,y)\mapsto\nu_{x,y}$ such that for each $(x,y)\in X\times Y$, $\nu_{x,y}$ belongs to $\mathcal{M}(dh_X(x),dh_Y(y))$ and $$d^{(\bbR,\Lambda_q)}_{\operatorname{W}p}(dh_X(x),dh_Y(y))=\left(\int_{\bbR_+\times\bbR_+}\big(\Lambda_q(a,b)\big)^p\,\nu_{x,y}(da\times db)\right)^\frac{1}{p}.$$ 
Next, define a measure $\nu$ on $\bbR_+\times\bbR_+$ by:
$$\nu:=\int_{X\times Y}\nu_{x,y}\,\gamma(dx\times dy).$$
Inspection of $\nu$'s marginals shows it is a coupling between $dH_X$ and $dH_Y$.

For each $S\in\Sigma_{\bbR_+}$,
\begin{align*}
    \nu(S\times\bbR_+)&=\int_{X\times Y}\nu_{x,y}(S\times\bbR_+)\,\gamma(dx\times dy)\\
    &=\int_{X\times Y}dh_X(x)(S)\,\gamma(dx\times dy)\\
    &=\int_X dh_X(x)(S)\,\mu_X(dx)=dH_X(S)
\end{align*}
where the last equality holds by Lemma \ref{lem:ggdd-gldd}. A similar argument proves that $\nu(\bbR_+\times S)=dH_Y(S)$ so indeed $\nu\in\mathcal{M}(dH_X,dH_Y)$. Therefore,
\begin{align*}
    &~\int_{X \times Y} \big(d^{(\bbR,\Lambda_q)}_{\operatorname{W}p}(dh_X(x),dh_Y(y))\big)^p \gamma(dx \times dy)\\
    &=\int_{X \times Y}\int_{\bbR_+\times\bbR_+}\big(\Lambda_q(a,b)\big)^p\,\nu_{x,y}(da\times db) \gamma(dx \times dy)\\
    &=\int_{\bbR_+\times\bbR_+}\big(\Lambda_q(a,b)\big)^p\,\nu(da\times db)
    \geq \big(d^{(\bbR,\Lambda_q)}_{\operatorname{W}p}(dH_X,dH_Y)\big)^p.
\end{align*}

The required inequality follows since the choice of $\gamma$ is arbitrary. In order to establish the claim when  $q=\infty$ we employ the  case when  $q<\infty$ and the fact that the following equalities hold:
$$
\operatorname{TLB}_{p,\infty}=\displaystyle\lim_{q\rightarrow\infty}\tlbpq
\text{
and } \operatorname{SLB}_{p,\infty}=\displaystyle\lim_{q\rightarrow\infty}\slbpq.
$$
These  can be verified by observing that $\Lambda_q$ uniformly converges to $\Lambda_\infty$ on the compact set $\{d_X(x,x')\vert\,x,x'\in X\}\cup\{d_Y(y,y')\vert\,y,y'\in Y\}\subset\bbR_+$ as $q$ goes to infinity. 
\end{proof} 

\begin{proof}[Proof of $\slbpq \geq \dlbpqmin$.] We divide the proof into two cases.\\

\noindent \textbf{Case 1. ($p\ge q$):} Observe first that, since $p<\infty$, we have
\begin{align*}
    ~(\slbpq(X,Y))^q&=\left(\int_0^1\big(\Lambda_q(H_X^{-1}(u),H_Y^{-1}(u))\big)^p \ du\right)^{q/p}\\
    &=\left(\int_0^1\big((H_X^{-1}(u))^q-(H_Y^{-1}(u))^q\big)^{p/q}\ du\right)^{q/p}\\
    &=\left(\int_0^1\big(F_{(S_q)_\#dH_X}^{-1}(u)-F_{(S_q)_\#dH_Y}^{-1}(u)\big)^{p/q}\ du\right)^{q/p}\\
    &=d^\bbR_{\operatorname{W}p/q}((S_q)_\#dH_X,(S_q)_\#dH_Y)\\
    &\geq \bigg\vert d^\bbR_{\operatorname{W}p/q}((S_q)_\#dH_X,\delta_0)-d^\bbR_{\operatorname{W}p/q}((S_q)_\#dH_Y,\delta_0)\bigg\vert
\end{align*}
where $\delta_0$ is the Dirac measure at zero. Note that the first equality follows from Remark \ref{rmk:clfrmdWR+Lbdaq} and the last inequality follows from the triangle inequality of $d^\bbR_{\operatorname{W}p/q}$. Next, we obtain via Example \ref{ex:wass-delta} and Remark \ref{rem:p-diameter and moment} that
$$d^\bbR_{\operatorname{W}p/q}((S_q)_\#dH_X,\delta_0)=\left(\int_X d^p_X(x,x')\,\mu_X(dx)\,\mu_X(dx')\right)^{q/p}=(\diamp(X))^q.$$
This establishes the inequality $\slbpq\ge \dlbpq$ when $p\ge q$ and $p<\infty$. \\ 

\noindent \textbf{Case 2. ($p\le q$):} Note that $\dlbpqmin=\operatorname{DLB}_{p,p}$ in this case. Also, it is easy to verify that $\slbpq\ge\operatorname{SLB}_{p,p}$ since $\Lambda_p\le\Lambda_q$. Moreover, $\operatorname{SLB}_{p,p}\ge\operatorname{DLB}_{p,p}$ by the previous Case 1. Hence, we achieve the inequality $\slbpq\ge\dlbpp=\dlbpqmin$. \end{proof}

Note that this finishes the proof the hierarchy in the case where $p<\infty$. More precisely, thus far we have proved that
\begin{equation}\label{eq:p-fin}
2\,\dgwpq\ge \tlbpq \ge \slbpq \ge \dlbpqmin
\quad
\text{for }
p\in[1,\infty)
\text{ and }
q\in[1,\infty].
\end{equation}

Lastly, in order to prove the $p=\infty$ case, we proceed as follows. Unless otherwise specified, we consider $q\in [1,\infty]$.

\begin{proof}[Proof of $2d_{\operatorname{GW}_{\infty,q}}\geq \operatorname{TLB}_{\infty,q}$.] Note that by Theorem~\ref{thm:props} 
$$
2d_{\operatorname{GW}_{\infty,q}}\ge 2\limsup_{p\to\infty}\dgwpq.
$$
Also, it is easy to verify that $\operatorname{TLB}_{\infty,q}\geq\displaystyle\limsup_{p\rightarrow\infty}\tlbpq$ since $d^{(\bbR,\Lambda_q)}_{\operatorname{W}p}\le d^{(\bbR,\Lambda_q)}_{\operatorname{W}p'}$ whenever $1\le p\le p'\le\infty$. Therefore, $2d_{\operatorname{GW}_{\infty,q}}\geq \operatorname{TLB}_{\infty,q}$ is achieved from the $p<\infty$ case.
\end{proof} 

\begin{proof}[Proof of $\operatorname{TLB}_{\infty,q}\ge\operatorname{SLB}_{\infty,q}$.]
This is straightforward since $\operatorname{TLB}_{\infty,q}\geq\displaystyle\limsup_{p\rightarrow\infty}\tlbpq$, $\operatorname{SLB}_{\infty,q}=\displaystyle\lim_{p\rightarrow\infty}\slbpq$ and we proved that $\tlbpq\ge\slbpq$ in the $p<\infty$ case.
\end{proof}

\begin{proof}[Proof of $\operatorname{SLB}_{\infty,q}\ge\operatorname{DLB}_{\infty,\infty\wedge q}$.]
If $q<\infty$, then it is easy to verify the claim from the facts $\operatorname{SLB}_{\infty,q}=\displaystyle\lim_{p\rightarrow\infty}\slbpq$, $\operatorname{DLB}_{\infty,q}=\displaystyle\lim_{p\rightarrow\infty}\dlbpq$ and the $p<\infty$ case. Finally, if $q=\infty$, one can use the facts $\operatorname{SLB}_{\infty,\infty}=\displaystyle\lim_{q\rightarrow\infty}\operatorname{SLB}_{\infty,q}$, $\operatorname{DLB}_{\infty,\infty}=\displaystyle\lim_{q\rightarrow\infty}\operatorname{DLB}_{\infty,q}$, and the previous case when $p=\infty$, $q<\infty$. This completes the proof. \end{proof} 

By the last three cases, we have thus proved that
\begin{equation}\label{eq:p-infin}
2d_{\operatorname{GW}_{\infty,q}}
\geq 
\operatorname{TLB}_{\infty,q}
\ge 
\operatorname{SLB}_{\infty,q}
\ge
\operatorname{DLB}_{\infty,\infty\wedge q}
\quad
\text{for }
q\in[1,\infty].
\end{equation}
Equations~\eqref{eq:p-fin} and \eqref{eq:p-infin} establish the theorem.
\end{proof}

\subsection{Proofs of Lemmas}\label{subsec:pf-lems}

\begin{proof}[Proof of Lemma~\ref{lem:GW-eu-eq}]
By equations~\eqref{eq:eu-dis},\eqref{eq:eu-dis-2} and \eqref{eq:defJmu} we have
$$
\dis_{4,2}^4(\gamma_{m,n})
=\dfrac{4}{n+1}+\dfrac{4}{m+1}
-8\int_{\Sp^n\times\Sp^m}\int_{\Sp^n\times\Sp^m}\langle x, x'\rangle\langle y, y'\rangle \ \gamma_{m,n}(dx \times dy) \ 
\gamma_{m,n}(dx' \times dy').
$$
We then compute
\begin{align*}
&\int_{\Sp^n\times\Sp^m}\int_{\Sp^n\times\Sp^m}\langle x, x'\rangle\langle y, y'\rangle \ \gamma_{m,n}(dx \times dy) \ \gamma_{m,n}(dx' \times dy')\\
    &=\iint_{\Sp^n\times\Sp^n}
    \langle e_{n,m}(y), e_{n,m}(y')\rangle 
    \langle y, y'\rangle
    \ \mu_n(dy) \ \mu_n(dy')\\
    &=\iint_{\Sp^n\times\Sp^n}
    \sum_{i=1}^{n+1}
    y_iy_i'\sum_{j=1}^{m+1}\frac{y_j}{\Vert y_A\Vert}\frac{y_j'}{\Vert y_A'\Vert} \ \mu_n(dy) \ \mu_n(dy')\\
    &=\sum_{i=1}^{n+1}\sum_{j=1}^{m+1}\left(\int_{\Sp^m}\frac{y_iy_j}{\Vert y_A\Vert} \ \mu_n(dy)\right)^2\\
    &=\sum_{j=1}^{m+1}\left(\int_{\Sp^n}
    \frac{y_j^2}{\Vert y_A\Vert} \ \mu_n(dy)\right)^2\\
    &=(m+1)\left(\int_{\Sp^n}\frac{y_1^2}{\Vert y_A\Vert} \ \mu_n(dy)\right)^2.
\end{align*}

Here, note that the fourth equality holds because $\int_{\Sp^n}\frac{y_iy_j}{\Vert y_A\Vert} \ \mu_n(dy)=0$ whenever $i\neq j$ because $\mu_m=(N_m)_\#\eta_{m+1}$ where $N_m:\bbR^{m+1}\rightarrow\Sp^m$ is the map such that $(y_1,\dots,y_{m+1})\mapsto\frac{1}{\sqrt{\sum_{i=1}^{m+1}y_i^2}}(y_1,\dots,y_{m+1})$ and $\eta_{m+1}$ is the standard Gaussian measure on $\bbR^{m+1}$.

Also,
\begin{align*}
    \int_{\Sp^n}\frac{y_1^2}{\Vert y_A\Vert} \ \mu_n(dy)
    =\frac{1}{m+1}\int_{\Sp^n}\frac{\sum_{j=1}^{m+1} y_j^2}{\Vert y_A\Vert} \ \mu_n(dy)=\frac{1}{m+1}\int_{\Sp^n}\Vert y_A\Vert \ \mu_n(dy).
\end{align*}
It remains to calculate the expectation of $\|y_A\|$, which follows from the identical calculation done in the proof of Theorem~\ref{thm:equ-opti}. In particular, we appeal to a characterization of $\mu_n$ in terms of standard Gaussian random variables $Z_1,\dots,Z_{n+1}$ in order to write that
$$
\|y_A\|^2\sim {\rm Beta}\left(\dfrac{m+1}{2},\dfrac{n-m}{2}\right).
$$
And hence
\begin{align*}
\int_{\Sp^n}\Vert y_A\Vert \ \mu_n(dy)
=&~
\dfrac{1}{\beta\left(\tfrac{m+1}{2},\,\tfrac{n-m}{2}\right)}
\int_0^1
\sqrt{t}\cdot
t^{(m+1)/2-1}
(1-t)^{(n-m)/2-1} \
du\\
=&~
\dfrac{\beta\left(\tfrac{m+2}{2},\,\tfrac{n-m}{2}\right)}
{\beta\left(\tfrac{m+1}{2},\,\tfrac{n-m}{2}\right)} 
=~\dfrac{\Gamma\left(\tfrac{m+2}{2}\right)
\Gamma\left(\tfrac{n+1}{2}\right)}{
\Gamma\left(\tfrac{m+1}{2}\right)
\Gamma\left(\tfrac{n+2}{2}\right)
}.
\end{align*}
This finishes the proof.
\end{proof}

\begin{proof}[Proof of Lemma~\ref{lem:optcplmble}] First, let $S:=\{d_X(x,x')\vert\,x,x'\in X\}\cup\{d_Y(y,y')\vert\,y,y'\in Y\}\subset\bbR_+$. Since both $X$ and $Y$ are compact, $S$ is also compact. Also, it is easy to verify that all $\Lambda_r$ (for $r\in [1,\infty)$) induce the same topology and thus the same Borel sets on $S$. Therefore all $d^{(\bbR,\Lambda_r)}_{\operatorname{W}p}$ (for $r\in [1,\infty)$) metrize the  weak topology on $\mathcal{P}(S)$. By \cite[Remark 1]{memoli2022distance}, the following two maps are continuous with respect to the weak topology and thus measurable:
$$\Phi_1:X\rightarrow\mathcal{P}(S),\,x\mapsto dh_X(x)$$
and
$$\Phi_2:Y\rightarrow\mathcal{P}(S),\,y\mapsto dh_Y(y).$$
Since $S$ is a compact space, the space $(\mathcal{P}(S),d^{(\bbR,\Lambda_q)}_{\operatorname{W}p})$ is separable \cite[Theorem 6.18]{villani2009optimal}. This yields that $\Sigma_{\mathcal{P}(S)\times\mathcal{P}(S)}=\Sigma_{\mathcal{P}(S)}\otimes\Sigma_{\mathcal{P}(S)}$ \cite[Proposition 1.5]{folland1999real}. Hence, the product $\Phi$ of $\Phi_1$ and $\Phi_2$, defined by
$$\Phi:X\times Y\rightarrow\mathcal{P}(S)\times\mathcal{P}(S),\,(x,y)\mapsto(dh_X(x),dh_Y(y))$$
is measurable \cite[Proposition 2.4]{folland1999real}. Since $\Phi$ is measurable, a direct application of \cite[Corollary 5.22]{villani2009optimal} gives the claim.
\end{proof}

\section{Calculations}\label{app:calc}

\subsection{Lower Bounds for $\dgwft$ between Spheres  with the Geodesic distance.} \label{sec:sphere-lb-geo}

In preparation for the diameter lower bounds, we first compute the $4$-diameters of $\Spg^m$ for $m=0,1,2$ using the formula for the global distance distribution given in Example~\ref{ex:gdd-geo}.

\begin{ex}(The global distance distributions of $\Spg^0$, $\Spg^1$ and $\Spg^2$) \label{ex:gdd-012}
    $\Spg^0$ consists of two points which are at distance  $\pi$  apart and so $ H_{\Spg^0}(t) = \mu_0 \otimes \mu_0 \{ (x,x') \in \Spg^0\times \Spg^0 | d_0(x,x') \le t  \} $ is
\[  H_{\Spg^0}(t)=
    \begin{cases} 
      \frac{1}{2} & 0 \le  t <\pi \\
      1 & t =  \pi.  
   \end{cases}
\] 
By Example~\ref{ex:gdd-geo}, the global distance distributions for $\Spg^1$ and $\Spg^2$ are 
    \[
         H_{\Spg^1}(t) = \frac{t}{\pi}, \qquad
         H_{\Spg^2}(t) = \frac{(1-\cos t)}{2}
         \text{ for }
         t\in[0,\pi]
         .    
    \]
Consequently, the generalized inverses (see equation~\eqref{eq:def-ginv}) are 
\[  H_{\Spg^0}^{-1}(u)= 
    \begin{cases} 
      0 & 0 \le  u \le \frac{1}{2} \\
      \pi & \frac{1}{2} <  u \le 1
   \end{cases}
\]
while
\[H_{\Spg^1}^{-1}(u)=u\pi 
\quad
\text{and}
\quad 
H_{\Spg^2}^{-1}(u)=\arccos(1-2u)
\quad 
\text{for}
\quad 
u\in [0,1]
.
\]
\end{ex}

\begin{ex}($\mathrm{DLB}_{4,2}(\Spg^0,\Spg^1)$ and $\mathrm{DLB}_{4,2}(\Spg^1,\Spg^2)$)\label{ex:dlb-geo} By Example~\ref{ex:gdd-012} the $4$-diameters of $\Spg^0$, $\Spg^1$ and $\Spg^2$ are
$$    \diam_4(\Spg^0)
=\bigg(\int_0^1\left(H^{-1}_{\Spg^0}(u)\right)^{4}du\bigg)^{1/4} =\bigg(\int_{1/2}^1\pi^{4}du\bigg)^{1/4} = \frac{\pi}{2^{1/4}}.
$$

$$    \diam_4(\Spg^1) =
    \bigg(\int_0^1\left(H^{-1}_{\Spg^1}(u)\right)^{4}du\bigg)^{1/4} 
    =
    \bigg(\int_0^1(u\pi)^{4}du\bigg)^{1/4} = \frac{\pi}{5^{1/4}}
$$
$$    \diam_4(\Spg^2) =
    \bigg(\int_0^1\left(H^{-1}_{\Spg^2}(u)\right)^{4}du\bigg)^{1/4} 
    =
    \bigg(\int_0^1(\arccos(1-2u))^{4}du\bigg)^{1/4} = \bigg(24- 6\pi^2 +\frac{\pi^4}{2}\bigg)^{1/4}. 
$$
Hence, by the definition of $\mathrm{DLB}_{4,2}$ we have 
    \[\mathrm{DLB}_{4,2}(\Spg^0,\Spg^1)= \bigg|\bigg(\frac{\pi^4}{2} \bigg)^{2} -\bigg(\frac{\pi^4}{5} \bigg)^{2}\bigg|^{1/2} = \pi\bigg(\frac{1}{\sqrt{2}}-\frac{1}{\sqrt{5}}\bigg)^{1/2} \approx 1.602.\]
and
    \[\mathrm{DLB}_{4,2}(\Spg^1,\Spg^2) = \bigg| \bigg(\frac{\pi}{5^{1/4}} \bigg)^{2} - \bigg(24- 6\pi^2 +\frac{\pi^4}{2}\bigg)^{2/4} \bigg|^{1/2} \approx 0.861. \]
\end{ex}

\begin{ex}($\mathrm{SLB}_{4,2}(\Spg^0,\Spg^1)$ and $\mathrm{SLB}_{4,2}(\Spg^1,\Spg^2)$)\label{ex:slb-geo} By Example~\ref{ex:gdd-012} and the definition of $\mathrm{SLB}_{4,2}$ we obtain
\begin{align*}
    \mathrm{SLB}_{4,2}(\Spg^0,\Spg^1) 
    =& \bigg(\int_0^1 |(H^{-1}_{\Spg^0}(u))^2 - (H^{-1}_{\Spg^1}(u))^2|^2 \ du\bigg)^{1/4} \\
    =&~ \bigg(\int_0^{\frac{1}{2}} |u^2\pi^2|^2 \ du + \int_{\frac{1}{2}}^1 |\pi^2-u^2\pi^2 |^2 \ du\bigg)^{1/4} 
    = \pi\left(\frac{1}{2}+\frac{1}{5}-\frac{7}{12}\right)^{1/4}\approx 1.836,
\end{align*}
and similarly
\begin{align*}
        \mathrm{SLB}_{4,2}(\Spg^1,\Spg^2) 
        =&~  
        \bigg(\int_0^1 |(H^{-1}_{\Spg^1}(u))^2 - (H^{-1}_{\Spg^2}(u))^2|^2 \ du\bigg)^{1/4}\\ 
        =&~
        \bigg(\int_0^1 |u^2\pi^2 - (\arccos(1-2u))^2|^2 \ du\bigg)^{1/4} \approx 0.931.
\end{align*}
\end{ex}

 \subsection{Lower Bounds for $\dgwft$ between Spheres  with the Euclidean metric.}\label{sec:sphere-lb-euc}
     \begin{ex}(The global distance distributions of $\Spe^0$, $\Spe^1$ and $\Spe^2$)\label{ex:gdd-euc-012}
          $\Spe^0$ consists of two points which are at distance  $2$  apart.
        \[  H_{\Spe^0}(t)=
        \begin{cases} 
        \frac{1}{2} & 0 \le  t <2\\
        1 & t = 2 
        \end{cases}
         \] 
    Consequently, 
    \[  H_{\Spe^0}^{-1}(u)= 
    \begin{cases} 
      0 & 0 \le  u \le \frac{1}{2} \\
      2 & \frac{1}{2} <  u \le 1
   \end{cases}
    \]
    Similarly, the global distance distributions for $\Spg^1$ and $\Spg^2$ are
    \[
         H_{\Spe^1}(t) = \frac{2}{\pi}\arcsin{\frac{t}{2}}, \qquad
         H_{\Spe^2}(t) = \frac{t^2}{4}   
         \quad
         \text{for }
         t\in[0,2].
         \]
    Thus, for $u\in[0,1]$
    \[H_{\Spe^1}^{-1}(u)=2\sin{\left(\frac{u\pi}{2}\right)}
    \quad
    \text{and}
    \quad 
    H_{\Spe^2}^{-1}(u)=2\sqrt{u}
    \]
    
    \end{ex}
     
     \begin{ex}($\mathrm{DLB}_{4,2}(\Spe^0,\Spe^1)$ and $\mathrm{DLB}_{4,2}(\Spe^1,\Spe^2)$)\label{ex:dlb-euc-012}
    By Example~\ref{ex:gdd-euc-012}, the $4$-diameters of $\Spe^0$, $\Spe^1$ and $\Spe^2$ are
    $$
    \diam_4(\Spe^0)
    =\bigg(\int_0^1\left(H^{-1}_{\Spe^0}(u)\right)^{4}du\bigg)^{1/4} =\bigg(\int_{1/2}^12^{4}du\bigg)^{1/4} = 2^{3/4}.
    $$   
    $$
    \diam_4(\Spe^1) =2\bigg(\int_0^1 \sin^4{\left(\frac{u\pi}{2}\right)} du\bigg)^{1/4} = 2\left(\frac{3}{8}\right)^{1/4}.
    $$  
    $$
    \diam_4(\Spe^2) =2\bigg(\int_0^1(\sqrt{u})^{4}du\bigg)^{1/4} = \frac{2}{3^{1/4}}. 
    $$
    Hence by the definition of $\dlbft$ we have
    \[\mathrm{DLB}_{4,2}(\Spe^0,\Spe^1)= \bigg|2^{3/2} -2^2\left(\frac{3}{8}\right)^{2/4}\bigg|^{1/2}  \approx 0.616.\]
    and
    \[\mathrm{DLB}_{4,2}(\Spe^1,\Spe^2) = \bigg| 2^2\left(\frac{3}{8}\right)^{2/4} - \left(\frac{2}{3^{1/4}}\right)^2 \bigg|^{1/2} \approx 0.374. 
    \]
    \end{ex}
  
    \begin{ex}($\mathrm{SLB}_{4,2}(\Spe^0,\Spe^1)$ and $\mathrm{SLB}_{4,2}(\Spe^1,\Spe^2)$)\label{ex:slb-euc-012}
    By Example~\ref{ex:gdd-euc-012} and the definition of $\slbft$ we obtain
    \begin{align*}
    \mathrm{SLB}_{4,2}(\Spe^0,\Spe^1) 
    =& \bigg(\int_0^1 |(H^{-1}_{\Spe^0}(u))^2 - (H^{-1}_{\Spe^1}(u))^2|^2 \ du\bigg)^{1/4}\\ 
    =& \bigg(\int_0^{\frac{1}{2}} |H^{-1}_{\Spe^1}(u)|^4 \ ds + \int_{\frac{1}{2}}^1 |2^2-(H^{-1}_{\Spe^1}(u))^2 |^2 \ du\bigg)^{1/4} \\
    =& \bigg(\int_0^{\frac{1}{2}} \left|2^2\sin^2{\left(\frac{u\pi}{2}\right)}\right|^2 \ du + \int_{\frac{1}{2}}^1 \left|2^2-2^2\sin^2{\left(\frac{u\pi}{2}\right)} \right|^2 \ du\bigg)^{1/4} \\ \approx&~ 0.976,
    \end{align*}
    and similarly
   \[
        \mathrm{SLB}_{4,2}(\Spe^1,\Spe^2) =  \bigg(\int_0^1 \left|2^2\sin^2{\left(\frac{u\pi}{2}\right)} - 2^2\,u\right|^2 \ du\bigg)^{1/4} \approx 0.549.
  \]
\end{ex}

\subsection{Distortion $\disft$ under the Equatorial coupling} 

\begin{ex}($\dis_{4,2}(\gamma_{0,1},\Spg^0,\Spg^1)$)\label{ex:dis42g01-app} By Remark~\ref{rmk:max-dis42} and Example~\ref{ex:dlb-geo}, we have
\begin{align*}\label{eq:dis42g01-app}
    &\big(\dis_{4,2}(\gamma_{0,1},\Spg^0,\Spg^1)\big)^4\\
    =&~
    (\diam_4(\Spg^0))^4+(\diam_4(\Spg^1))^4
    -2\int 
    \big(d_0(e_{1,0}(y),e_{1,0}(y'))\big)^2\,
    \big(d_1(y,y')\big)^2\,
    \mu_1(dy)\mu_1(dy')\\
    =&~
    \dfrac{\pi^4}{2}+\dfrac{\pi^4}{5}
    -2\int_{\Sp^1}\int_{\Sp^1}
    (d_0(e_{1,0}(y),e_{1,0}(y')))^2(d_1(y,y'))^2
    \mu_1(dy)\mu_1(dy').\numberthis
\end{align*}
    To compute the integral on the right hand side, we use polar coordinates. We write $y=(\cos\theta,\sin\theta)$ and $y'=(\cos\theta',\sin\theta')$, where $\theta,\theta'\in[0,2\pi]$.  Note also that with this parametrization and by the definition of the $e_{0,1}$ map, we can write
$$
e_{1,0}(y)={\rm sign}(\cos\theta)\,;\,
e_{1,0}(y')={\rm sign}(\cos\theta')
$$
Hence by definition of $d_0$, 
$$
d_0(e_{1,0}(y),e_{1,0}(y'))
=\pi\mathbbm{1}(\cos\theta\cdot \cos\theta'<0).
$$
Moreover,
$$
d_1(y,y')
=\arccos(\cos\theta\cos\theta'+\sin\theta\sin\theta')
=\arccos(\cos(\theta-\theta'))
.
$$
We then have 
\begin{align*}
&\int_{\Sp^1}\int_{\Sp^1}
    (d_0(e_{1,0}(y),e_{1,0}(y')))^2(d_1(y,y'))^2
    \mu_1(dy)\mu_1(dy')\\
=&~   
\int_0^{2\pi} 
\int_0^{2\pi}
(\pi\mathbbm{1}(\cos\theta\cdot \cos\theta'<0))^2
(\arccos(\cos(\theta-\theta')))^2
\times
\left(\dfrac{1}{2\pi}\right)^2
d\theta d\theta'\\
=&~
\dfrac{1}{4}
\left(
\int_0^{\pi/2}
\int_{\pi/2}^{3\pi/2}
(\arccos(\cos(\theta'-\theta)))^2
d\theta d\theta'
+
\int_{3\pi/2}^{2\pi}
\int_{\pi/2}^{3\pi/2}
(\arccos(\cos(\theta-\theta')))^2
d\theta d\theta'
\right)\times 2
\\
=&~\dfrac{1}{2}
\left(
\int_0^{\pi/2}
\int_{\pi/2-\theta}^{3\pi/2-\theta}
(\arccos(\cos(t)))^2
dt d\theta
+
\int_{3\pi/2}^{2\pi}
\int_{\theta-3\pi/2}^{\theta-\pi/2}
(\arccos(\cos(t)))^2
dt d\theta
\right).
\end{align*}
The first integral becomes
\begin{align*}
\int_0^{\pi/2}
\int_{\pi/2-\theta}^{3\pi/2-\theta}
(\arccos(\cos(t)))^2
dt d\theta
=&~
\int_0^{\pi/2}
\left[
\int_{\pi/2-\theta}^{\pi}
t^2
dt 
+
\int_{\pi}^{3\pi/2-\theta}
(2\pi-t)^2
dt 
\right]
d\theta\\
=&~
\int_0^{\pi/2}
\left[
\left(
\dfrac{\pi^3}{3}
-
\dfrac{(\pi/2-\theta)^3}{3}
\right)
+
\left(
\dfrac{\pi^3}{3}
-
\dfrac{(\theta+\pi/2)^3}{3}
\right)
\right]
d\theta\\
=&~
\dfrac{\pi^4}{3}
-
\int_0^{\pi}\dfrac{u^3}{3}du
=
\dfrac{\pi^4}{3}
-
\dfrac{\pi^4}{12}
=\dfrac{\pi^4}{4}.
\end{align*}
Similarly, the second integral also evaluates to 
\begin{align*}
\int_{3\pi/2}^{2\pi}
\int_{\theta-3\pi/2}^{\theta-\pi/2}
(\arccos(\cos(t)))^2
dt d\theta
=&~
\int_{3\pi/2}^{2\pi}
\left[
\int_{\theta-3\pi/2}^{\pi}
t^2
dt
+
\int_{\pi}^{\theta-\pi/2}
(2\pi-t)^2
dt 
\right]
d\theta
=\dfrac{\pi^4}{4}.
\end{align*}
Thus, 
$$
\int_{\Sp^1}\int_{\Sp^1}
    (d_0(e_{1,0}(y),e_{1,0}(y')))^2(d_1(y,y'))^2
    \mu_1(dy)\mu_1(dy')
=\dfrac{\pi^4}{4}.    
$$
The rest of the calculations are given in Example~\ref{ex:dis42g01}.
\end{ex}

\section{Another experiment: Varying dimensions experiment}\label{app:exps}
In this experiment, we fixed the number of samples taken at $100$ and varied
the dimensions of the two spheres between 1 and 7. The subsampling mehtod was chosen to be FPS and the Weights were those produced by the Voronoi method. Finally, we fixed the solver to the CGD solver from POT.

Using the results of ten trials  ($n_\textrm{trials}=10$) for fixed sphere dimensions, $m$ and $n$,
$\hat{d}^i_{m, n}$, where $i = 1, \ldots, n_\textrm{trials}$, we estimated the
true distance via the average over trials,
\[ d_{m, n} \approx \hat{d}_{m, n} :=
    \frac{1}{n_\textrm{trials}}\sum_{i=1}^{n_\textrm{trials}} \hat{d}^i_{m, n}.
\]

We then recorded the relative error of this estimator: $\textrm{relative
error}_{m, n} := \frac{\hat{d}_{m, n} - d_{m, n}}{d_{m, n}}$ in the corresponding entry of the heatmap shown in
Figure~\ref{fig:dimension-heatmaps-pot}.  We observe a dramatic decrease in accuracy as the dimensions of both spheres increase which of course one would expect to reduce by using a larger number of points.

\begin{figure}
    \centering
    \begin{tabular}{cc}
        \includegraphics[width=0.65\linewidth]{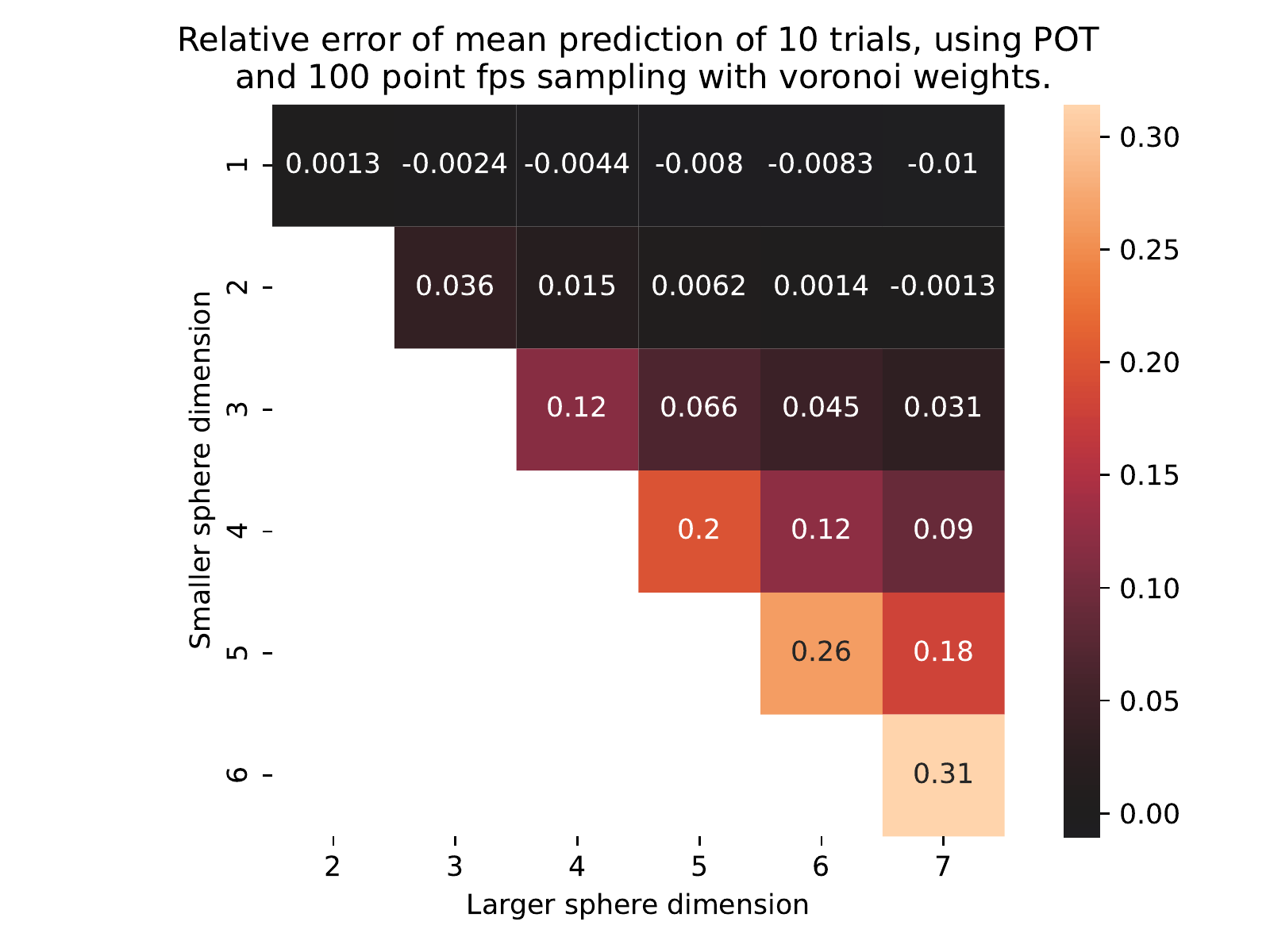} 
    \end{tabular}
    \caption{Relative errors of computed and true differences using the CGD solver from POT, FPS as the sampling procedure and Voronoi weights.}\label{fig:dimension-heatmaps-pot}
\end{figure}


\begin{thebibliography}{MDFdRA02}
	
	\bibitem[ABC{\etalchar{+}}22]{adams2022gromov}
	Henry Adams, Johnathan Bush, Nate Clause, Florian Frick, Mario G{\'o}mez,
	Michael Harrison, R~Amzi Jeffs, Evgeniya Lagoda, Sunhyuk Lim, and Facundo
	M{\'e}moli.
	\newblock {G}romov-{H}ausdorff distances, {B}orsuk-{U}lam theorems, and
	{V}ietoris-{R}ips complexes.
	\newblock {\em arXiv preprint arXiv:2301.00246}, 2022.
	
	\bibitem[AGS05]{ambrosio2005gradient}
	Luigi Ambrosio, Nicola Gigli, and Giuseppe Savar{\'e}.
	\newblock {\em Gradient flows: in metric spaces and in the space of probability
		measures}.
	\newblock Springer Science \& Business Media, 2005.
	
	\bibitem[AH12]{atkinson2012spherical}
	Kendall Atkinson and Weimin Han.
	\newblock {\em Spherical harmonics and approximations on the unit sphere: an
		introduction}, volume 2044.
	\newblock Springer Science \& Business Media, 2012.
	
	\bibitem[AMJ18]{alvarez2018gromov}
	David Alvarez-Melis and Tommi Jaakkola.
	\newblock {G}romov-{W}asserstein {A}lignment of {W}ord {E}mbedding {S}paces.
	\newblock In {\em Proceedings of the 2018 Conference on Empirical Methods in
		Natural Language Processing}, pages 1881--1890, Brussels, Belgium,
	October-November 2018. Association for Computational Linguistics.
	
	\bibitem[BBI22]{burago2022course}
	Dmitri Burago, Yuri Burago, and Sergei Ivanov.
	\newblock {\em A course in metric geometry}, volume~33.
	\newblock American Mathematical Society, 2022.
	
	\bibitem[Boo77]{bookstein1977study}
	Fred~L Bookstein.
	\newblock The study of shape transformation after {D'A}rcy {T}hompson.
	\newblock 1977.
	
	\bibitem[CM19]{chowdhury2019gromov}
	Samir Chowdhury and Facundo M{\'e}moli.
	\newblock The {G}romov--{W}asserstein distance between networks and stable
	network invariants.
	\newblock {\em Information and Inference: A Journal of the IMA}, 8(4):757--787,
	2019.
	
	\bibitem[CMPT{\etalchar{+}}22]{cuturi2022optimal}
	Marco Cuturi, Laetitia Meng-Papaxanthos, Yingtao Tian, Charlotte Bunne, Geoff
	Davis, and Olivier Teboul.
	\newblock Optimal transport tools (ott): A jax toolbox for all things
	{W}asserstein.
	\newblock {\em arXiv preprint arXiv:2201.12324}, 2022.
	
	\bibitem[CN20]{chowdhury2020gromov}
	Samir Chowdhury and Tom Needham.
	\newblock Gromov-{W}asserstein averaging in a {R}iemannian framework.
	\newblock In {\em Proceedings of the IEEE/CVF Conference on Computer Vision and
		Pattern Recognition Workshops}, pages 842--843, 2020.
	
	\bibitem[CSA15]{collyer2015method}
	Michael~L Collyer, David~J Sekora, and Dean~C Adams.
	\newblock A method for analysis of phenotypic change for phenotypes described
	by high-dimensional data.
	\newblock {\em Heredity}, 115(4):357--365, 2015.
	
	\bibitem[DDS22]{salmona2021gromov}
	Julie Delon, Agnes Desolneux, and Antoine Salmona.
	\newblock {G}romov-{W}asserstein distances between {G}aussian distributions.
	\newblock {\em Journal of Applied Probability}, 59(4):1178--1198, 2022.
	
	\bibitem[DLV24]{dumont2022existence}
	Th{\'e}o Dumont, Th{\'e}o Lacombe, and Fran{\c{c}}ois-Xavier Vialard.
	\newblock On the existence of {M}onge maps for the {G}romov--{W}asserstein
	problem.
	\newblock {\em Foundations of Computational Mathematics}, pages 1--48, 2024.
	
	\bibitem[DM16]{dryden2016statistical}
	Ian~L Dryden and Kanti~V Mardia.
	\newblock {\em Statistical shape analysis: with applications in {R}}, volume
	995.
	\newblock John Wiley \& Sons, 2016.
	
	\bibitem[DS81]{dubins1981equidiscontinuity}
	Lester Dubins and Gideon Schwarz.
	\newblock Equidiscontinuity of {B}orsuk-{U}lam functions.
	\newblock {\em Pacific Journal of Mathematics}, 95(1):51--59, 1981.
	
	\bibitem[DS12]{degroot2012probability}
	Morris~H DeGroot and Mark~J Schervish.
	\newblock {\em Probability and statistics}.
	\newblock Pearson Education, 2012.
	
	\bibitem[ELPZ97]{eldar1997farthest}
	Yuval Eldar, Michael Lindenbaum, Moshe Porat, and Yehoshua~Y Zeevi.
	\newblock The farthest point strategy for progressive image sampling.
	\newblock {\em IEEE transactions on image processing}, 6(9):1305--1315, 1997.
	
	\bibitem[FCG{\etalchar{+}}21]{flamary2021pot}
	R{\'e}mi Flamary, Nicolas Courty, Alexandre Gramfort, Mokhtar~Z. Alaya,
	Aur{\'e}lie Boisbunon, Stanislas Chambon, Laetitia Chapel, Adrien Corenflos,
	Kilian Fatras, Nemo Fournier, L{\'e}o Gautheron, Nathalie~T.H. Gayraud,
	Hicham Janati, Alain Rakotomamonjy, Ievgen Redko, Antoine Rolet, Antony
	Schutz, Vivien Seguy, Danica~J. Sutherland, Romain Tavenard, Alexander Tong,
	and Titouan Vayer.
	\newblock {POT}: Python optimal transport.
	\newblock {\em Journal of Machine Learning Research}, 22(78):1--8, 2021.
	
	\bibitem[FLWM15]{fan2015spatiotemporal}
	Chao Fan, Wenwen Li, Levi~J Wolf, and Soe~W Myint.
	\newblock A spatiotemporal compactness pattern analysis of congressional
	districts to assess partisan gerrymandering: a case study with {C}alifornia
	and {N}orth {C}arolina.
	\newblock {\em Annals of the Association of American Geographers},
	105(4):736--753, 2015.
	
	\bibitem[Fol99]{folland1999real}
	Gerald~B Folland.
	\newblock {\em Real analysis: modern techniques and their applications},
	volume~40.
	\newblock John Wiley \& Sons, 1999.
	
	\bibitem[GKPS99]{gromov1999metric}
	Mikhael Gromov, Misha Katz, Pierre Pansu, and Stephen Semmes.
	\newblock {\em Metric structures for {R}iemannian and non-{R}iemannian spaces},
	volume 152.
	\newblock Springer, 1999.
	
	\bibitem[Gon85]{gonzalez1985clustering}
	Teofilo~F Gonzalez.
	\newblock Clustering to minimize the maximum intercluster distance.
	\newblock {\em Theoretical computer science}, 38:293--306, 1985.
	
	\bibitem[Goo91]{goodall1991procrustes}
	Colin Goodall.
	\newblock Procrustes methods in the statistical analysis of shape.
	\newblock {\em Journal of the Royal Statistical Society: Series B
		(Methodological)}, 53(2):285--321, 1991.
	
	\bibitem[Hen16]{Hendrikson2016usingGD}
	Reigo Hendrikson.
	\newblock Using {G}romov-{W}asserstein distance to explore sets of networks.
	\newblock 2016.
	
	\bibitem[KKK21]{kaufman2021measure}
	Aaron~R Kaufman, Gary King, and Mayya Komisarchik.
	\newblock How to measure legislative district compactness if you only know it
	when you see it.
	\newblock {\em American Journal of Political Science}, 65(3):533--550, 2021.
	
	\bibitem[KM98]{klingenberg1998geometric}
	Christian~Peter Klingenberg and Grant~S McIntyre.
	\newblock Geometric morphometrics of developmental instability: analyzing
	patterns of fluctuating asymmetry with {P}rocrustes methods.
	\newblock {\em Evolution}, 52(5):1363--1375, 1998.
	
	\bibitem[LH05]{leordeanu2005spectral}
	Marius Leordeanu and Martial Hebert.
	\newblock A spectral technique for correspondence problems using pairwise
	constraints.
	\newblock In {\em Tenth IEEE International Conference on Computer Vision
		(ICCV'05) Volume 1}, volume~2, pages 1482--1489 Vol. 2, 2005.
	
	\bibitem[LMS23]{lim2023gromov}
	Sunhyuk Lim, Facundo M{\'e}moli, and Zane Smith.
	\newblock The gromov--hausdorff distance between spheres.
	\newblock {\em Geometry \& Topology}, 27(9):3733--3800, 2023.
	
	\bibitem[LPD13]{lipman2013conformal}
	Yaron Lipman, Jesus Puente, and Ingrid Daubechies.
	\newblock Conformal wasserstein distance: Ii. computational aspects and
	extensions.
	\newblock {\em Mathematics of Computation}, pages 331--381, 2013.
	
	\bibitem[MDFdRA02]{monteiro2002geometric}
	Leandro~R Monteiro, Jos{\'e} Alexandre~F Diniz-Filho, S{\'e}rgio~F dos Reis,
	and Edilson~D Ara{\'u}jo.
	\newblock Geometric estimates of heritability in biological shape.
	\newblock {\em Evolution}, 56(3):563--572, 2002.
	
	\bibitem[M{\'e}m07]{memoli2007use}
	Facundo M{\'e}moli.
	\newblock On the use of {G}romov-{H}ausdorff distances for shape comparison.
	\newblock {\em Proceedings of Point Based Graphics}, 2007.
	
	\bibitem[M{\'e}m11]{memoli2011gromov}
	Facundo M{\'e}moli.
	\newblock {G}romov--{W}asserstein distances and the metric approach to object
	matching.
	\newblock {\em Foundations of computational mathematics}, 11(4):417--487, 2011.
	
	\bibitem[Mil04]{miller2004computational}
	Michael~I Miller.
	\newblock Computational anatomy: shape, growth, and atrophy comparison via
	diffeomorphisms.
	\newblock {\em NeuroImage}, 23:S19--S33, 2004.
	
	\bibitem[MMWW23]{memoli2021ultrametric}
	Facundo M{\'e}moli, Axel Munk, Zhengchao Wan, and Christoph Weitkamp.
	\newblock The ultrametric {G}romov-{W}asserstein distance.
	\newblock {\em Discrete and Computational Geometry. To appear}, 2023.
	
	\bibitem[MN22]{memoli2022distance}
	Facundo M{\'e}moli and Tom Needham.
	\newblock Distance distributions and inverse problems for metric measure
	spaces.
	\newblock {\em Studies in Applied Mathematics}, 149(4):943--1001, 2022.
	
	\bibitem[MS05]{memoli2005theoretical}
	Facundo M{\'e}moli and Guillermo Sapiro.
	\newblock A theoretical and computational framework for isometry invariant
	recognition of point cloud data.
	\newblock {\em Foundations of Computational Mathematics}, 5:313--347, 2005.
	
	\bibitem[MSS18]{kritika}
	Facundo M{\'e}moli, Anastasios Sidiropoulos, and Kritika Singhal.
	\newblock Sketching and clustering metric measure spaces.
	\newblock {\em arXiv preprint arXiv:1801.00551}, 2018.
	
	\bibitem[Mul59]{muller1959note}
	Mervin~E Muller.
	\newblock A note on a method for generating points uniformly on n-dimensional
	spheres.
	\newblock {\em Communications of the ACM}, 2(4):19--20, 1959.
	
	\bibitem[MW22]{memoli2022metric}
	Facundo M{\'e}moli and Zhengchao Wan.
	\newblock On $p$-metric spaces and the $p$-{G}romov-{H}ausdorff {D}istance.
	\newblock {\em p-Adic Numbers, Ultrametric Analysis and Applications},
	14(3):173--223, 2022.
	
	\bibitem[Pac24]{dan-github}
	Daniel Packer.
	\newblock {G}romov-{W}asserstein distance between spheres.
	\newblock \url{https://github.com/Daniel-Packer/gw-between-spheres}, 2024.
	
	\bibitem[PC19]{peyre2019computational}
	Gabriel Peyr{\'e} and Marco Cuturi.
	\newblock Computational optimal transport: With applications to data science.
	\newblock {\em Foundations and Trends{\textregistered} in Machine Learning},
	11(5-6):355--607, 2019.
	
	\bibitem[PCS16]{peyre2016gromov}
	Gabriel Peyr{\'e}, Marco Cuturi, and Justin Solomon.
	\newblock Gromov-{W}asserstein averaging of kernel and distance matrices.
	\newblock In {\em International conference on machine learning}, pages
	2664--2672. PMLR, 2016.
	
	\bibitem[RBSB02]{robinson2002impact}
	DL~Robinson, PG~Blackwell, EC~Stillman, and AH~Brook.
	\newblock Impact of landmark reliability on the planar {P}rocrustes analysis of
	tooth shape.
	\newblock {\em Archives of oral biology}, 47(7):545--554, 2002.
	
	\bibitem[RGK23]{rioux2023entropic}
	Gabriel Rioux, Ziv Goldfeld, and Kengo Kato.
	\newblock Entropic gromov-wasserstein distances: Stability and algorithms.
	\newblock {\em arXiv preprint arXiv:2306.00182}, 2023.
	
	\bibitem[Roh90]{rohlf1990morphometrics}
	F~James Rohlf.
	\newblock Morphometrics.
	\newblock {\em Annual Review of Ecology and Systematics}, 21(1):299--316, 1990.
	
	\bibitem[RTG00]{Rubner2000earth}
	Y.~Rubner, C.~Tomasi, and L.~J. Guibas.
	\newblock The earth mover's distance as a metric for image retrieval.
	\newblock {\em International Journal of Computer Vision}, 40(2):99--121, 2000.
	
	\bibitem[SPC22]{scetbon2022linear}
	Meyer Scetbon, Gabriel Peyr{\'e}, and Marco Cuturi.
	\newblock Linear-time {G}romov--{W}asserstein distances using low rank
	couplings and costs.
	\newblock In {\em International Conference on Machine Learning}, pages
	19347--19365. PMLR, 2022.
	
	\bibitem[Stu12]{sturm-ss}
	Karl-Theodor Sturm.
	\newblock The space of spaces: curvature bounds and gradient flows on the space
	of metric measure spaces.
	\newblock {\em arXiv preprint arXiv:1208.0434}, 2012.
	
	\bibitem[SXK{\etalchar{+}}07]{scher2007hippocampal}
	Ann~I Scher, Yuan Xu, ESC Korf, Lon~R White, Philip Scheltens, Arthur~W Toga,
	Paul~M Thompson, SW~Hartley, MP~Witter, Daniel~J Valentino, et~al.
	\newblock Hippocampal shape analysis in {A}lzheimer’s disease: a
	population-based study.
	\newblock {\em Neuroimage}, 36(1):8--18, 2007.
	
	\bibitem[Vay20]{vayer2020contribution}
	Titouan Vayer.
	\newblock {\em A contribution to Optimal Transport on incomparable spaces}.
	\newblock PhD thesis, Lorient, 2020.
	
	\bibitem[VCF{\etalchar{+}}18]{vayer2018optimal}
	Titouan Vayer, Laetitia Chapel, R{\'e}mi Flamary, Romain Tavenard, and Nicolas
	Courty.
	\newblock Optimal transport for structured data with application on graphs.
	\newblock {\em arXiv preprint arXiv:1805.09114}, 2018.
	
	\bibitem[VCF{\etalchar{+}}20]{vayer2020fused}
	Titouan Vayer, Laetitia Chapel, R{\'e}mi Flamary, Romain Tavenard, and Nicolas
	Courty.
	\newblock Fused {G}romov-{W}asserstein distance for structured objects.
	\newblock {\em Algorithms}, 13(9):212, 2020.
	
	\bibitem[VCTF19]{titouan2019optimal}
	Titouan Vayer, Nicolas Courty, Romain Tavenard, and R{\'e}mi Flamary.
	\newblock Optimal transport for structured data with application on graphs.
	\newblock In {\em International Conference on Machine Learning}, pages
	6275--6284. PMLR, 2019.
	
	\bibitem[VFC{\etalchar{+}}19]{titouan2019sliced}
	Titouan Vayer, R{\'e}mi Flamary, Nicolas Courty, Romain Tavenard, and Laetitia
	Chapel.
	\newblock Sliced {G}romov--{W}asserstein.
	\newblock {\em Advances in Neural Information Processing Systems}, 32, 2019.
	
	\bibitem[Vil09]{villani2009optimal}
	C{\'e}dric Villani.
	\newblock {\em Optimal transport: old and new}, volume 338.
	\newblock Springer, 2009.
	
	\bibitem[Vil21]{villani2021topics}
	C{\'e}dric Villani.
	\newblock {\em Topics in optimal transportation}, volume~58.
	\newblock American Mathematical Soc., 2021.
	
	\bibitem[WBR{\etalchar{+}}07]{wang2007large}
	Lei Wang, Faisal Beg, Tilak Ratnanather, Can Ceritoglu, Laurent Younes, John~C
	Morris, John~G Csernansky, and Michael~I Miller.
	\newblock Large deformation diffeomorphism and momentum based hippocampal shape
	discrimination in dementia of the {A}lzheimer type.
	\newblock {\em IEEE transactions on medical imaging}, 26(4):462--470, 2007.
	
\end{thebibliography}
\end{document}